\documentclass{amsart}
\usepackage{amsmath,amsfonts,amsthm,amssymb,framed}
\usepackage[pdftex,hyperfootnotes=false]{hyperref}

\usepackage{tikz}

\usetikzlibrary{arrows,decorations.pathmorphing,backgrounds,positioning,fit,trees}
\usepackage{setspace}
\newtheorem{thm}{Theorem}[section]
\newtheorem*{thm*}{Theorem}
\newtheorem*{lemma*}{Lemma}
\newtheorem*{coro*}{Corollary}
\newtheorem{lemma}[thm]{Lemma}

\newtheorem{coro}[thm]{Corollary}
\newtheorem{prop}[thm]{Proposition}
\theoremstyle{definition}
\newtheorem{defn}[thm]{Definition}
\theoremstyle{remark}
\newtheorem{rem}[thm]{Remark}
\newtheorem{ex}[thm]{Example}

\numberwithin{equation}{section}
\def\k{\mathsf{k}}
\def\A{\mathcal{A}}
\def\U{\mathcal{U}}

\def\S{\Sigma}

\def\R{\mathbb{R}}

\def\Z{\mathbb{Z}}
\def\Q{\mathsf{Q}}
\def\M{\mathcal{M}}

\def\N{\mathbb{N}}

\newcommand{\curve}[1][x]{\mathsf{#1}}
\newcommand{\multi}[1][X]{\mathsf{#1}}
\newcommand{\link}[1][X]{\mathbf{#1}}

\def\rq{q^{\frac{1}{2}}}
\def\Zq{\mathbb{Z}_{q}}
\def\ex{\mathbf{ex}}
\def\B{\mathbf{B}}
\def\E{\mathsf{E}}

\def\T{\mathbb{T}}

\def\Links{\mathsf{Links}}
\def\Sk{\mathsf{Sk}}
\def\Multi{{\mathsf{SMulti}}}

\title{Skein and cluster algebras of marked surfaces}

\author{Greg Muller}
\address{Department of Mathematics,
University of Michigan, Ann Arbor, MI 48108, USA}
\email{morilac@umich.edu}
\thanks{$2010$ \emph{Mathematics Subject Classification.} Primary 13F60, Secondary 57M50, 16S99}
\thanks{This work was supported by the VIGRE program at LSU, National Science Foundation grant DMS-0739382.}

\tikzstyle{mutable}=[inner sep=0.5mm,circle,draw,minimum size=2mm]
\tikzstyle{frozen}=[inner sep=0.5mm,rectangle,draw]
\tikzstyle{marked}=[inner sep=0.5mm,circle,draw,fill=black!50]
\tikzstyle{outline}=[thick,line width=1.5mm,draw=black!10]

\begin{document}

\begin{abstract}
This paper considers several algebras associated to an oriented surface $\S$ with a finite set of marked points on its boundary.  The first is  the \emph{skein algebra} $\Sk_q(\S)$, which is spanned by links in the surface which are allowed to have endpoints at the marked points, modulo several locally defined relations.  The product is given by superposition of links.  A basis of this algebra is given, as well as several algebraic results.

When $\S$ is triangulable, a \emph{quantum cluster algebra} $\A_q(\S)$ and \emph{quantum upper cluster algebra} $\U_q(\S)$ can be defined.  These are algebras coming from the triangulations of $\S$ and the elementary moves between them.  Cluster algebras have been a subject of significant recent interest, due in part to their extraordinary positivity and Laurent properties.

%There is also a related algebra $\U_q(\S)$, the (quantum) upper cluster algebra.

Natural inclusions $\A_q(\S)\subseteq \Sk_q^o(\S)\subseteq \U_q(\S)$ are shown, where $\Sk_q^o(\S)$ is a certain Ore localization of $\Sk_q(\S)$.  When $\S$ has at least two marked points in each component, these inclusions are strengthened to equality, exhibiting a quantum cluster structure on $\Sk_q^o(\S)$.

The method for proving these equalities has the potential to show $\A_q=\U_q$ for other classes of cluster algebras.  As a demonstration of this fact, a new proof is given that $\A_q=\U_q$ for acyclic cluster algebras.% (reproving a result of Berenstein and Zelevinsky).
\end{abstract}

\maketitle

\section{Introduction}

In this paper, we consider \textbf{marked surfaces}: compact, oriented surfaces, possibly with boundary, together with a finite set of marked points in the boundary.\footnote{This contrasts with some references, where `marked surfaces' may have interior marked points.}

%  Given a marked surface $\S$, there are several algebras we may associate to $\S$, each of which is a generalization or modification of an algebra which already exists in the literature.

\subsection{The skein module}

Motivated by computing the Jones polynomial of a knot, Kauffman \cite{Kau87} introduced the \emph{Kauffman bracket}, a (framed)\footnote{We suppress the details of framing a knot; all drawn knots will be given the blackboard framing.} knot invariant defined by the two local relations in Figure \ref{fig: skeinintro1}.
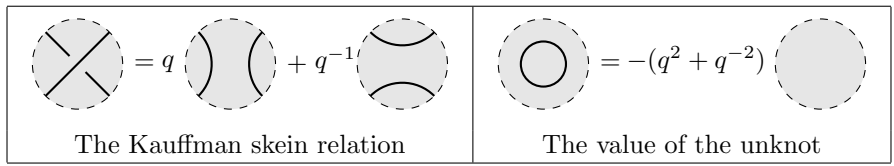
\begin{figure}[h!]
\centering
\begin{tabular}{|c|c|}
\hline
\begin{tikzpicture}
\path[use as bounding box] (-1.15in,-.3in) rectangle (1.15in,.3in);
\begin{scope}[xshift=-.85in,scale=.15]
    \draw[fill=black!10,dashed] (0,0) circle (4);
    \draw[thick] (-2.83,-2.83) to (2.83,2.83);
    \draw[thick] (-2.83,2.83) to (-.71,.71);
    \draw[thick] (.71,-.71) to (2.83,-2.83);
\end{scope}
\node (=) at (-.5in,0) {$=$};
\node (q) at (-.375in,0) {$q$};
\begin{scope}[xshift=-.05in,scale=.15]
    \draw[fill=black!10,dashed] (0,0) circle (4);
    \draw[thick] (-2.83,-2.83) to [out=45,in=-45] (-2.83,2.83);
    \draw[thick] (2.83,-2.83) to [out=135,in=-135] (2.83,2.83);
\end{scope}
\node (+) at (.3in,0) {$+$};
\node (q') at (.5in,.02in) {$q^{-1}$};
\begin{scope}[xshift=.85in,scale=.15]
    \draw[fill=black!10,dashed] (0,0) circle (4);
    \draw[thick] (-2.83,-2.83) to [out=45,in=135] (2.83,-2.83);
    \draw[thick] (-2.83,2.83) to [out=-45,in=-135] (2.83,2.83);
\end{scope}
\end{tikzpicture}
&
\begin{tikzpicture}
\path[use as bounding box] (-1.15in,-.3in) rectangle (.9in,.3in);
\begin{scope}[xshift=-.85in,scale=.15]
    \draw[fill=black!10,dashed] (0,0) circle (4);
    \draw[thick] (0,0) circle (2);
\end{scope}
\node (=) at (-.5in,0) {$=$};
\node (2) at (-.05in,.02in) {$-(q^2+q^{-2})$};
\begin{scope}[xshift=.6in,scale=.15]
    \draw[fill=black!10,dashed] (0,0) circle (4);
\end{scope}
\end{tikzpicture}\\
The Kauffman skein relation
&
The value of the unknot
\\
\hline
\end{tabular}
\caption{The skein relations (without marked points).}
\label{fig: skeinintro1}
\end{figure}

These relations are defined as manipulations of a knot (or link) in an oriented 3-manifold, where the dashed circle represents a small sphere, and the links are understood to be kept identical outside this sphere.  Using these relations, any link in $\mathbb{R}^3$ can be reduced to a Laurent polynomial in $q$ (the Kauffman bracket of the link) times the empty link.

For a general oriented $3$-manifold, these relations can be encoded in the `skein module', introduced independently by Turaev \cite{Tur88},\cite{Tur89} and Przytycki \cite{Prz91}.  Let $\Zq:=\Z[q^{\pm\frac{1}{2}}]$ denote the Laurent ring in the indeterminant\footnote{The justification for including the half-power of $q$ will come later.} $\rq$, and let $\Zq^\Links$ be the module of $\Zq$-linear combinations of ambient isotopy classes of framed links.  Imposing the skein relations defines a quotient $\Zq$-module of $\Zq^\Links$, called the \emph{skein module} of the 3-manifold.  The skein module of $\mathbb{R}^3$ is the free $\Zq$-module spanned by the empty link.%; any link is sent to its Kauffman bracket times the empty link.

\subsection{The skein algebra $Sk_q(\S)$ (without marked points)}

When the 3-manifold in question is $\S\times [0,1]$ for an unmarked surface $\S$, two extra structures appear.  First, two links in $\S\times [0,1]$ can be `stacked' vertically to give a new link in $\S\times [0,1]$ which contains the first link in $\S\times [0,\frac{1}{2}]$ and the second in $\S\times [\frac{1}{2},1]$.  This gives a well-defined \emph{superposition product} on the skein module of $\S\times[0,1]$ and makes it into an associative $\Zq$-algebra called the \emph{skein algebra} of $\S$.

Second, any link in $\S\times[0,1]$ can be projected into $\S$, with overcrossings and undercrossings used to keep track of the original link.\footnote{Ambient homotopy may be required to ensure the projection has simple transverse intersections.}  As an abuse of terminology, such a diagram will be called a \emph{link} in $\S$.  The skein algebra of $\S$ can be computed directly from the set of links in $\S$, as the quotient of $\Zq^\Links$ by a submodule generated by the skein relations.  In this way, the skein algebra can be associated directly to the surface $\S$. %Here, the skein relations are interpreted as manipulations of links in $\S$, and dashed circles represent a small disc.

\subsection{The skein algebra $Sk_q(\S)$ (with marked points)}

Motivated by examples coming from the theory of cluster algebras and Teichm\"uller theory, we define a generalization of skein algebras to marked surfaces.

Let $\S$ be a surface with a finite set of `marked points' $\M$.  A `\textbf{link}' in $\S$ will be a collection of immersed curves in $\S$, with transverse intersections and boundary contained in $\M$, together with a `crossing data'. This is a choice, for each intersection, of the order in which the curves pass over each other (see Section \ref{section: link}).  Links are considered up to homotopy through the set of links.
\begin{rem}
Actually, we will extend this definition of link to allow simultaneous crossings at marked endpoints.  Two curves can then arrive transversely at a marked point in three ways: over, under and simultaneous.  This generalization does not affect the subsequent skein algebra (see Remark \ref{rem: nosimul}).
%
%Then, two curves arriving transversely at a marked point can go over, under, or through each other.
\end{rem}

Let $\Zq^{\Links}$ denote the free $\Zq$-module spanned by (homotopy classes of) links in $\S$, and define a quotient $\Zq$-module $\Sk_q(\S)$ by imposing the relations in Figure \ref{fig: skeinintro}.  
%A dashed circle denotes a small disc in $\S$, and the links in each term of the equality are understood to be identical outside the circle.  We also allow additional undrawn curves at the marked points which have the same order with respect to the drawn curves.
The element in $\Sk_q(\S)$ corresponding to a link $\link$ will be denoted $[\link]$.

\begin{figure}[h]
\centering
\begin{tabular}{|c|c|}
\hline
\begin{tikzpicture}
\path[use as bounding box] (-1.15in,-.3in) rectangle (1.15in,.3in);
\begin{scope}[xshift=-.85in,scale=.15]
    \draw[fill=black!10,dashed] (0,0) circle (4);
    \draw[thick] (-2.83,-2.83) to (2.83,2.83);
    \draw[thick] (-2.83,2.83) to (-.71,.71);
    \draw[thick] (.71,-.71) to (2.83,-2.83);
\end{scope}
\node (=) at (-.5in,0) {$=$};
\node (q) at (-.375in,0) {$q$};
\begin{scope}[xshift=-.05in,scale=.15]
    \draw[fill=black!10,dashed] (0,0) circle (4);
    \draw[thick] (-2.83,-2.83) to [out=45,in=-45] (-2.83,2.83);
    \draw[thick] (2.83,-2.83) to [out=135,in=-135] (2.83,2.83);
\end{scope}
\node (+) at (.3in,0) {$+$};
\node (q') at (.5in,.02in) {$q^{-1}$};
\begin{scope}[xshift=.85in,scale=.15]
    \draw[fill=black!10,dashed] (0,0) circle (4);
    \draw[thick] (-2.83,-2.83) to [out=45,in=135] (2.83,-2.83);
    \draw[thick] (-2.83,2.83) to [out=-45,in=-135] (2.83,2.83);
\end{scope}
\end{tikzpicture}
&
\begin{tikzpicture}
\path[use as bounding box] (-.8in,-.3in) rectangle (1.6in,.3in);
\node (q) at (-.625in,0.04in) {$q^{-\frac{1}{2}}$};
\begin{scope}[xshift=-.25in,scale=.15]
	\begin{scope}
	\clip (0,0) circle (4);
	\draw[fill=black!10,thick] (-5,-3) to [in=180,out=30] (0,-2) to [in=150,out=0] (5,-3) to [line to] (5,5) to (0,5) to (-5,5);
	\node (1) at (0,-2) [marked] {};
	\draw[thick] (1) to (4,3);
	\draw[thick] (-.8,-1) to (-4,3);
	\end{scope}
	\draw[dashed] (0,0) circle (4);
\end{scope}
\node (=) at (.1in,0) {$=$};
\begin{scope}[xshift=.45in,scale=.15]
	\begin{scope}
	\clip (0,0) circle (4);
	\draw[fill=black!10,thick] (-5,-3) to [in=180,out=30] (0,-2) to [in=150,out=0] (5,-3) to [line to] (5,5) to (0,5) to (-5,5);
	\node (1) at (0,-2) [marked] {};
	\draw[thick] (1) to (4,3);
	\draw[thick] (1) to (-4,3);
	\end{scope}
	\draw[dashed] (0,0) circle (4);
\end{scope}
\node (=) at (.8in,0) {$=$};
\node (q) at (.975in,0.04in) {$q^{\frac{1}{2}}$};
\begin{scope}[xshift=1.3in,scale=.15]
	\begin{scope}
	\clip (0,0) circle (4);
	\draw[fill=black!10,thick] (-5,-3) to [in=180,out=30] (0,-2) to [in=150,out=0] (5,-3) to [line to] (5,5) to (0,5) to (-5,5);
	\node (1) at (0,-2) [marked] {};
	\draw[thick] (1) to (-4,3);
	\draw[thick] (.8,-1) to (4,3);
	\end{scope}
	\draw[dashed] (0,0) circle (4);
\end{scope}
\end{tikzpicture}
\\
The Kauffman skein relation
&
The boundary skein relation
\\

\hline

\begin{tikzpicture}
\path[use as bounding box] (-1.15in,-.3in) rectangle (.9in,.3in);
\begin{scope}[xshift=-.85in,scale=.15]
    \draw[fill=black!10,dashed] (0,0) circle (4);
    \draw[thick] (0,0) circle (2);
\end{scope}
\node (=) at (-.5in,0) {$=$};
\node (2) at (-.05in,.02in) {$-(q^2+q^{-2})$};
\begin{scope}[xshift=.6in,scale=.15]
    \draw[fill=black!10,dashed] (0,0) circle (4);
\end{scope}
\end{tikzpicture}
&
\begin{tikzpicture}
\path[use as bounding box] (-1.4in,-.3in) rectangle (.875in,.3in);
\begin{scope}[xshift=-1.1in,scale=.15]
	\begin{scope}
	\clip (0,0) circle (4);
	\draw[fill=black!10,thick] (-5,-3) to [in=180,out=30] (0,-2) to [in=150,out=0] (5,-3) to [line to] (5,5) to (0,5) to (-5,5);
	\node (1) at (0,-2) [marked] {};
	\draw[thick] (1) to [out=45,in=0] (0,2) to [out=180,in=135] (-.8,-1);
	\end{scope}
	\draw[dashed] (0,0) circle (4);
\end{scope}
\node (='') at (-.75in,0) {$=$};
\begin{scope}[xshift=-.4in,scale=.15]
	\begin{scope}
	\clip (0,0) circle (4);
	\draw[fill=black!10,thick] (-5,-3) to [in=180,out=30] (0,-2) to [in=150,out=0] (5,-3) to [line to] (5,5) to (0,5) to (-5,5);
	\node (1) at (0,-2) [marked] {};
	\draw[thick] (1) to [out=45,in=0] (0,2) to [out=180,in=135] (1);
	\end{scope}
	\draw[dashed] (0,0) circle (4);
\end{scope}
\node (=) at (-.05in,0) {$=$};
\begin{scope}[xshift=.3in,scale=.15]
	\begin{scope}
	\clip (0,0) circle (4);
	\draw[fill=black!10,thick] (-5,-3) to [in=180,out=30] (0,-2) to [in=150,out=0] (5,-3) to [line to] (5,5) to (0,5) to (-5,5);
	\node (1) at (0,-2) [marked] {};
	\draw[thick] (1) to [out=135,in=180] (0,2) to [out=0,in=45] (.8,-1);
	\end{scope}
	\draw[dashed] (0,0) circle (4);
\end{scope}
\node (=') at (.65in,0) {$=$};
\node (0) at (.775in,.01in) {$0$};
\end{tikzpicture}
\\
The value of the unknot
&
The value of a contractible arc
\\
\hline
\end{tabular}
\caption{The skein relations with marked points.
%\small A dashed circle denotes a small disc in $\S$, and the links in each term of the equality are understood to be identical outside the circle.  We also allow additional undrawn curves at the marked points which have the same order with respect to the drawn curves.
}
\label{fig: skeinintro}
\end{figure}

Notes on the figure.
\begin{itemize}
    \item A dashed circle denotes a small disc in $\S$, and the links in each term of the equality are understood to be identical outside the circle.
    \item A solid curve between grey and white denotes the boundary of $\S$, and a dark dot denotes a marked point.
    \item We also allow additional undrawn curves at the marked points which have the same order with respect to the drawn curves.
\end{itemize}

These relations imply several other relations (Proposition \ref{prop: Reide}): the (framed) Reidemeister moves from knot theory, as well as an additional marked variation of the second Reidemeister move  (Figure \ref{fig: Reide}).

Given two transverse links $\link[X]$ and $\link[Y]$, their `\textbf{superposition}' $\link[X]\cdot\link[Y]$ is the union of the two links, with every curve in $\link[X]$ passing over every curve in $\link[Y]$.  This extends to a well-defined product on $\Sk_q(\S)$ and makes $\Sk_q(\S)$ into an associative $\Zq$-algebra (Proposition \ref{prop: skeinalgebra}), which we call the `\textbf{skein algebra}' of $\S$.

Many properties of $\Sk_q(\S)$ are shown, which generalize known unmarked results.
\begin{itemize}
%	\item (Section\ref{section: bar}) $\Sk_q(\S)$ has an anti-involution $\dagger$.
%	\item (Section \ref{section: grading}) $\Sk_q(\S)$ has a grading indexed by the marked points of $\S$.
	\item (Corollary \ref{coro: domain}) $\Sk_q(\S)$ is a domain.
	\item (Theorem \ref{thm: fingen}) $\Sk_q(\S)$ is finitely generated.
	\item (Lemma \ref{lemma: basis}) $\Sk_q(\S)$ has a $\Zq$-basis parametrized by `\emph{simple multicurves}'.
\end{itemize}

\subsection{Triangulations}

Curves in $\S$ come in two types,
\begin{itemize}
\item `\emph{loops}': immersed images of $S^1$, and
\item `\emph{arcs}': immersed images of $[0,1]$, with endpoints mapping to marked points.
\end{itemize}
A \textbf{triangulation} $\Delta$ of $\S$ is a simple multicurve consisting of arcs, such that the complement of the arcs is a disjoint union of discs with three marked points.  As elements in $\Sk_q(\S)$, the arcs in $\Delta$ \emph{quasi-commute}; that is, for $\curve_i,\curve_j\in \Delta$, there is a $\Lambda^\Delta_{i,j}\in\Z$ such that
\[ [\curve_i][\curve_j] = q^{\Lambda^\Delta_{i,j}}[\curve_i][\curve_j]\]
The numbers $\Lambda^\Delta_{i,j}$ correspond to entries in a `\emph{orientation matrix}' (Section \ref{section: skewmatrix}).

Triangulations of $\S$ give embeddings of the skein algebra into well-behaved algebras.
Let the \emph{quantum torus} $\T_\Delta$ associated to $\Delta$ be the $\Zq$-algebra with a $\Zq$-basis of elements of the form $M^\alpha$, $\forall \alpha\in \Z^\Delta$, and multiplication defined by\footnote{Here, $\langle - ,-\rangle$ is the natural dot product on $\Z^\Delta$.}
\[ M^\alpha M^\beta = q^{\frac{1}{2}\langle\alpha,\Lambda^\Delta \beta\rangle} M^{\alpha+\beta} = q^{\langle\alpha,\Lambda^\Delta \beta\rangle} M^\beta M^{\alpha}\]
\begin{thm*} [\textbf{\ref{thm: Laurent}}]
For each triangulation $\Delta$ of $\S$, there is an injective Ore localization
\[ \Sk_q(\S)\hookrightarrow \Sk_q(\S)[\Delta^{-1}]\simeq\T_\Delta\]
which sends $[\curve_i]$ to $M^{e_i}$.
\end{thm*}
The theorem says that $\Sk_q(\S)$ embeds into its skew-field of fractions $\mathcal{F}$, and inside that skew-field, every element of $\Sk_q(\S)$ can be written as a skew-Laurent polynomial in the arcs in $\Delta$.

\subsection{Three algebras}

When $\S$ is triangulable, Theorem \ref{thm: Laurent} leads to the definition of three related $\Zq$-algebras.

%Theorem \ref{thm: Laurent} leads to the definition of three $\Zq$-algebras.

\begin{itemize}
\item \textsc{The localized skein algebra $\Sk_q^o(\S)$ (Section \ref{section: localskein}).}

A `\textbf{boundary arc}' is a simple arc in $\S$ which is homotopic to an arc contained in the boundary.  A triangulation $\Delta$ of $\S$ contains the set of boundary arcs, and so the localization $\Sk_q(\S)[\Delta^{-1}]$ contains the inverse to each boundary arc.  The `\textbf{localized skein algebra}' $\Sk_q^o(\S)$ is the Ore localization of $\Sk_q(\S)$ at the boundary arcs in $\S$.

\item \textsc{The (quantum) cluster algebra $\A_q(\S)$ (Section \ref{section: cluster2}).}

The skein algebra $\Sk_q(\S)$ is generated by simple curves (Corollary \ref{coro: gens}), and so $\Sk_q^o(\S)$ is generated by simple curves and the inverses to boundary curves.  The `\textbf{(quantum) cluster algebra}' $\A_q(\S)$ of $\S$ is the $\Zq$-subalgebra of $\Sk_q^o(\S)$ generated by simple arcs and the inverses to boundary arcs.

\item \textsc{The (quantum) upper cluster algebra $\U_q(\S)$ (Section \ref{section: cluster2}).}

Since $\S$ may have many triangulations, Theorem \ref{thm: Laurent} provides many distinct skew-Laurent expressions for an element in $\Sk_q(\S)$.  This property may be turned into a criterion for defining another algebra.  The `\textbf{(quantum) upper cluster algebra}' $\U_q(\S)$ of $\S$ is the $\Zq$-algebra consisting of elements in the skew-field $\mathcal{F}$ which can be written as a skew-Laurent polynomial in each triangulation.
\end{itemize}
%\begin{rem}
%While these exact algebras do not appear to have been considered elsewhere, they are each generalizations or modifications of algebras which are well-studied.
%\end{rem}

These algebras satisfy the following containments.
\begin{thm*} [\textbf{\ref{thm: main1}}]
For any triangulable marked surface $\S$,
\[ \A_q(\S)\subseteq \Sk_q^o(\S)\subseteq \U_q(\S)\]
\end{thm*}
Our main result is that these are equalities for most marked surfaces.
\begin{thm*} [\textbf{\ref{thm: main2}}]
For a triangulable marked surface $\S$ with at least two marked points in each connected component,
\[ \A_q(\S)= \Sk_q^o(\S)= \U_q(\S)\]
\end{thm*}
%For other triangulable marked surfaces, $\A_q(\S)\neq\Sk_q(\S)$ (Theorem \ref{thm: missing}).

%The missing case, when $\S$ has a component with a single boundary component and a single marked point, is addressed in Section \ref{section: missing}.

\begin{rem}
The definitions given above for $\A_q(\S)$ and $\U_q(\S)$ make Theorem \ref{thm: main1} immediate, but it is not clear they are quantum cluster algebras in the sense of \cite{BZ05}.  The body of the paper takes the opposite approach.  In Section \ref{section: cluster2}, $\A_q(\S)$ and $\U_q(\S)$ are defined as quantum cluster algebras, and equivalence to the previous definitions will be a consequence of Theorem \ref{thm: main1}.
\end{rem}

\subsection{Quantum cluster algebras}

Any two triangulations of $\S$ can be related by a sequence of `\emph{flips}', where a single arc is replaced by a distinct arc. % (Section \ref{section: ???}).
The flip of an arc in $\Delta$ has a simple expression as a skew-Laurent polynomial in $\Delta$, and by iterating these expressions, any arc in any triangulation can be obtained.

This process is a specific case of a more general framework: the theory of quantum cluster algebras (introduced in \cite{FZ02}, quantized in \cite{BZ05}).  We sketch this theory now, precise definitions are in Section \ref{section: cluster1}.
One starts with a `\textbf{quantum seed}':
\begin{itemize}
\item a finite set of quasi-commuting `\textbf{cluster variables}' in a skew-field, which are designated either `\emph{exchangeable}' or `\emph{frozen}', and
\item a rule (called `\textbf{mutation}') for replacing any exchangeable cluster variable by a new exchangeable cluster variable, resulting in a new quantum seed.  %simple skew-Laurent polynomial in the other cluster variables.
\end{itemize}
The `\textbf{quantum cluster algebra}' $\A_q$ associated to a quantum seed is the $\Zq$-algebra generated by all the cluster variables obtained by iterated mutations, and the inverses to the frozen cluster variables.  A quantum seed also determines a `\textbf{quantum upper cluster algebra}' $\U_q$, which is an algebra containing $\A_q$ defined as an intersection of quantum tori.\footnote{In this paper, `cluster algebras' are quantum cluster algebras unless otherwise specified.} %\footnote{Precise definitions are found in Section \ref{section: cluster1}.}

In case of marked surfaces,  a triangulation $\Delta$ of $\S$ determines a quantum seed,
\begin{itemize}
\item The cluster variables are the arcs in $\Delta$, as elements in $\mathcal{F}$, the skew-field of $\Sk_q(\S)$. An arc is frozen if it is a boundary arc and exchangeable otherwise.
\item The mutation rule is determined from the relative orientations of the arcs in $\Delta$ at the endpoints.\footnote{Specifically, the exchange matrix is a restriction of the `\emph{skew-adjacency matrix}' in Section \ref{section: skewmatrix}.}
\end{itemize}
The resulting algebras $\A_q(\S)$ and $\U_q(\S)$ do not depend on the choice of triangulation (Definition \ref{defn: CAmarked}), and coincide with the definitions of $\A_q(\S)$ and $\U_q(\S)$ in the previous section (Theorem \ref{thm: main1} and Remark \ref{rem: main1rem}).

The specialization $\rq=1$ of $\A_q(\S)$ becomes a commutative cluster algebra $\A_1(\S)$.  Commutative cluster algebras associated to marked surfaces have been introduced (in \cite{GSV05}, \cite{FG06}) and extensively studied (in \cite{FST08},\cite{FT12} and \cite{Sch08},\cite{ST09},\cite{Sch10},\cite{MSW11}).  The relation of $\A_1(\S)$ to skein algebras was noticed in \cite[Section 12.3]{FG06}.  The equality $\A_1(\S)=\Sk_1^o(\S)$ for triangulable surfaces with at least two marked points (the commutative specialization of Theorem \ref{thm: main2}) has been independently proven by Musiker, Schiffler and Williams (in \cite{MW13},\cite{MSW12}) using more explicit methods than this paper.  %While the results of this paper are presented on the level of quantum cluster algebras,

\begin{rem}
The commutative cluster algebra of a marked surface (as defined in \cite{FST08}) depends on a choice of coefficents.  The commutative specialization $\A_1(\S)$ has coefficients in the Laurent ring generated by the boundary arcs.
\end{rem}

%\subsection{Consequences of Theorem \ref{thm: main2}}

\subsection{The structure of the paper}

The first part of the paper focuses on skein algebras of general marked surfaces.

\begin{enumerate}
\setcounter{enumi}{1}
\item \textbf{Curves and links in marked surfaces.}  This section gives our definitions of `curve', `multicurve' and `link' for marked surfaces.
\item \textbf{The skein algebra $\Sk_q(\S)$.} The skein algebra is defined, first as a $\Z_q$-module, and then as a $\Z_q$-algebra under the superposition product.  An anti-involution and a grading of $\Sk_q(\S)$ are given.
\item \textbf{Simple multicurves.}  Lemma \ref{lemma: basis} proves that the simple multicurves define a $\Z_q$-basis of $\Sk_q(\S)$.  This is used to prove that simple curves are not zero-divisors (Lemma \ref{lemma: nonzero}), and multiplication by a simple arc $\curve$ reduces the `crossing number' with $\curve$ (Lemma \ref{lemma: reducing}).
%\item \textbf{Multiplication by simple curves.}  This section proves two important lemmas regarding the algebraic behavior of simple curves in $\Sk_q(\S)$.  First, simple curves are not zero-divisors (Lemma \ref{lemma: nonzero}).  Second, multiplication by a simple arc $\curve$ reduces crossing number with that arc (Lemma \ref{lemma: reducing}).
\item \textbf{The localized skein algebra $\Sk_q^o(\S)$.} The localized skein algebra is defined, shown to be an Ore localization, and a $\Z_q$-basis by certain weighted simple multicurves is given.
\end{enumerate}

\noindent The second part of the paper focuses on the case when $\S$ is triangulable, and the connection to cluster algebras.

\begin{enumerate}
\setcounter{enumi}{5}
\item \textbf{Triangulations.} Triangulations and some of their basic properties are reviewed.  A method is given for expressing an element of $\Sk_q(\S)$ as a skew-Laurent polynomial in a given triangulation (Corollary \ref{coro: Laurent}).  This is used to prove that the localization of $\Sk_q(\S)$ at $\Delta$ is a quantum torus.
\item \textbf{Quantum cluster algebras of marked surfaces.} Section \ref{section: cluster1} reviews the generalities of quantum cluster algebras.  Section \ref{section: cluster2} defines the quantum seed associated to a triangulation of a marked surface (Proposition \ref{prop: quantumseed}) and checks that the corresponding cluster algebras $\A_q(\S)$ and $\U_q(\S)$ only depend on $\S$ (Corollary \ref{coro: independent}).  These are related to the localized skein algebra by $\A_q(\S)\subseteq \Sk_q^o(\S)\subseteq \U_q(\S)$ (Theorem \ref{thm: main1}).
\item \textbf{A general technique for $\A_q=\U_q$.}  This section develops an approach for showing $\A_q=\U_q$ for large classes of quantum cluster algebras.  The final criterion is given in Lemma \ref{lemma: A=Uclass}.  This criterion is used to provide a new proof that $\A_q=\U_q$ for `acyclic' cluster algebras (Proposition \ref{prop: acyclic}).
\item \textbf{$\A_q(\S)=\Sk_q^o(\S)=\U_q(\S)$ for (most) marked surfaces.} Theorem \ref{thm: main2} is proven using the techniques of the preceding section.
\end{enumerate}

\noindent The last part of the paper explores some cases and consequences of Theorem \ref{thm: main2}.

\begin{enumerate}
\setcounter{enumi}{9}
%\item \textbf{General consequences of Theorem \ref{thm: main2}.}  %This section discusses some applications of the equalities $\A_q(\S)=\Sk_q^o(\S)=\U_q(\S)$.
%These includes loop elements in $\A_q(\S)$, a $\Zq$-basis for $\A_q(\S)$ and $\U_q(\S)$, and certain nice algebraic properties.
\item \textbf{Loop elements.}  The simple loops in $\Sk_q^o(\S)$ define extra elements of $\A_q(\S)$ which are not cluster variables.  Considering these elements simplifies computations and provides a free $\Zq$-basis of $\A_q(\S)$.
\item \textbf{The commutative specialization $\rq=1$.} This section discusses the commutative specialization $\A_1(\S)=\Sk_1^o(\S)=\U_1(\S)$.
The commutative cluster algebra $\A_1(\S)$ is `\emph{locally acyclic}', which implies additional results.
\item \textbf{Examples and non-examples.} This section explores specific cases of $\S$, such as discs and an annulus.  %The case of triangulable marked surfaces with a single marked point is investigated, and it is shown that Theorem \ref{thm: main2} cannot be extended to this case (Theorem \ref{thm: missing}).
\end{enumerate}
The paper concludes with an appendix showing that $\Sk_q(\S)$ is finitely generated, by directly generalizing the original proof of Bullock in the unmarked case \cite{Bul99}.

\subsection{Acknowledgements}

The author would like to thank a great many people for useful discussions and support; including A. Berenstein, M. Gekhtman, A. Knutson, G. Musiker, P. Plamondon, D. Thurston, M. Yakimov, and A. Zelevinsky.  Helpful comments and edits were provided by J. Geiger, J. Matherne, G. Musiker, and H. Thomas.

\section{Curves and links in marked surfaces}

This section gives our definitions of `curve', `multicurve' and `link' for marked surfaces.

%\subsection{Marked surfaces}
%
%In this paper, a \textbf{marked surface} $\S$ will be an oriented 2-dimensional real manifold\footnote{We do not assume $\S$ is connected.} $\S$, with boundary $\partial \S$, together with a finite set of \textbf{marked points} $\M\subset  \partial\S$ in the boundary.
%%\footnote{In this, we deviate from \cite{FST08}, who allow marked points in the interior of $\S$.  Hence, we are working in a narrower generality, but one which circumvents the use of `tagged arcs'.}
%%\begin{rem}
%%In this paper, marked surfaces will \underline{never have interior marked points}.  The results of this paper can be extended to marked surfaces with interior marked points only for the commutative specialization $\rq= 1$.  The quantum case is obstructed; see Remark \ref{rem: punctures}.
%%%
%%%The case of marked surfaces with interior marked points will be addressed in a later paper; see Remark \ref{rem: punctures}.
%%\end{rem}
%
%By the classification of surfaces, any connected marked surface can be constructed by starting with a genus $g$ surface, cutting out a finite number of discs, and marking a finite number of points on the boundary of each disc.  The boundary of an excised disc will be called a \emph{hole}.%, and a marked point in the interior of $\S$ will be called a \textbf{puncture}.

\subsection{Curves}

A \textbf{(framed) curve} $\curve$ in $\S$ is an immersion $\curve:C\rightarrow \S$ of a compact, connected, 1-dimensional manifold into $\S$, such that any boundary of $C$ maps to $\M$ and the interior of $C$ does not map to $\M$.  There are two kinds of curves.
\begin{itemize}
\item \textbf{Arcs}: curves with endpoints in $\M$.
\item \textbf{Loops}: closed loops without endpoints.
\end{itemize}
Homotopies between curves are always through the class of curves; that is, we only allow homotopies during which...
\begin{itemize}
\item $C$ remains immersed (\emph{regular homotopy}),
\item the endpoints remain in $\M$ (\emph{endpoint-fixed}), and
\item the interior remains disjoint from $\M$.
\end{itemize}
 %\footnote{This sneaks a framing in, because Reidemeister 1 necessary moves through a non-immersed curve.}
As an abuse of terminology, two curves will be called \textbf{homotopic} if they may be related by homotopy and orientation-reversal (so a homotopy class has no intrinsic orientation).

\subsection{Multicurves}

A \textbf{multicurve} $\multi$ in $(\S,\M)$ will mean an unordered, finite set of curves in $\S$ which may contain duplicates (i.e., homotopic curves).  Two multicurves are \textbf{homotopic} if there is a bijection between their constituent curves which takes a curve to a homotopic one.  A curve can always be thought of as a single element multicurve.

We will often focus on multicurves locally, by restricting to arbitrarily small discs around a point in $\S$.  A \emph{strand} in a multicurve $\multi$ near a point $p\in\S$ will be a component of the restriction of $\multi$ to an arbitrarily small disc around $p$.

A multicurve $\multi$ is \textbf{transverse} if...
\begin{itemize}
\item at each intersection in $\multi$, each strand has a different tangent direction, and
\item each interior intersection (called a \emph{crossing}) is between only two strands.
\end{itemize}
Every multicurve is homotopic to a transverse multicurve.

A transverse multicurve is \textbf{simple} if it has no interior intersections, and no curves which are contractible.  Contractible curves are either topologically trivial loops (called \emph{unknots}) or arcs which cut out a disc (called \emph{contractible arcs}).%That is, any intersections between curves or from a curve to itself occur at the endpoints.  Each curve in a simple multicurve is again simple.

\begin{rem} A transverse multicurve will be drawn as the union of its curves.  By the transverse condition, it is unambigious what the constituent curves are.%, where homotopy has been used to remove any overlapping that would confuse the picture.
\end{rem}

%For any multicurve $\mc$, let the \textbf{intersection number} $\kappa(\mc)$ be the minimum (over all representatives of $\mc$) number of intersections which are not between endpoints.  A multicurve $\mc$ is \textbf{compatible} if $\kappa(\mc)=0$.  A \emph{compatible curve} is then a curve equivalent to a curve with no self-intersections.
%\begin{lemma}
%A multicurve $\mc$ is compatible if and only if each pair $\{c_1,c_2\}\subset\mc$ is compatible.
%\end{lemma}
%
%Given a compatible multicurve $\mc$ in $(\S,\M)$, there is a new marked surface $(\S',\M')$ which is the \textbf{cut} of $(\S,\M)$ along $\mc$.  Here, $\S'$ is the unique compact, oriented surface with boundary with a map $\gamma:\S'\rightarrow \S$ which is a homemorphism away from $\mc$, and $\gamma^{-1}(\mc)$ is contained in $\partial\S'$ and a 2-to-1 cover of $\mc$.  The marked region $M'$ is $\gamma^{-1}(\M)$.

\subsection{Links}\label{section: link}

%We now define links, by equipping a transverse multicurve with crossing information, where a strand at a crossing must pass over or under the other strand, and two strands must at a marked point must pass over, under, or be simultaneous (the equivalence relation).
%
We now define links, by equipping a transverse multicurve with crossing data, about which strands are `passing over' other strands.  It will be convenient to allow strands at a marked point to either pass over each other, or to arrive simultaneously.  This generalization is a convenience, not a necessity; see Remark \ref{rem: nosimul}.

%However, it will be convenient to allow strands to arrive simultaneously at a marked point.

A \textbf{(framed) link} $\link$ is a transverse multicurve $\multi$, together with...
\begin{itemize}
\item at each crossing, an ordering of the two strands,
\item at each marked point, an equivalence relation on the strands and an ordering on the equivalence classes of the strands.
\end{itemize}
Intuitively, a strand at a crossing must pass over or under the other strand, and two strands at a marked point must pass over, under, or be simultaneous (the equivalence relation).  This is drawn in the natural way (Figure \ref{fig: crossings}).

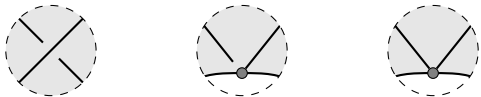
\begin{figure}[h!]
\begin{tikzpicture}
\begin{scope}[xshift=-.5in,scale=.15]
    \draw[fill=black!10,dashed] (0,0) circle (4);
    \draw[thick] (-2.83,-2.83) to (2.83,2.83);
    \draw[thick] (-2.83,2.83) to (-.71,.71);
    \draw[thick] (.71,-.71) to (2.83,-2.83);
\end{scope}
\begin{scope}[xshift=.5in,scale=.15]
	\begin{scope}
	\clip (0,0) circle (4);
	\draw[fill=black!10,thick] (-5,-3) to [in=180,out=30] (0,-2) to [in=150,out=0] (5,-3) to [line to] (5,5) to (0,5) to (-5,5);
	\node (1) at (0,-2) [marked] {};
	\draw[thick] (1) to (4,3);
	\draw[thick] (-.8,-1) to (-4,3);
		\end{scope}
	\draw[dashed] (0,0) circle (4);
\end{scope}
\begin{scope}[xshift=1.5in,scale=.15]
	\begin{scope}
	\clip (0,0) circle (4);
	\draw[fill=black!10,thick] (-5,-3) to [in=180,out=30] (0,-2) to [in=150,out=0] (5,-3) to [line to] (5,5) to (0,5) to (-5,5);
	\node (1) at (0,-2) [marked] {};
	\draw[thick] (1) to (4,3);
	\draw[thick] (1) to (-4,3);
		\end{scope}
	\draw[dashed] (0,0) circle (4);
\end{scope}
\end{tikzpicture}
\caption{Crossings, ordered strands, and simultaneous strands}
\label{fig: crossings}
\end{figure}

A simple multicurve $\multi$ can be regarded as a link with the simultaneous ordering at each endpoint; this will also be denoted by $\multi$.

\begin{rem}
Knot theory often considers `links', which would be links without arcs by the above definition, and `virtual links', which would be links without arcs, but where simultaneous crossings are allowed \cite{Kau99}.  Thus, the above definition can be thought of as `links with endpoints in $\M$, which can be virtual links at their endpoints'.
\end{rem}

Links without arcs arise in knot theory, as projections of knots in 3-dimensional space onto 2-dimensional space.  Similarly, our notion of links can be thought of as describing a multicurve in $\S\times [0,1]$, where $[0,1]$ is the dimension coming `out of the paper'.

\emph{Homotopies} between links are through the class of transverse multicurves, where crossing data are not changed.  This means the intersections are required to stay transverse, and so intersections can neither be created nor removed (in contrast with our definition of homotopy of multicurves).  We will say two links are \emph{homotopic} if they may be related by homotopy and orientation-reversal.

\begin{rem}
This notion of equivalence is weaker than the usual definition of equivalent links in knot theory, which uses Reidemeister moves and captures the notion of when two links describe ambient isotopic links in $3$-dimensional space.  This difference will become irrelevant later, as the skein relations will imply the Reidemeister moves (Proposition \ref{prop: Reide}).
\end{rem}

%\begin{rem}
%Frequently, it will be necessary to consider a link or multicurve locally (in a small disc).  These will be denoted by a dashed circle.  When more than one dashed circle appears in an expression or equation, it will be understood that they describe identical links or multicurves outside the disc.  A local component of the link or multicurve in a disc will be called a \textbf{strand}; this allows for two strands to be part of the same curve in the larger picture.
%\end{rem}

%To reflect their additional structure, links have a weaker notion of equivalence.  Two links are \textbf{equivalent} if they can be related by...
%\begin{itemize}
%\item homotopy through the class of transverse multicurves (the transverse requirement means that intersections can neither be created nor removed),
%\item Modified Reidemeister 1,
%%\begin{center}
%%\begin{tikzpicture}
%%
%%\end{tikzpicture}
%%\end{center}
%\item Reidemeister 2,
%\item Marked Reidemeister 2, and
%\item Reidemeister 3.
%\end{itemize}
%This notion of equivalence is intended to capture the behavior of ambient isotopy of links in 3-dimensional space.  The only relation new to the context of marked surfaces is the Marked Reidemeister 2.

%\begin{center}
%{\large \textsm Marked surfaces without punctures}
%\end{center}

%\specialsection*{\textbf{I. Marked surfaces without punctures}}

\section{The skein algebra $\Sk_q(\S)$}

%For the following section, let $\S$ be a marked surface \emph{without punctures}.
Inspired by knot theory, we now define an algebra associated to a marked surface, which consists of linear combinations of links modulo certain local relations, and whose product corresponds to superimposing links.

\subsection{The skein relations} Let $\Zq$ denote the ring $\mathbb{Z}[q^{\pm\frac{1}{2}}]$ of Laurent polynomials in the indeterminant $q^{\frac{1}{2}}$. %\footnote{We only need the algebra $\mathbb{Z}[q^{\pm1}]$, but later expressions will be simpler with a square root of $q$.}
For any marked surface $\S$, let $\Zq^{\Links}$ denote the free $\Zq$-module with basis given by equivalence classes of links in $\S$.  %This module can be made into a $\Zq$-algebra by the \textbf{superposition product}.  The product of two links $\alpha$ and $\beta$ is given first by finding equivalent $\alpha'$ and $\beta'$ such that $\alpha'\cup \beta'$ is a transverse multicurve, and then ordering the crossings of $\alpha'\cup \beta'$ so that each curve in $\alpha'$ always passes over each curve in $\beta'$.

%\begin{prop}
%The superposition product gives a well-defined associative $\Zq$-linear product on $\Zq^{\Links}$ with identity corresponding to the empty link.
%\end{prop}

%\subsection{The skein algebra}

We will define a quotient $\Zq$-module of $\Zq^{\Links}$ by imposing several classes of relations (Figure \ref{fig: skein}), which are all defined in terms of local manipulations of a link.  These relations are expressed in terms of small discs, where it is understood that they describe links identical to each other outside the disc.
We also allow additional, undrawn curves at marked points, provided their order with respect to the drawn curves and each other does not change.

\begin{figure}[h]
\centering
\begin{tabular}{|c|c|}
\hline
\begin{tikzpicture}
\path[use as bounding box] (-1.15in,-.3in) rectangle (1.15in,.3in);
\begin{scope}[xshift=-.85in,scale=.15]
    \draw[fill=black!10,dashed] (0,0) circle (4);
    \draw[thick] (-2.83,-2.83) to (2.83,2.83);
    \draw[thick] (-2.83,2.83) to (-.71,.71);
    \draw[thick] (.71,-.71) to (2.83,-2.83);
\end{scope}
\node (=) at (-.5in,0) {$=$};
\node (q) at (-.375in,0) {$q$};
\begin{scope}[xshift=-.05in,scale=.15]
    \draw[fill=black!10,dashed] (0,0) circle (4);
    \draw[thick] (-2.83,-2.83) to [out=45,in=-45] (-2.83,2.83);
    \draw[thick] (2.83,-2.83) to [out=135,in=-135] (2.83,2.83);
\end{scope}
\node (+) at (.3in,0) {$+$};
\node (q') at (.5in,.02in) {$q^{-1}$};
\begin{scope}[xshift=.85in,scale=.15]
    \draw[fill=black!10,dashed] (0,0) circle (4);
    \draw[thick] (-2.83,-2.83) to [out=45,in=135] (2.83,-2.83);
    \draw[thick] (-2.83,2.83) to [out=-45,in=-135] (2.83,2.83);
\end{scope}
\end{tikzpicture}
&
\begin{tikzpicture}
\path[use as bounding box] (-.8in,-.3in) rectangle (1.6in,.3in);
\node (q) at (-.625in,0.04in) {$q^{-\frac{1}{2}}$};
\begin{scope}[xshift=-.25in,scale=.15]
	\begin{scope}
	\clip (0,0) circle (4);
	\draw[fill=black!10,thick] (-5,-3) to [in=180,out=30] (0,-2) to [in=150,out=0] (5,-3) to [line to] (5,5) to (0,5) to (-5,5);
	\node (1) at (0,-2) [marked] {};
	\draw[thick] (1) to (4,3);
	\draw[thick] (-.8,-1) to (-4,3);
	\end{scope}
	\draw[dashed] (0,0) circle (4);
\end{scope}
\node (=) at (.1in,0) {$=$};
\begin{scope}[xshift=.45in,scale=.15]
	\begin{scope}
	\clip (0,0) circle (4);
	\draw[fill=black!10,thick] (-5,-3) to [in=180,out=30] (0,-2) to [in=150,out=0] (5,-3) to [line to] (5,5) to (0,5) to (-5,5);
	\node (1) at (0,-2) [marked] {};
	\draw[thick] (1) to (4,3);
	\draw[thick] (1) to (-4,3);
	\end{scope}
	\draw[dashed] (0,0) circle (4);
\end{scope}
\node (=) at (.8in,0) {$=$};
\node (q) at (.975in,0.04in) {$q^{\frac{1}{2}}$};
\begin{scope}[xshift=1.3in,scale=.15]
	\begin{scope}
	\clip (0,0) circle (4);
	\draw[fill=black!10,thick] (-5,-3) to [in=180,out=30] (0,-2) to [in=150,out=0] (5,-3) to [line to] (5,5) to (0,5) to (-5,5);
	\node (1) at (0,-2) [marked] {};
	\draw[thick] (1) to (-4,3);
	\draw[thick] (.8,-1) to (4,3);
	\end{scope}
	\draw[dashed] (0,0) circle (4);
\end{scope}
\end{tikzpicture}
\\
The Kauffman skein relation
&
The boundary skein relation
\\

\hline

\begin{tikzpicture}
\path[use as bounding box] (-1.15in,-.3in) rectangle (.9in,.3in);
\begin{scope}[xshift=-.85in,scale=.15]
    \draw[fill=black!10,dashed] (0,0) circle (4);
    \draw[thick] (0,0) circle (2);
\end{scope}
\node (=) at (-.5in,0) {$=$};
\node (2) at (-.05in,.02in) {$-(q^2+q^{-2})$};
\begin{scope}[xshift=.6in,scale=.15]
    \draw[fill=black!10,dashed] (0,0) circle (4);
\end{scope}
\end{tikzpicture}
&
\begin{tikzpicture}
\path[use as bounding box] (-1.4in,-.3in) rectangle (.875in,.3in);
\begin{scope}[xshift=-1.1in,scale=.15]
	\begin{scope}
	\clip (0,0) circle (4);
	\draw[fill=black!10,thick] (-5,-3) to [in=180,out=30] (0,-2) to [in=150,out=0] (5,-3) to [line to] (5,5) to (0,5) to (-5,5);
	\node (1) at (0,-2) [marked] {};
	\draw[thick] (1) to [out=45,in=0] (0,2) to [out=180,in=135] (-.8,-1);
	\end{scope}
	\draw[dashed] (0,0) circle (4);
\end{scope}
\node (='') at (-.75in,0) {$=$};
\begin{scope}[xshift=-.4in,scale=.15]
	\begin{scope}
	\clip (0,0) circle (4);
	\draw[fill=black!10,thick] (-5,-3) to [in=180,out=30] (0,-2) to [in=150,out=0] (5,-3) to [line to] (5,5) to (0,5) to (-5,5);
	\node (1) at (0,-2) [marked] {};
	\draw[thick] (1) to [out=45,in=0] (0,2) to [out=180,in=135] (1);
	\end{scope}
	\draw[dashed] (0,0) circle (4);
\end{scope}
\node (=) at (-.05in,0) {$=$};
\begin{scope}[xshift=.3in,scale=.15]
	\begin{scope}
	\clip (0,0) circle (4);
	\draw[fill=black!10,thick] (-5,-3) to [in=180,out=30] (0,-2) to [in=150,out=0] (5,-3) to [line to] (5,5) to (0,5) to (-5,5);
	\node (1) at (0,-2) [marked] {};
	\draw[thick] (1) to [out=135,in=180] (0,2) to [out=0,in=45] (.8,-1);
	\end{scope}
	\draw[dashed] (0,0) circle (4);
\end{scope}
\node (=') at (.65in,0) {$=$};
\node (0) at (.775in,.01in) {$0$};
\end{tikzpicture}
\\
The value of the unknot
&
The value of a contractible arc
\\
\hline
\end{tabular}
\caption{The defining relations of $\Sk_q(\S)$.}
\label{fig: skein}
\end{figure}
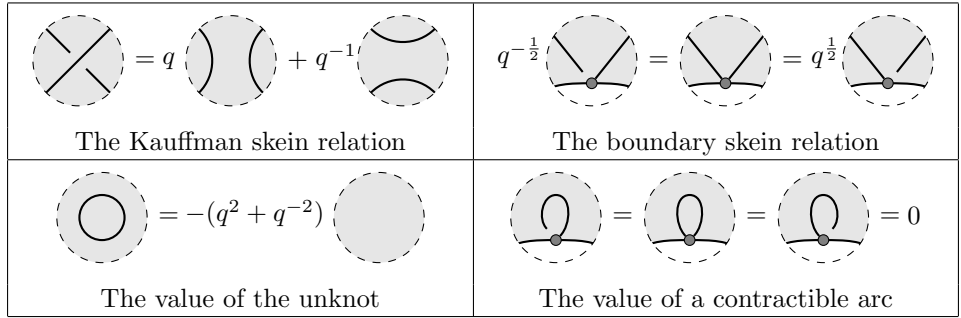

Define the quotient $\Zq$-module
\[ \Sk_q(\S):= \Zq^{\Links}/I\]
where $I$ is the submodule generated by the set $\{l-r\}$, running over relations of the form $l=r$ in Figure \ref{fig: skein}.  For a link $\link$, the class of $\link$ in $\Sk_q(\S)$ will be denoted $[\link]$.

\begin{rem}\label{rem: nosimul}
By the boundary skein relation, $\Sk_q(\S)$ is spanned over $\Zq$ by classes of links with no simultaneous endpoints.  It would have been possible to define $\Sk_q(\S)$ only in terms of those links; this would also eliminate the need for choosing a square root of $q$.  However, allowing simultaneous endpoints gives topological realizations of the multicurve elements defined in Section \ref{section: multicurve}.
%
%However, allowing simultaneous endpoints gives us more topologically-defined elements in $\Sk_q(\S)$, which will be important in the next section.
\end{rem}

\subsection{The Reidemeister moves}

The relations imposed in $\Sk_q(\S)$ imply additional local relations which will be important (Figure \ref{fig: Reide}).  These are the (modified) Reidemeister moves from knot theory, together with an additional relation coming from the addition of marked endpoints.

\begin{figure}[h!]
\centering
\begin{tabular}{|c|c|}
\hline
\begin{tikzpicture}
\path[use as bounding box] (-.8in, -.3in) rectangle (.8in,.3in);
\begin{scope}[xshift=-.4in,scale=.12]
    \draw[fill=black!10,dashed] (0,0) circle (5);
    \draw[thick] (2,2) to [out=90,in=45] (-1,2) to [out=-135, in=135] (-1,-2) to [out=-45,in=-90] (2,-2);
    \draw[thick,line width=1.5mm,draw=black!10] (-3,-4) to [out=53,in=-135] (-1,-2) to [out=45,in=90] (2,-2);
    \draw[thick,line width=1.5mm,draw=black!10] (-3,4) to [out=-53,in=135] (-1,2) to [out=-45,in=-90] (2,2);
    \draw[thick] (-3,-4) to [out=53,in=-135] (-1,-2) to [out=45,in=90] (2,-2);
    \draw[thick] (-3,4) to [out=-53,in=135] (-1,2) to [out=-45,in=-90] (2,2);
\end{scope}
\node (=) at (0,0) {$=$};
\begin{scope}[xshift=.4in,scale=.12]
    \draw[fill=black!10,dashed] (0,0) circle (5);
    \draw[thick] (-3,4) to [out=-53,in=90] (-1,0) to [out=270,in=53] (-3,-4);
\end{scope}
\end{tikzpicture}
&
\begin{tikzpicture}
\path[use as bounding box] (-.8in, -.3in) rectangle (.8in,.3in);
\begin{scope}[xshift=-.4in,scale=.12]
    \draw[fill=black!10,dashed] (0,0) circle (5);
    \draw[thick] (3,4) to [out=-127,in=90] (-2,0) to [out=270,in=127] (3,-4);
    \draw[thick,line width=1.5mm,draw=black!10] (-3,4) to [out=-53,in=90] (2,0) to [out=270,in=53] (-3,-4);
    \draw[thick] (-3,4) to [out=-53,in=90] (2,0) to [out=270,in=53] (-3,-4);
\end{scope}
\node (=) at (0,0) {$=$};
\begin{scope}[xshift=.4in,scale=.12]
    \draw[fill=black!10,dashed] (0,0) circle (5);
    \draw[thick] (3,4) to [out=-127,in=90] (1,0) to [out=270,in=127] (3,-4);
    \draw[thick,line width=1.5mm,draw=black!10] (-3,4) to [out=-53,in=90] (-1,0) to [out=270,in=53] (-3,-4);
    \draw[thick] (-3,4) to [out=-53,in=90] (-1,0) to [out=270,in=53] (-3,-4);
\end{scope}
\end{tikzpicture}
\\
%\hline
Modified Reidemeister 1.
&
Reidemeister 2.
\\
\hline
\begin{tikzpicture}
\path[use as bounding box] (-.8in, -.3in) rectangle (.8in,.3in);
\begin{scope}[xshift=-.4in,scale=.12]
	\begin{scope}
	\clip (0,0) circle (5);
	\draw[fill=black!10,thick] (-5,-4) to [in=180,out=30] (0,-3) to [in=150,out=0] (5,-4) to [line to] (5,6) to (0,6) to (-5,6);
	\node (1) at (0,-3) [marked] {};
	\draw[thick] (3,4) to [out=-127,in=90] (-2,0) to [out=270,in=135] (-1,-2);
	\draw[thick,line width=1.5mm,draw=black!10] (-3,4) to [out=-53,in=90] (2,0) to [out=270,in=45] (1);
	\draw[thick] (-3,4) to [out=-53,in=90] (2,0) to [out=270,in=45] (1);
	\end{scope}
	\draw[dashed] (0,0) circle (5);
\end{scope}
\node (=) at (0,0) {$=$};
\begin{scope}[xshift=.4in,scale=.12]
	\begin{scope}
	\clip (0,0) circle (5);
	\draw[fill=black!10,thick] (-5,-4) to [in=180,out=30] (0,-3) to [in=150,out=0] (5,-4) to [line to] (5,6) to (0,6) to (-5,6);
	\node (1) at (0,-3) [marked] {};
	\draw[thick] (3,4) to [out=-127,in=45] (1,-2);
	\draw[thick] (-3,4) to [out=-53,in=135] (1);
	\end{scope}
	\draw[dashed] (0,0) circle (5);
\end{scope}
\end{tikzpicture}
&
\begin{tikzpicture}
\path[use as bounding box] (-.8in, -.3in) rectangle (.8in,.3in);
\begin{scope}[xshift=-.4in,scale=.12]
    \draw[fill=black!10,dashed] (0,0) circle (5);
    \draw[thick,in=180,out=0] (-5,0) to (0,2) to (5,0);
    \draw[thick,in=60,out=-120,line width=1.5mm,draw=black!10] (2.5,4.33) to (1.73,-.5) to (-2.5,-4.33);
    \draw[thick,in=60,out=-120] (2.5,4.33) to (1.73,-.5) to (-2.5,-4.33);
    \draw[thick,in=-60,out=120,line width=1.5mm,draw=black!10] (2.5,-4.33) to (-1.73,-.5) to (-2.5,4.33);
    \draw[thick,in=-60,out=120] (2.5,-4.33) to (-1.73,-.5) to (-2.5,4.33);
\end{scope}
\node (=) at (0,0) {$=$};
\begin{scope}[xshift=.4in,scale=.12,rotate=180]
    \draw[fill=black!10,dashed] (0,0) circle (5);
    \draw[thick,in=180,out=0] (-5,0) to (0,2) to (5,0);
    \draw[thick,in=60,out=-120,line width=1.5mm,draw=black!10] (2.5,4.33) to (1.73,-.5) to (-2.5,-4.33);
    \draw[thick,in=60,out=-120] (2.5,4.33) to (1.73,-.5) to (-2.5,-4.33);
    \draw[thick,in=-60,out=120,line width=1.5mm,draw=black!10] (2.5,-4.33) to (-1.73,-.5) to (-2.5,4.33);
    \draw[thick,in=-60,out=120] (2.5,-4.33) to (-1.73,-.5) to (-2.5,4.33);
\end{scope}
\end{tikzpicture}
\\
Marked Reidemeister 2.
&
Reidemeister 3.
\\
\hline
\end{tabular}
\caption{The Reidemeister moves for links with endpoints.}
\label{fig: Reide}
\end{figure}
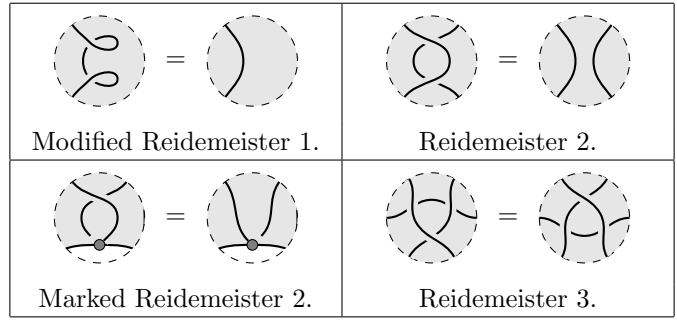

\begin{prop} \label{prop: Reide}
The locally defined relations in Figure \ref{fig: Reide} hold in $\Sk_q(\S)$.
\end{prop}
\begin{proof}
All four results follow from direct application of the relations in Figure \ref{fig: skein}.  We show the computation for Reidemeister 2; the others are similar.
\[\begin{tikzpicture}
\begin{scope}[xshift=0in,scale=.12]
    \draw[fill=black!10,dashed] (0,0) circle (5);
    \draw[thick] (3,4) to [out=-127,in=90] (-2,0) to [out=270,in=127] (3,-4);
    \draw[thick,line width=1.5mm,draw=black!10] (-3,4) to [out=-53,in=90] (2,0) to [out=270,in=53] (-3,-4);
    \draw[thick] (-3,4) to [out=-53,in=90] (2,0) to [out=270,in=53] (-3,-4);
\end{scope}
\node (=1) at (.35in,0) {$=$};
\node (q1) at (.5in,0.02in) {$q^2$};
\begin{scope}[xshift=.85in,scale=.12]
    \draw[fill=black!10,dashed] (0,0) circle (5);
    \draw[thick] (3,4) to [out=-127,in=0] (0,3) to [out=180,in=-53] (-3,4);
    \draw[thick] (-3,-4) to [out=53,in=270] (-.5,-2) to [out=90,in=270] (-2,0) to [out=90,in=180] (0,2) to [out= 0,in=90] (2,0) to [out=270,in=90] (.5,-2) to [out=270,in=127] (3,-4);
\end{scope}
\node (=2) at (1.2in,0) {$+$};
\begin{scope}[xshift=1.55in,scale=.12]
    \draw[fill=black!10,dashed] (0,0) circle (5);
    \draw[thick] (-3,-4) to [out=53,in=270] (-.5,-2) to [out=90,in=270] (-2,0) to [out=90,in=270] (-.5,2) to [out=90,in=-53] (-3,4);
    \draw[thick] (3,-4) to [out=127,in=270] (.5,-2) to [out=90,in=270] (2,0) to [out=90,in=270] (.5,2) to [out=90,in=-127] (3,4);
\end{scope}
\node (=3) at (1.9in,0) {$+$};
\begin{scope}[xshift=2.25in,scale=.12]
    \draw[fill=black!10,dashed] (0,0) circle (5);
    \draw[thick] (3,4) to [out=-127,in=0] (0,3) to [out=180,in=-53] (-3,4);
    \draw[thick] (3,-4) to [out=127,in=0] (0,-3) to [out=180,in=53] (-3,-4);
		\draw[thick] (0,0) circle (2);
\end{scope}
\node (=4) at (2.6in,0) {$+$};
\node (q4) at (2.8in,0) {$q^{-2}$};
\begin{scope}[xshift=3.2in,scale=.12,rotate=180]
    \draw[fill=black!10,dashed] (0,0) circle (5);
    \draw[thick] (3,4) to [out=-127,in=0] (0,3) to [out=180,in=-53] (-3,4);
    \draw[thick] (-3,-4) to [out=53,in=270] (-.5,-2) to [out=90,in=270] (-2,0) to [out=90,in=180] (0,2) to [out= 0,in=90] (2,0) to [out=270,in=90] (.5,-2) to [out=270,in=127] (3,-4);\end{scope}
\node (=5) at (3.55in,0) {$=$};
\begin{scope}[xshift=3.9in,scale=.12]
    \draw[fill=black!10,dashed] (0,0) circle (5);
    \draw[thick] (3,4) to [out=-127,in=90] (1,0) to [out=270,in=127] (3,-4);
    \draw[thick] (-3,4) to [out=-53,in=90] (-1,0) to [out=270,in=53] (-3,-4);
\end{scope}
\end{tikzpicture}
\]
The four terms in the middle come from applying the Kauffman skein relation to the two crossings.  The value of the unknot then cancels the first and last terms.
\end{proof}
\begin{rem}
The fixed values of unknots and contractible arcs defined in Figure \ref{fig: skein} are the only values for which the above Reidemeister moves hold (assuming the Kauffman and boundary skein relation).
\end{rem}

\begin{rem}
The Reidemeister moves describe the minimal relations needed to relate two links (drawn in $\S$ with crossings) which describe the ambient isotopic %\footnote{The `same' here means ambiently homotopic in $\S\times [0,1]$.}
 framed links in $\S\times [0,1]$.  The endpoints of the framed link in $\S\times[0,1]$ are required to stay in $\M\times [0,1]$, and the framing at the endpoints must stay tangent to $\M\times [0,1]$.
\end{rem}

\subsection{The superposition product}

The $\Zq$-module $\Sk_q(\S)$ may be equipped with a $\Zq$-bilinear, non-commutative product called the \emph{superposition product}.  If $\link$ and $\link[Y]$ are two links such that the union of the underlying multicurves $\multi\cup \multi[Y]$ is transverse, define the \emph{superposition} $\link\cdot\link[Y]$ to be the link which is $\multi\cup \multi[Y]$ where each strand of $\link$ crosses over each strand of $\link[Y]$ and all other crossings are ordered as in $\link$ and $\link[Y]$.
\begin{prop}
$[\link \cdot \link[Y]]$ only depends on the homotopy classes of $\link$ and $\link[Y]$.
\end{prop}
\begin{proof}
Let $(\link',\link[Y]')$ be a pair of links homotopic to $(\link,\link[Y])$, such that the union of underlying multicurves $\multi'\cup\multi[Y]'$ is transverse.  There exists a family of pairs of links $(\link_t,\link[Y]_t)$ for $t\in [0,1]$ such that...%\marginpar{Citation?}
\begin{itemize}
\item $\link_t$ is a homotopy between $\link$ and $\link'$,
\item $\link[Y]_t$ is a homotopy between $\link[Y]$ and $\link[Y]'$,
\item there is a finite subset $S\subset [0,1]$ such that, for $t\in [0,1]- S$, the union of underlying multicurves $\multi_t\cup \multi[Y]_t$ is transverse, and
\item for $t\in S$, $\multi_t\cup \multi[Y]_t$ is transverse except for a single intersection, which is of one of the three types in Figure \ref{fig: nontrans}.
\end{itemize}
\begin{figure}[h!]
\begin{tikzpicture}
\begin{scope}[xshift=-1in,scale=.12]
    \draw[fill=black!10,dashed] (0,0) circle (5);
    \draw[thick] (3,4) to [out=-127,in=90] (0,0) to [out=270,in=127] (3,-4);
    \draw[thick] (-3,4) to [out=-53,in=90] (0,0) to [out=270,in=53] (-3,-4);
\end{scope}
\begin{scope}[xshift=0in,scale=.12]
	\begin{scope}
	\clip (0,0) circle (5);
	\draw[fill=black!10,thick] (-5,-4) to [in=180,out=30] (0,-3) to [in=150,out=0] (5,-4) to [line to] (5,6) to (0,6) to (-5,6);
	\node (1) at (0,-3) [marked] {};
	\draw[thick] (3,4) to [out=-127,in=90] (1);
	\draw[thick] (-3,4) to [out=-53,in=90] (1);
	\end{scope}
	\draw[dashed] (0,0) circle (5);
\end{scope}
\begin{scope}[xshift=1in,scale=.12]
    \draw[fill=black!10,dashed] (0,0) circle (5);
    \draw[thick,in=180,out=0] (-5,0) to (0,0) to (5,0);
    \draw[thick,in=60,out=-120] (2.5,4.33) to (0,0) to (-2.5,-4.33);
    \draw[thick,in=-60,out=120] (2.5,-4.33) to (0,0) to (-2.5,4.33);
\end{scope}
\end{tikzpicture}
\caption{Elementary failures of transversality}
\label{fig: nontrans}
\end{figure}
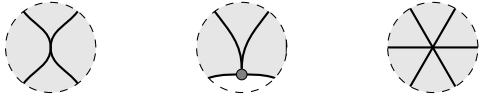
The superpositions $\link_{t_0}\cdot \link[Y]_{t_0}$ and $\link_{t_1}\cdot \link[Y]_{t_1}$ are homotopic if $t_0$ and $t_1$ are in the same component of $[0,1]-S$. If there is a single element of $S$ between $t_0$ and $t_1$, then the two superpositions will be related by Reidemeister 2, Marked Reidemeister 2, or Reidemeister 3, depending on which of the three non-transverse intersections occurs.  Then superpositions in adjacent components of $[0,1]-S$ are related by a single Reidemeister move, and so  $[\link\cdot \link[Y]]$ and $[\link'\cdot\link[Y]']$ are related by a finite sequence of Reidemeister moves.
\end{proof}
\begin{rem}
The quotient of $\Zq^{\Links}$ by the $\Zq$-submodule generated by Reidemeister 2, Marked Reidemeister 2 and Reidemeister 3 also admits a well-defined superposition product, and $\Sk_q(\S)$ can be defined as a quotient algebra of this algebra.  Modified Reidemeister 1 is unnecessary for the product to be well-defined.
\end{rem}

For general $\link$ and $\link[Y]$, define the \textbf{superposition product} $[\link][\link[Y]]$ by choosing homotopic links $\link'$ and $\link[Y]'$ such that $\link'\cup \link[Y]'$ is transverse, and letting
\[ [\link][\link[Y]] := [\link'\cdot \link[Y]']\]
By the proposition, this doesn't depend on the  choice of $\link'$ and $\link[Y]'$.  Extend this product to all of $\Sk_q(\S)$ by $\Zq$-bilinearity.

\begin{prop}\label{prop: skeinalgebra}
The superposition product makes $\Sk_q(\S)$ into an associative $\Zq$-algebra with unit $[\emptyset]$, the class of the empty link.
\end{prop}
\begin{proof}
For links $\link, \link[Y],\link[Z]$, find homotopic links $\link',\link[Y]',\link[Z]'$ such that the union of the underlying multicurves $\multi'\cup\multi[Y]'\cup\multi[Z]'$ is transverse. Then
\[ ([\link][\link[Y]])[\link[Z]] = [\link'\cdot\link[Y]'\cdot\link[Z]']=[\link]([\link[Y]][\link[Z]])\]
We also have $[\link][\emptyset]=[\link\cdot\emptyset]=[\link]=[\emptyset\cdot\link] =[\emptyset][\link]$.
\end{proof}

\begin{defn}
The algebra $\Sk_q(\S)$ is the \textbf{(Kauffman) skein algebra} of $\S$.
\end{defn}

When $\S$ has no marked points, this definition coincides with the usual definition of the Kauffman skein algebra of an (unmarked) surface, defined in \cite{Prz91}.  %This algebra has been extensively studied, for its connections to knot theory, quantum field theory and quantum groups.  A survey of this theory can be found in \ref{?}.  A central theme of this paper is how generalizing to marked surfaces can result in more structure and useful techniques; eg, embeddings into quantum tori (Theorem \ref{thm: Laurent}).

\begin{rem}
Some authors replace $\Zq$ with a field $\k$ with a distinguished non-zero element $\lambda$, which plays the role of $\rq$.  This setup can be recovered from ours as follows. The map $\Zq\rightarrow \k$ with $\rq\mapsto \lambda$ makes $\k$ into a $\Zq$-algebra.  Then the $\k$-algebra $\k\otimes_{\Zq}\Sk_q(\S)$ is the skein algebra defined over $\k$.  Since $\Sk_q(\S)$ is a free $\Zq$-module (Lemma \ref{lemma: basis}), no torsion complications arise.
\end{rem}

\begin{rem}
In \cite[Definition 2.5]{RY11}, the authors also generalize skein algebras to `marked surfaces'. However, they generalize skein algebras in an orthogonal direction, in that they require $\partial\S=\emptyset$ but allow interior marked points.  It is not clear if the two definitions can be combined in some `best' way; see Remark \ref{rem: punctures}.
\end{rem}

%\section{Structures on $\Sk_q(\S)$ and $\Sk_q^o(\S)$}
%
%The skein algebra $\Sk_q(\S)$ has several extra structures we will need.

\subsection{The bar involution}\label{section: bar}

For $\link$ any link, let $\link^\dag$ be the link with the same underlying multicurve, but all crossing orders reversed.
\begin{prop}
The map $[\link]^\dag:= [\link^\dag]$ and $(\rq)^\dag:=q^{-\frac{1}{2}}$ extends to an involutive ring antiautomorphism of $\Sk_q(\S)$, called the \textbf{bar involution}.
\end{prop}
\begin{proof}
Let $\dag:\Zq^{\Links(\S)}\rightarrow \Zq^{\Links(\S)}$ send $[\link]$ to $[\link^\dag]$ and $\rq$ to $q^{-\frac{1}{2}}$; this map is manifestly an involution.  Each relation in Figure \ref{fig: skein} goes to a relation of the same type, and so there is a quotient involution $\dag:\Sk_q(\S)\rightarrow \Sk_q(\S)$.

For links $\link,\link[Y]$, let $\link'$ and $\link[Y]'$ be homotopic links with $\link'\cup\link[Y]'$ transverse.  Then
\[ [\link]^\dag[\link[Y]]^\dag = [\link^\dag][\link[Y]^\dag] = [\link^\dag \cdot \link[Y]^\dag] = [(\link[Y]'\cdot \link')^\dag]
= [\link[Y]'\cdot \link']^\dag = ([\link[Y]] [\link])^\dag\]
Since $[\emptyset]^\dag=[\emptyset]$, this is a ring homomorphism.
\end{proof}

The bar involution will be useful for two reasons.  First, it shows $\Sk_q(\S)$ is isomorphic to its opposite algebra, which cuts some proofs in half.  Second, we are particularly interested in elements of $\Sk_q(\S)$ which are fixed by the bar involution.

\subsection{The endpoint $\E$-grading}\label{section: grading}

The skein algebra had an \emph{endpoint grading}, where the degree of an arc is the formal sum of its endpoints, in the lattice $\Z^\M$ spanned by the marked points. This grading restricts to the following sublattice $\E$ in $\Z^\M$. % parametrized by the marked points.
%For any set $\M$, define the free abelian group
%Let $\N^\M$ be the set of maps $f:\M\rightarrow \N$; this is naturally a semigroup by addition in the target.
%
%Let $\M$ denote the set of marked points in $\S$, and define the following semigroup.  %Let $\N^\M$ denote the semi-group of functions $f:\M\rightarrow \N$, and let $E\subset\N^\M$ be the sub-semigroup such that $\sum_{m\in \M}f(m)$ is even.
\[ \E:= \left\{f:\M\rightarrow \Z\, |\, \forall \text{ connected components } \S'\subseteq \S, \sum_{m\in \M\cap\S'}f(m) \text{ is even}\right\}\]
Let $\E_+$ be the subsemigroup whose image lands in $\N\subset \Z$.

%When $\M$ is the marked points in $\S$ and
For any
 $f:\M\rightarrow \Z$, let $(\Sk_q(\S))_f$ be the $\Z_q$-submodule spanned by links with $f(m)$ strands at each marked point $m$; note that this is zero unless $f\in \E_+$.
\begin{prop}
This defines an $\E_+$-grading on $\Sk_q(\S)$.
\end{prop}
\begin{proof}
Two homotopic links have the same set of endpoints, so $\Zq^{\Links}$ is naturally $\E_+$-graded.  The defining relations in $\Sk_q(\S)$ are $\E_+$-homogeneous by inspection.
\end{proof}
%A map $f:\M\rightarrow \M'$ induces an inclusion of semigroups
%\[ f_*:\E_+(\M)\rightarrow \E_+(\M')\]
%
%let $\E_+'\subset \E_+$ be the subsemigroup of maps $f:\M\rightarrow \N$ which are zero on $\M-\M'$.  There is a natural quotient map of semigroups $\E_+\rightarrow \E_+'$ induced by the inclusion $\M'\rightarrow \M$.
%
%Then there is a canonical inclusion
%\[ \Sk_q(\S')\rightarrow \Sk_q(\S)\]
%%which sends links to themselves.
%%
%%quotient
%%\[ \Sk_q(\S)\rightarrow\Sk_q(\S')\]
%%and
The degree zero part is the subalgebra $(\Sk_q(\S))_0$ spanned by links without arcs; this is isomorphic to $\Sk_q(\S_0)$, where $\S_0$ is the unmarked version of $\S$.

%\subsection{Functoriality of $\Sk_q$}
%
%The construction $\S\mapsto \Sk_q(\S)$ can be made functorial as follows.  Define a \emph{map of marked surfaces} $f:\S\rightarrow \S'$ to be an inclusion of the underlying surfaces\footnote{We do not require that $f(\partial\S)\subset \partial \S'$.}, together with a distinguished subset $S\subset \M$ such that $f(S)\subset \M'$.  That is, it is a map of surfaces where marked points can be created or destroyed, and the set $S$ specifies which marked points do not get destroyed.
%
%Any link $\link\subset \S$ whose endpoints are in $S$ maps to a link $f(\link)\subset \S'$.  Define
%\[ f_*:\Sk_q(\S)\rightarrow \Sk_q(\S')\]
%to be the $\Zq$-module map which sends $\link$ to $f(\link)$ if the endpoints are in $S$, and $0$ otherwise; this can be checked to be a homomorphism of $\Zq$-algebras.

%\subsection{The $q=1$ skein algebra}
%
%Let $\Sk_1(\S)$ the quotient of $\Sk_q(\S)$ by the ideal generated by $q-1$; this specializes the value of $q$ to $1$.
%
%\begin{lemma}
%For links $\link,\link[y]$ whose underlying multicurves are equivalent, $[\link]=[\link[y]]\in\Sk_1(\S)$.
%\end{lemma}
%That is, $[\link]\in \Sk_1(\S)$ does not depend on the ordering of its crossings.  In this way, $\Sk_1(\S)$ can be generated as a $\Z$-module by equivalence classes of multicurves, instead of links.
%\begin{coro}
%$\Sk_1(\S)$ is a commutative ring.
%\end{coro}

\section{Simple multicurves}\label{section: multicurve}

\subsection{Simple multicurves in $\Sk_q(\S)$}

Recall that a \textbf{simple multicurve} $\multi$ in $\S$ is a transverse multicurve with no crossings, no unknots and no contractible arcs.  A simple multicurve can be regarded as a link with simultaneous endpoints; let $[\multi]$ denote the corresponding element in $\Sk_q(\S)$.  No factor of $q$ appears in the definition of $[\multi]$, though it can be defined as a $q$-multiple of an ordered version of $\multi$.

This element is fixed by the bar involution,
that is, $[\multi]^\dag= [\multi]$.
%\begin{prop}
%If $\link$ and $\link'$ are two links with the same underlying simple multicurve $\multi$, then $[\link]$, $[\link']$ and $[\multi]$ are $\rq$-multiples of each other in $\Sk_q(\S)$.
%\end{prop}
%For any link $\link$ whose underlying multicurve is $\multi$, $[\link]$ is a $\rq$-multiple of $[\multi]$, and all such classes of links are $\rq$-multiples of each other.
Moreso, the element $[\multi]$ is the only $\rq$-multiple of itself or any other ordering of its endpoints which is fixed by the bar-involution.  This gives an alternate definition of $[\multi]$.

Let $\Multi$ be the set of homotopy classes of simple multicurves.
\begin{lemma}\label{lemma: basis}
Under $\multi\mapsto[\multi]$, the set $\Multi$ maps to a $\Zq$-basis of $\Sk_q(\S)$.
\end{lemma}
\begin{proof}
Let $\Zq^\Multi$ be the free $\Zq$-module with basis $\Multi$.  There is a map
\[s:\Zq^\Multi\rightarrow \Sk_q(\S)\]
which sends $\multi$ to $[\multi]$.

Define a map $\tilde{r}:\Zq^{\Links}\rightarrow \Zq^{\Multi}$ as follows.  Let $\link$ be a link.
\begin{enumerate}
\item First, find $n\in \Z$ such that $[\link]=q^{\frac{n}{2}}[\link']$, where $\link'$ is identical to $\link$ except with the simultaneous ordering on endpoints.
\item Then, by applying the Kauffman skein relation to each crossing in $\link'$, find links $\link_i$ and $m_i\in \Z$ (for an index set $I$) such that each $\link_i$ has no crossings, and
\[ [\link'] =\sum_{i\in I} q^{m_i}[\link_i]\]
\item Finally, remove contractible components of each $\link_i$ using the defined values.  That is, let $\overline{\link}_i$ be $\link_i$ with the contractible components removed, and let $\lambda_i\in \Zq$ be such that $[\link_i]=\lambda_i[\overline{\link}_i]$.
\end{enumerate}
Since each $\overline{\link}_i$ is a simple multicurve, define
\[ \tilde{r}(\link) = \sum_{i\in I} (q^{\frac{n}{2}+m_i}\lambda_i)\overline{\link}_i\in \Zq^\Multi\]
Because the relations are local, the steps in the construction of $r(\link)$ could have been taken in any order, with the exception of removing contractible components created by applying the Kauffman skein relation.  It follows that $r$ descends to a map
\[ r:\Sk_q(\S)\rightarrow \Zq^{\Multi}\]
By construction, $s(r([\link]))=[\link]$ as elements of $\Sk_q(\S)$.  If $\multi$ is a simple multicurve, then the construction of $r(\multi)$ makes no changes, and so $r(s(\multi))=\multi$.  Then $s$ and $r$ are inverses.
%
%For a link $\link$, the intersections cannot be created or destroyed by homotopy.  This means that the defining relations of $\Sk_q(\S)$ can be applied in any order.
%
%We now construct an inverse to the above map. For an arbitrary link $\link$,
%\begin{enumerate}
%\item first, use the values of an unknot and any contractible arc to factor them out of $[\link]$,
%\item then, express the resulting multiple of $[\link]$ as a multiple of $[\link']$, where $\link'$ is the same link as $\link$ except all the endpoints of arcs are simultaneous, and
%\item finally, use the Kauffman skein relation on each intersection in $\link'$ to express $[\link']$ as a $\Zq$-linear combination of simple, simultaneous multicurves. SWITCH ORDER!
%\end{enumerate}
%The resulting expression is well-defined, by the previous observation that the above steps commute.
%This expression of $[\link]$ as a $\Zq$-combination of simple, simultaneous multicurves extends to a map $\Sk_q(\S)\rightarrow \Zq^\Multi$.  It is immediate that this map is a two-sided inverse to the map $\multi\mapsto [\multi]$.
\end{proof}
For any element $x\in \Sk_q(\S)$, there is a unique subset $Supp(x)\subset \Multi$ (called the \textbf{support} of $x$) and unique non-zero $\lambda_{\multi[Y]}\in \Zq$ for each $\multi[Y]\in Supp(x)$, such that
\[ x=\sum_{\multi[Y]\in Supp(x)} \lambda_{\multi[Y]}[\multi[Y]]\]

\begin{rem}\label{rem: support}
Because the skein relations are local, for any link $\link$, there are simple multicurves $\multi_i$ which are each identical to $\link$ away from small neighborhoods of each crossing and marked point, such that $Supp([\link])=\{\multi_i\}$.
\end{rem}

\begin{coro}\label{coro: gens}
The $\Zq$-algebra $\Sk_q(\S)$ is generated by the set of simple curves.
\end{coro}
\begin{proof}
If $\multi$ is a simple multicurve consisting of simple curves $\curve_1,\curve_2,...,\curve_n$, then $[\multi] = q^{\frac{\lambda}{2}}[\curve_1][\curve_2]...[\curve_n]$ for some $\lambda\in \mathbb{Z}$.  Then the simple curves generate a $\Zq$-subalgebra of $\Sk_q(\S)$ which contains a basis, and so it coincides with all of $\Sk_q(\S)$.
\end{proof}

\subsection{Counting crossings}

For any two simple multicurves $\multi$ and $\multi[Y]$, let $\mu(\multi,\multi[Y])$ denote the minimum number of crossings between $\multi'$ and $\multi[Y]'$, over all transverse pairs $(\multi',\multi[Y]')$ homotopic to $(\multi,\multi[Y])$.  Note that intersections at marked points are not counted.  We will say $\multi$ and $\multi[Y]$ have \emph{minimal crossings} if $\multi\cdot\multi[Y]$ has $\mu(\multi,\multi[Y])$ crossings.

 %If $\mu(\multi,\multi[Y])=0$, then there is a simple multicurve $\multi\cup \multi[Y]$ which is unique up to homotopy.
%A key idea is that multiplying a link by a simple arc reduces the number of intersections with that arc.
%\begin{lemma}
%Let $\curve$ be a simple arc, and let $\link[Y]$ be a link with $\mu(\curve,\link[Y])>0$.  Then there are links $\link[Z]_1,\link[Z]_2$ such that
%\[ [\curve][\link[Y]]=q[\link[Z]_1]+q^{-1}[\link[Z]_2]\]
%and $\mu(\curve,\link[Z]_i)<\mu(\curve,\link[Y])$ for all $i$.
%\end{lemma}
\begin{lemma}\label{lemma: simul}\cite{FHS82}
Let $\multi_1,\multi_2,...,\multi_n$ be a finite collection of simple multicurves.  Then there are simple multicurves $\multi_1',\multi_2',...,\multi_n'$ such that,
\begin{itemize}
\item for all $i$, $\multi_i'$ is homotopic to $\multi_i$, and
\item for all $i$ and $j$, $\multi_i'$ and $\multi_j'$ have minimal crossings.
\end{itemize}
\end{lemma}
\begin{proof}[Idea of proof]
This is done by choosing a hyperbolic metric on $\S$.  Then curve-shortening flow takes $\multi_i$ to a geodesic $\multi_i'$, which also minimizes pairwise intersections.  This may create intersections of higher order, but these can be resolved by a small perturbation.
\end{proof}
\begin{coro}\label{coro: crossingsum}
If $\multi$ and $\multi[Y]$ are simple multicurves with components $\curve_1,\curve_2,...,\curve_m$ and $\curve[y]_1,\curve[y]_2,...,\curve[y]_n$, then
\[ \mu(\multi,\multi[Y]) = \sum_{\substack{1\leq i\leq n \\ 1\leq j\leq m}} \mu(\curve_i,\curve[y]_j)\]
\end{coro}
\begin{proof}
By Lemma \ref{lemma: simul}, find $\multi'$ and $\multi[Y]'$ homotopic to $\multi$ and $\multi[Y]$, respectively, so that $\multi'\cup\multi[Y]'$ is transverse and each pair of components has minimal crossings. In particular, components in $\multi'$ do not cross each other, and so $\multi'$ is still a simple multicurve; likewise, $\multi[Y]'$ is still a simple multicurve.  Since these are simple multicurves homotopic to $\multi$ and $\multi[Y]$ with minimal total crossings, $\mu(\multi,\multi[Y])$ is the number of crossings in $\multi'\cup\multi[Y]'$, which can be counted by summing over all pairs.
\end{proof}

Extend $\mu$ to a map
\[ \mu:\Sk_q(\S)\times\Sk_q(\S)\rightarrow \mathbb{N}\]
\[ \mu(x,y)
:= \max\{ \mu\left(\multi,\multi[Y]\right)\,|\,{\multi\in Supp(x),\multi[Y]\in Supp(y)}\}\]
%\[ \mu\left( \sum_{\multi\in Supp(x)}\lambda_{\multi}[\multi],\sum_{\multi[Y]\in Supp(y)}\lambda_{\multi[Y]}'[\multi[Y]]\right)
%:= \max\{ \mu\left(\multi,\multi[Y]\right)\,|\,{\multi\in Supp(x),\multi[Y]\in Supp(y)}\}\]
Define $\mu(0,x)=0$ for all $x$.
\begin{rem}
If $\link$ and $\link[Y]$ are general links, then $\mu([\link],[\link[Y]])$ is less than or equal to the minimum number of crossings between $\link'$ and $\link[Y]'$, over all pairs $(\link',\link[Y]')$ homotopic to $(\link,\link[Y])$.  Equality is not always true; see Lemma \ref{lemma: reducing} for an example.
\end{rem}
%\begin{prop}
%For links $\link$ and $\link[Y]$, $\mu([\link],[\link[Y]])$ is the minimum number of crossings between $\link'$ and $\link[Y]'$, over all pairs $(\link',\link[Y]')$ homotopic to $(\link,\link[Y])$.
%\end{prop}
%\begin{proof}
%First, choose $(\link',\link[Y]')$ so that $\link'\cdot\link[Y]'$ is transverse.  Then, by Remark \ref{rem: support}, the supports of $[\link']$ and $[\link[Y]']$ consist of simple multicurves which are identical to the original link away from intersections.  Then the maximal
%\end{proof}
For a fixed element $x\in \Sk_q(\S)$, $\mu(x,-)$ behaves like the degree of a polynomial.\footnote{If the reader enjoys complicated words, they may correctly call $\mu$ a \emph{bisubtropical} map.}
\begin{lemma}\label{lemma: crossings}
For $x,y,z\in \Sk_q(\S)$,
\begin{enumerate}
\item $\mu(x,y)=\mu(y,x)$.
\item $\mu(x,y+z)\leq \max(\mu(x,y),\mu(x,z))$.
\item $\mu(x,yz)\leq \mu(x,y)+\mu(x,z)$.
\item If $\mu(y,z)=0$, then $\mu(x,yz)=\mu(x,y)+\mu(x,z)$.
\end{enumerate}
\end{lemma}
\begin{proof}
The first two facts are clear from definitions, and the fourth follows from Corollary \ref{coro: crossingsum}.  The third fact is the only one which requires some work.

Let $\multi$, $\multi[Y]$ and $\multi[Z]$ be simple multicurves whose union is transverse, and such that each pair has minimal crossings.  By Remark \ref{rem: support}, there are simple multicurves $\multi[T]_i$ such that $Supp(\multi[Y]\cdot\multi[Z])=\{\multi[T]_i\}$ and each of the $\multi[T]_i$ are identical to $\multi[Y]\cdot\multi[Z]$ outside small neighborhoods of the intersections between $\multi[Y]$ and $\multi[Z]$.  Then the number of crossings between $\multi$ and any $\multi[T]_i$ is $\mu(\multi,\multi[Y])+\mu(\multi,\multi[Z])$, and so
\[ \mu(\multi,\multi[T]_i)\leq \mu(\multi,\multi[Y])+\mu(\multi,\multi[Z])\]
Then $\mu([\multi],[\multi[Y]][\multi[Z]])$ is the maximum of $\mu(\multi,\multi[T]_i)$ over all $i$, so it also satisfies the inequality.

Since any three simple multicurves are homotopic to such a triple by Lemma \ref{lemma: simul}, the inequality is true for arbitrary simple multicurves.  The general form of the inequality follows directly.
%
%If $\multi[Y]$ and $\multi[Z]$ do not cross, then the support of $\multi[Y]\cdot\multi[Z]$ consist of the single simple multicurve $\multi[Y]\cup\multi[Z]$.  Since the number of crossings between $\multi$ and $\multi[Y]\cup\multi[Z]$ is already minimal,
%\[ \mu(\multi,\multi[Y]\cup\multi[Z]) = \mu(\multi,\multi[Y])+\mu(\multi,\multi[Z])\]
%As before, this extends to the general case directly.
\end{proof}
Lemma \ref{lemma: reducing} will give a non-trivial example where inequality (3) is strict.
\begin{rem}
Any element $x\in \Sk_q(\S)$ gives an ascending filtration on $\Sk_q(\S)$, by
\[ \mathcal{F}_{x,i}(\Sk_q(\S)):=\{y\in \Sk_q(\S)\,|\, \mu(x,y)\leq i\}\]
%This will not be used in the sequel.
\end{rem}

%\subsection{Extra structure on $\Multi$} The distinguished basis of simple multicurves allows for

%\subsubsection{The crossing pairing}

%For any two simple multicurves $\multi$ and $\multi[Y]$, let $\mu(\multi,\multi[Y])$ denote the minimum number of crossings among all pairs $(\multi',\multi[Y]')$ homotopic to $(\multi,\multi[Y])$.  %If $\mu(\multi,\multi[Y])=0$, then there is a simple multicurve $\multi\cup \multi[Y]$ which is unique up to homotopy.

%\subsubsection{$q$-initial terms}

%By Lemma \ref{lemma: basis}, any element $x$ in $\Sk_q(\S)$ can be uniquely expressed
%\[ x=\sum_{i\in \Z/2}q^{i}x_{i}\]
%where $x_i$ is a $\Z$-linear combination of simple multicurves, with all but finitely many $x_i$ zero.  Define the \textbf{$q$-initial} part $in_q(x)$ to be the non-zero $x_i$ with largest $i$; and define $in_q(0)=0$.

%\section{Multiplication by simple curves}\label{section: curve}
%
%In this section, we explore the algebraic behavior of simple curves in $\Sk_q(\S)$.%, regarded as a special case of a simple multicurve.

\subsection{Cancelling simple arcs}  %The first result will be that $[\curve]$ for $\curve$ a simple curve is not a zero divisor in $\Sk_q(\S)$.

%%%%%

\def\init{\text{in}}
\def\R{\widehat{\multi[R]}}

A useful algebraic lemma is the observation that simple arcs are not zero-divisors in $\Sk_q(\S)$; this will eventually be used to show that all non-zero elements in $\Sk_1(\S)$ are not zero-divisors (Corollary \ref{coro: domain}).  This observation follows by associating an \emph{initial multicurve} to every non-zero element of $\Sk_q(\S)$, and then showing that multiplication by a simple curve induces an injective map on initial multicurves.

Fix an arbitrary total ordering $\preceq$ on the set of simple curves in $\S$.\footnote{This will be used to break ties for defining initial terms, and will not meaningfully affect the outcome.}  We may extend this to a total ordering $\preceq$ on the set $\Multi(\S)$ of simple multicurves in $\S$ with the following rules.\footnote{This is analogous to a \emph{graded lexicographic order} among monomials in a polynomial ring.}
\begin{itemize}
	\item If two multicurves have a different number of curves, then the multicurve with more curves is larger.
	\item If two multicurves have the same number of curves, then the multicurve with the $\preceq$-largest curve not possessed by the other multicurve is larger.
\end{itemize}

By Lemma \ref{lemma: basis}, any element $x$ in $\Sk_q(\S)$ can be uniquely expressed as
\[ x=\sum_{i\in \Z/2}q^{i}x_{i}\]
where $x_i$ is a $\Z$-linear combination of simple multicurves, with all but finitely many $x_i$ zero.  Define the \textbf{initial} multicurve $\init(x)$ of $x$ to be the $\preceq$-largest simple multicurve in $Supp(x_i)$, where $x_i$ is the non-zero term in $x$ with the largest $i$.

%By Lemma \ref{lemma: basis}, any element $x$ in $\Sk_q(\S)$ can be uniquely expressed as
%\[ x=\sum_{\multi[Y]\in Supp(y)} \left(\sum_{i\in \Z/2}\lambda_{i,\multi[Y]}q^i\right)[\multi[Y]]\]
%where each $\lambda_{\multi[Y]}\in \mathbb{Z}\smallsetminus \{0\}$.

\begin{lemma}\label{lemma: ugh}
Let $\curve$ be a simple arcs in $\S$.  Then the map
\[ \multi[Y] \mapsto \init([\curve][\multi[Y]]) \]
is injection from $\Multi$ to itself.
\end{lemma}
\noindent In fact, this map is given by choosing the `positive smoothing' of every crossing in $\curve\cdot \multi[Y]$; consequently, the above map does not depend on $\preceq$.  The proof of this lemma may be found in Appendix \ref{app: lemma}.

The lemma has the following algebraic consequence.
\begin{lemma}\label{lemma: nonzero}
If $\curve$ is a simple arc, then $[\curve]$ is a not a zero divisor in $\Sk_q(\S)$.
\end{lemma}
\begin{proof}
Let $y\in \Sk_q(\S)$ be such that $[\curve]y=0$.  Write %Using Lemma \ref{lemma: basis}, there is a subset $I\subset \Multi$ such that
\[ y = \sum_{\multi[Y]\in Supp(y)} \lambda_{\multi[Y]}[\multi[Y]]\]
for $\lambda_{\multi[Y]}$ non-zero in $\Zq$.  Then
\[ 0=[\curve]y = \sum_{\multi[Y]\in Supp(y)} \lambda_{\multi[Y]}[x][\multi[Y]]
 = \sum_{\multi[Y]\in I} \lambda_{\multi[Y]}(q^{i_{\multi[Y]}}[\gamma_{\curve}(\multi[Y])] +\text{lower order terms in }q)\]
Let $i=\max_I(\deg_q(\lambda_{\multi[Y]})+i_{\multi[Y]})$, the maximal power of $q$ appearing above.
\[ 0=in_q([\curve]y) = \sum_{\substack{\multi[Y]\in Supp(y) \\ \deg(\lambda_{\multi[Y]})+i_{\multi[Y]}=i}}
in_q(\lambda_{\multi[Y]})[\gamma_{\curve}(\multi[Y])]
\]
Since the map $\gamma_{\curve}$ is an injection and $\Multi$ is a basis, the elements $[\gamma_{\curve}(\multi[Y])]$ are independent over $\Zq$.  Since $in_q(\lambda_{\multi[Y]})$ cannot be zero, the support $Supp(y)$ must be empty, and so $y=0$.

Then $[\curve]$ is not a left zero divisor.  By applying the bar involution, $[\curve]^\dag=[\curve]$ is not a right zero divisor.
\end{proof}

\subsection{Reducing crossings}\label{section: reducing}

Next, we consider how multiplication by a simple curve affects crossing number.  We observe that multiplication by the class of a simple arc $\curve$ reduces the crossing number with respect to that curve.
\begin{lemma}\label{lemma: reducing}
If $\curve$ is a simple arc, then for all $y\in \Sk_q(\S)$ such that $\mu([\curve],y)>0$,
\[ \mu([\curve],[\curve]y)\leq\mu([\curve],y)-1\]
%If $\curve$ is a simple loop, then for all $y\in \Sk_q(\S)$,
%\[ \mu([\curve],[\curve]y)=\mu([\curve],y)\]
\end{lemma}
\begin{proof}
First consider the case when $y$ is a simple multicurve $\multi[Y]$ so that $\curve\cdot \multi[Y]$ is transverse, and $\curve\cdot\multi[Y]$ has $\mu(\curve,\multi[Y])$ crossings (the minimal number, up to homotopy). %As in the proof of Lemma \ref{lemma: injection}, choose a tubular neighborhood of $\curve$ small enough that each component of the restriction of $\multi[Y]$ intersects $\curve$ exactly once.
Consider the set $I$ of multicurves which can be obtained by applying some combination of the following two local relations to each crossing in $\curve\cdot \multi[Y]$.
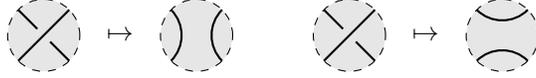
\begin{center}
\begin{tikzpicture}[scale=.8]
\begin{scope}[xshift=-1in]
\begin{scope}[xshift=-.5in,scale=.15]
    \draw[fill=black!10,dashed] (0,0) circle (4);
    \draw[thick] (-2.83,-2.83) to (2.83,2.83);
    \draw[thick] (-2.83,2.83) to (-.71,.71);
    \draw[thick] (.71,-.71) to (2.83,-2.83);
\end{scope}
\node (=) at (0,0) {$\mapsto$};
\begin{scope}[xshift=.5in,scale=.15]
    \draw[fill=black!10,dashed] (0,0) circle (4);
    \draw[thick] (-2.83,-2.83) to [out=45,in=-45] (-2.83,2.83);
    \draw[thick] (2.83,-2.83) to [out=135,in=-135] (2.83,2.83);
\end{scope}
\end{scope}

\begin{scope}[xshift=1in]
\begin{scope}[xshift=-.5in,scale=.15]
    \draw[fill=black!10,dashed] (0,0) circle (4);
    \draw[thick] (-2.83,-2.83) to (2.83,2.83);
    \draw[thick] (-2.83,2.83) to (-.71,.71);
    \draw[thick] (.71,-.71) to (2.83,-2.83);
\end{scope}
\node (=') at (0,0) {$\mapsto$};
\begin{scope}[xshift=.5in,scale=.15]
    \draw[fill=black!10,dashed] (0,0) circle (4);
    \draw[thick] (-2.83,-2.83) to [out=45,in=135] (2.83,-2.83);
    \draw[thick] (-2.83,2.83) to [out=-45,in=-135] (2.83,2.83);
\end{scope}
\end{scope}
\end{tikzpicture}
\end{center}
Since the simple multicurves in the support $Supp([\curve][\multi[Y]])$ come from applying the Kauffman skein relation to the crossings in $\curve\cdot \multi[Y]$, we have $Supp([\curve][\multi[Y]])\subset I$.

Consider a simple multicurve $\multi[Z]\in I$.  For two adjacent crossings in $\curve\cdot \multi[Y]$ along $\curve$, there are two local possibilities for $\multi[Z]$, up to reflection across $\curve$.
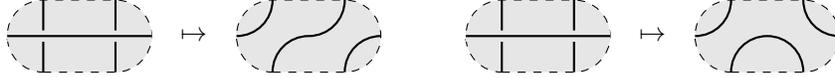
\begin{center}
\begin{tikzpicture}[scale=.8]
\begin{scope}[xshift=-1.5in]
\begin{scope}[xshift=-.75in,scale=.15]
    \draw[fill=black!10,dashed] (-4,-4)  to (4,-4) arc (-90:90:4) to (-4,4) arc (90:270:4);
    \draw[thick] (-4,4) to (-4,-4);
    \draw[thick] (4,4) to (4,-4);
    \draw[outline] (-8,0) to (8,0);
    \draw[thick] (-8,0) to (8,0);
\end{scope}
\node (=) at (0,0) {$\mapsto$};
\begin{scope}[xshift=.75in,scale=.15]
    \draw[fill=black!10,dashed] (-4,-4)  to (4,-4) arc (-90:90:4) to (-4,4) arc (90:270:4);
    \draw[thick] (-8,0) to [out=0,in=270] (-4,4);
    \draw[thick] (-4,-4) to [out=90,in=180] (0,0) to [out=0,in=270] (4,4);
    \draw[thick] (4,-4) to [out=90,in=180] (8,0);
\end{scope}
\end{scope}

\begin{scope}[xshift=1.5in]
\begin{scope}[xshift=-.75in,scale=.15]
    \draw[fill=black!10,dashed] (-4,-4)  to (4,-4) arc (-90:90:4) to (-4,4) arc (90:270:4);
    \draw[thick] (-4,4) to (-4,-4);
    \draw[thick] (4,4) to (4,-4);
    \draw[outline] (-8,0) to (8,0);
    \draw[thick] (-8,0) to (8,0);
\end{scope}
\node (=) at (0,0) {$\mapsto$};
\begin{scope}[xshift=.75in,scale=.15]
    \draw[fill=black!10,dashed] (-4,-4)  to (4,-4) arc (-90:90:4) to (-4,4) arc (90:270:4);
    \draw[thick] (-8,0) to [out=0,in=270] (-4,4);
    \draw[thick] (-4,-4) to [out=90,in=180] (0,0) to [out=0,in=90] (4,-4);
    \draw[thick] (4,4) to [out=270,in=180] (8,0);
\end{scope}
\end{scope}
\end{tikzpicture}
\end{center}
In this local picture, the first case is homotopic to a multicurve with crossing $\curve$ once, and the second is homotopic to a multicurve which does not cross $\curve$.

Between a crossing in $\curve\cdot\multi[Y]$ and an end of $\curve$, there is one local possibility for $\multi[Z]$, up to reflection across $\curve$.
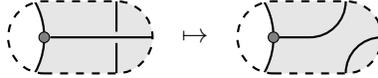
\begin{center}
\begin{tikzpicture}[scale=.8]
\begin{scope}[xshift=-1.5in]
\begin{scope}[xshift=-.75in,scale=.15]
	\begin{scope}
    \clip (-4,-4)  to (4,-4) arc (-90:90:4) to (-4,4) arc (90:270:4);
    \draw[thick, fill=black!10] (-12,-8) to [out=0,in=270] (-4,0) to [out=90,in=0] (-12,8) to [line to] (12,8) to (12,-8) to (-12,-8);
		\node[marked] (1) at (-4,0) {};
    \draw[thick] (4,4) to (4,-4);
    \draw[outline] (1) to (8,0);
    \draw[thick] (1) to (8,0);
	\end{scope}
	\draw[thick, dashed] (-4,-4)  to (4,-4) arc (-90:90:4) to (-4,4) arc (90:270:4);
\end{scope}
\node (=) at (0,0) {$\mapsto$};
\begin{scope}[xshift=.75in,scale=.15]
	\begin{scope}
    \clip (-4,-4)  to (4,-4) arc (-90:90:4) to (-4,4) arc (90:270:4);
    \draw[thick, fill=black!10] (-12,-8) to [out=0,in=270] (-4,0) to [out=90,in=0] (-12,8) to [line to] (12,8) to (12,-8) to (-12,-8);
		\node[marked] (1) at (-4,0) {};
    \draw[thick] (4,-4) to [out=90,in=180] (8,0);
    \draw[thick] (1) to (0,0) to [out=0,in=270] (4,4);
	\end{scope}
	\draw[thick, dashed] (-4,-4)  to (4,-4) arc (-90:90:4) to (-4,4) arc (90:270:4);
\end{scope}
\end{scope}
\end{tikzpicture}
\end{center}
This local picture is homotopic to one which does not cross $\curve$.

Then $\multi[Z]$ is homotopic to a simple multicurve $\multi[Z]'$, such that $\curve\cdot \multi[Z]'$ is transverse and the crossings in $\curve\cdot \multi[Z]'$ occur at most once between each pair of adjacent crossings in $\curve\cdot \multi[Y]$ (and $\curve\cdot\multi[Z]'$ has no other crossings).  Therefore, $\curve\cdot \multi[Z]'$ has strictly fewer crossings than $\curve \cdot \multi[Y]$.  Since the latter already has $\mu(\curve,\multi[Y])$ crossings,
\[ \mu([\curve],[\multi[Z]'])\leq\mu(\curve,\multi[Y])-1\]
Because $Supp([\curve][\multi[Y]])\subset I$,
\[ \mu([\curve],[\curve][\multi[Y]])\leq\mu([\curve],[\multi[Y]])-1\]
The general form of the lemma follows from this case.
\end{proof}
\begin{rem}
Multiplication by simple loops does not reduce crossing number.
\end{rem}

Lemma \ref{lemma: reducing} is useful, because multiplication by a sufficiently high power of $[\curve]$ will make an element $y\in \Sk_q(\S)$ have zero crossing number with $[\curve]$.
\begin{coro}\label{coro: reducing}
If $\curve$ is a simple arc, then for all $y\in \Sk_q(\S)$,
\[ \mu([\curve],[\curve]^{\mu([x],y)}y)=0\]
\end{coro}
\begin{proof}
By iterating Lemma \ref{lemma: reducing}, if $i\leq\mu([\curve],y)$
\[ \mu([\curve],[\curve]^iy)\leq \mu([\curve],y)-i\]
In particular, $\mu([\curve],[\curve]^{\mu([x],y)}y)\leq0$, so it is zero.
\end{proof}

%\begin{rem}
%This lemma is not true for simple loops.  For $\S$ the torus without marked points, and $\curve, \curve[y]$ a pair of simple loops which intersect once, then $\mu([\curve],[\curve][\curve[y]])=\mu([\curve],[\curve[y]])=1$.
%\end{rem}

%For a transverse multicurve $\multi$, let $\nu(\multi)$ denote the number of interior intersections in $\multi$.  If $\multi$ and $\multi[Y]$ are two transverse multicurves with $\multi\cup \multi[Y]$ transverse, then define
%\[ \nu(\multi,\multi[Y]):=\nu(\multi\cup\multi[Y])-\nu(\multi)-\nu(\multi[Y]),\]
%which counts the number of intersections between the two multicurves.
%
%Using the Kauffman skein relation, multiplying a link by a simple arc reduces the number of intersections with that arc.
%\begin{lemma}
%Let $\curve$ be a simple arc, and let $\link[Y]$ be a link with $\mu(\curve,\link[Y])>0$.  Then there are links $\link[Z]_1,\link[Z]_2$ such that
%\[ [\curve][\link[Y]]=q[\link[Z]_1]+q^{-1}[\link[Z]_2]\]
%and $\mu(\curve,\link[Z]_i)<\mu(\curve,\link[Y])$ for all $i$.
%\end{lemma}
%
%\begin{lemma}
%Let $\curve$ be a simple arc, and let $\link[Y]$ be a link with underlying multicurve $\multi[Y]$, so that $\link\cup\multi[Y]$ is transverse.  Then there are links $\link[Z]_i$ and $\lambda_i\in \Zq$ such that
%\[ [\curve]^{\nu(\curve,\multi[Y])}[\link[Y]]=\sum_i \lambda_i[\link[Z]_i]\]
%and $\nu(\curve,\link[Z]_i)=0$ for all $i$.
%\end{lemma}

\section{The localized skein algebra $\Sk_q^o(\S)$}\label{section: localskein}

The connection from $\Sk_q(\S)$ to cluster algebras will be through the localization $\Sk_q^o(\S)$ of $\Sk_q(\S)$ at the set of boundary curves.  %This section defines this algebra and collects basic properties.

%\subsection{Ore localizations}
%
%We briefly recall Ore localizations and Ore's Theorem.  Given a ring $A$ and a subset $S$, the \emph{localization} of $A$ at $S$ is the universal ring $A'$ with a map $\iota:A\rightarrow A'$ such that $\iota(S)$ consists of units.  A localization is \emph{Ore} if every element of $A'$ can be written as $\iota(a_1)\iota(s_1)^{-1}$ and as $\iota(s_2)^{-1}\iota(a_2)$, for some $a_1,a_2\in A$ and $s_1,s_2\in S$.
%
%\begin{lemma}[Ore's Theorem]
%Let $A$ be an algebra, and let $S\subset A$ be a subset.  Assume, for all $a\in A$ and $s_1s_2\in S$,
%\begin{itemize}
%\item (\emph{multiplicatively closed, up to a central unit}) $\exists \lambda\in Z(\A)^\times$ such that $\lambda S_1S_2\in S$,
%\item (\emph{right and left permutable}) $aS\cap s_1A$ and $Sa\cap As_1$ are non-empty, and
%\item (\emph{non-zero divisors}) $s_1a=0$ or $as_1=0$ implies $s_1$ or $a$ is $0$.
%\end{itemize}
%Then $A$ embeds in its Ore localization.
%\end{lemma}
%Ore localizations will be denoted $A[S^{-1}]$, and we identify elements in $A$ with their image in $A[S^{-1}]$.
%
%\begin{rem}
%Localizing sets $S$ are typically required to be multiplicatively closed.  However, the broader generality here will make
%\end{rem}

\subsection{The localized skein algebra}

A \textbf{boundary curve} is a simple curve which is homotopic to a subset of the boundary $\partial\S$.  A boundary curve is either an arc connecting adjacent marked points on the same boundary component, or a loop homotopic to an unmarked boundary component.  The set of boundary curves is finite; it is the number of marked points plus the number of unmarked boundary components.
\begin{defn}
The localization of $\Sk_q(\S)$ at the set of boundary curves is the \textbf{localized skein algebra} of $\S$, denoted $\Sk_q^o(\S)$.
\end{defn}

For the moment, $\Sk_q^o(\S)$ is defined as an abstract localization; that is, the universal algebra with a map from $\Sk_q(\S)$ such that every boundary curve is sent to a unit.  This is improved with the following proposition.

\begin{prop}
The algebra $\Sk_q^o(\S)$ is an injective Ore localization of $\Sk_q(\S)$.%; that is, every element of $\Sk_q^o(\S)$ can be written as $xy^{-1}$, for $x\in \Sk_q(\S)$ and $y$ a product of boundary curves.
\end{prop}
\begin{proof}
Given a boundary curve $\curve$ and a link $\link[Y]$, there are homotopic links $\curve'$ and $\link[Y]'$ which only intersect at the boundary.  Then, there is some $\lambda\in \Z$ such that
$ [\curve][\link[Y]] = q^{\frac{\lambda}{2}}[\link[Y]][\curve]$.
Therefore, the set of all products of boundary curves is right and left permutable.  By Ore's theorem, the localization is injective Ore.
\end{proof}
%\noindent As a consequence, every element of $\Sk_q^o(\S)$ can be written as $x[\multi[Y]]^{-1}$, where $x\in \Sk_q(\S)$ and $Y$ is a multicurve consisting of boundary curves.
%
%The following proposition uses a result proven later in the paper (Lemma \ref{lemma: cancelling}).
%
%\begin{prop}
%The localization map $\Sk_q(\S)\rightarrow \Sk_q^o(\S)$ is an inclusion.
%\end{prop}
%\begin{proof}
%By Lemma \ref{lemma: cancel}, for any boundary curve $\curve$, the element $[\curve]$ is not a zero divisor.  Any product of boundary curves is similarly not a zero-divisor.  By \cite{?}, the Ore localization is an inclusion.
%\end{proof}
%
%\begin{prop}
%If $\curve$ is a boundary curve and $\link$ is a link, then $\exists \lambda\in \mathbb{Z}$ with
%\[ [\curve][\link] = q^{\lambda}[\link][\curve]\]
%\end{prop}
%\begin{proof}
%There are homotopic links $\curve'\subset\partial\S$ and $\link'$ which only intersect at the endpoints.  By the boundary skein relation, $[\curve'][\link']$ and $[\link'][\curve']$ are related by powers of $q$.
%\end{proof}
\noindent We will identify $\Sk_q(\S)$ with its image in $\Sk_q^o(\S)$.

The extra structures on $\Sk_q(\S)$ extend to $\Sk_q^o(\S)$.  The bar involution $\dag$ on $\Sk_q(\S)$ (see Section \ref{section: bar}) extends to an involutive ring antiautomorphism on $\Sk_q^o(\S)$, by $(xy^{-1})^\dag=(y^\dag)^{-1}x^\dag$.  %This is well-defined because the set of multicurves consisting of boundary curves is $\dag$-fixed.
The endpoint $\E$-grading on $\Sk_q^o(\S)$ (see Section \ref{section: grading}) extends to an endpoint $\E$-grading on $\Sk_q^o(\S)$, with $\deg(xy^{-1})= \deg(x)-\deg(y)$.

%\begin{rem}
%Our main motivation for considering this localization is pragmatic; $\Sk_q^o(\S)$ is the algebra which is a cluster algebra.  However, one
%\end{rem}

\subsection{The basis of weighted simple multicurves}

The $\Zq$-basis of $\Sk_q(\S)$ by the set $\Multi$ of simple multicurves can be extended to a $\Zq$-basis of $\Sk_q^o(\S)$ in the following (somewhat artificial) way.

Define a \textbf{weighted simple multicurve} $\multi$ to be a simple multicurve $\multi$, together with an integer `weight' $w_{\curve}$ for each $\curve\in\multi$.  Two weighted simple multicurves $\multi$ and $\multi[Y]$ are \textbf{equivalent} if, for each simple curve $\curve$ in $\S$, the sum of the weights on curves in $\multi$ homotopic to $\curve$ is the same as the sum of the weights on curves in $\multi[Y]$ homotopic to $\curve$.  Intuitively, a curve $\curve$ of weight $w_{\curve}\in\mathbb{N}$ is equivalent to $w_{\curve}$-many copies of $\curve$.

Let $\Multi^o$ be the set of equivalence classes of weighted simple multicurves with positive weights on non-boundary curves (and arbitrary integral weights on boundary curves).  Given $\multi\in \Multi^o$, define an element $[\multi]\in \Sk_q^o(\S)$ by
\[ [\multi] :=q^{\frac{\lambda}{2}} \prod_{\curve\in \multi}[\curve]^{w_{\curve}}\]
where $q^{\frac{\lambda}{2}}$ is the unique $q$-power such that $[\multi]^\dag=[\multi]$.  %When $\multi$ has all weights $1$, $[\multi]$ is the element of $\Sk_q(\S)$ corresponding to the underlying

The $\Zq$-basis of $\Sk_q(\S)$ then extends to a $\Zq$-basis of $\Sk_q^o(\S)$.
\begin{prop}\label{prop: basiso}
Under $\multi\rightarrow [\multi]$, the set $\Multi^o$ maps to a $\Zq$-basis of $\Sk_q^o(\S)$.
\end{prop}
\begin{proof}
Any element of $\Sk_q^o(\S)$ can be written as $xy^{-1}$, with $y$ a product of boundary curves.  %The element $x$ can be written as a $\Zq$-linear combination of simple multicurves $\multi_i$.
Then $y=q^j[\multi[Y]]$,  where $\multi[Y]$ is some simple multicurve of boundary arcs.
The element $x$ can be written as
\[ x = \sum_i \lambda_i[\multi_i],\]
a $\Zq$-linear combination of simple multicurves $\multi_i$.  Then
\[ xy^{-1} = \sum_i q^{-j}\lambda_i [\multi_i][\multi[Y]]^{-1}\]
It is always possible to add boundary curves to any simple multicurve, without violating simplicity.  Let $\multi_i'$ be the weighted simply multicurve which contains all the curves in $\multi_i$ and $\multi[Y]$, with each weight counting how many times a given curve appeared in $\multi_i$ minus how many times it appeared in $\multi[Y]$.  Then there are $\lambda_i'$ such that
\[ xy^{-1} = \sum_i \lambda_i' [\multi_i']\]
and so $\Multi^o$ spans $\Sk_q^o(\S)$ over $\Zq$.

To show this is a $\Zq$-basis, consider any relation between the weighted simple multicurves.  Denominators may be cleared by multiplying by a sufficiently large multicurve $[\multi[Z]]$ in the boundary curves, giving a relation between weighted simple multicurves with positive weights.  This gives a relation between simple multicurves in $\Sk_q(\S)$, which must be the trivial relation (Lemma \ref{lemma: basis}).  Since $[\multi[Z]]$ is not a zero divisor (Lemma \ref{lemma: nonzero}), the original relation was also trivial.
\end{proof}
\noindent This basis is fixed by the bar involution, and is homogeneous for the $\E$-grading.

\section{Triangulations}
This section explores the extra structure on $\Sk_q(\S)$ coming from a triangulation of $\S$.  Since triangulations only exist when there are enough marked points, this demonstrates an advantage over the unmarked case.

\subsection{Triangulations}
A marked surface $\S$ is \textbf{triangulable} if...
\begin{itemize}
\item $\partial\S$ is not empty,
\item each boundary component contains a marked point, and
\item no connected component of $\S$ is a disc with one or two marked points.\footnote{This last condition is unnecessary for subsequent results on skein and cluster algebras.}
\end{itemize}
A \textbf{triangulation}\footnote{This is sometimes called an \emph{ideal triangulation}, to distinguish from triangulations which are allowed to have vertices away from marked points.} of a triangulable $\S$ is a simple multicurve $\Delta$ such that...
\begin{itemize}
\item no two curves in $\Delta$ are homotopic,
\item $\Delta$ is maximal amongst simple multicurves with the first property, and
\item $\Delta$ consists entirely of arcs.
\end{itemize}
A triangulation of $\Delta$ is a collection of arcs which cut $\S$ into a union of triangles.%\footnote{In this article, the marked points are on the boundary.  Therefore, the three edges of every triangle in a triangulation are distinct, and so the obstacles which arise from 'self-folded triangles' in \cite{GSV05} and \cite{FST08} do not arise.}
\begin{rem}
If only the first two conditions hold, $\Delta$ is called a \emph{maximal multicurve}.
\end{rem}
%\begin{rem}
%These definitions
%
%The two degenerate cases are the disc with one or two marked points on the boundary.  Each of these has a unique `triangulation' (up to equivalence) corresponding to no arcs or one arc, respectively.  These cases do not need to be admitted, but the subsequent theory is extends to them.
%\end{rem}

%We collect some basic facts about triangulations.

If $\curve\in \Delta$ is a non-boundary arc, then is it an edge in two distinct triangles in $\S-\Delta$.\footnote{This is a consequence of requiring that marked points are on the boundary.  If there are interior marked points, there can be 'self-folded triangles': triangles without all edges distinct.}  %Let $\curve_{k_1},\curve_{k_2},\curve_{k_3},\curve_{k_4}$ be the other arcs, as in Figure \ref{fig: adjacent}.  These arcs need not be distinct nor have distinct endpoints, despite how they are drawn, but the do bound a well-defined quadralateral.
There is a unique other curve $\curve'$ such that $(\Delta-\curve)\cup\curve'$ is also a triangulation; both the curve and the resulting triangulation may be called the \textbf{flip} of $\curve$ in $\Delta$.

%then there is exactly one other triangulation $\Delta'$ which contains $\Delta\backslash \curve$, called the \textbf{flip} of $\Delta$ at $\alpha$.  MAKE A FIGURE!
%The curve $\curve_j$   Note that
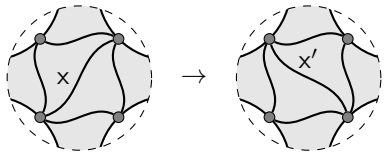
\begin{figure}[h]
\begin{tikzpicture}[scale=.6]
\begin{scope}[scale=.4]
	\begin{scope}
	\clip (0,0) circle (4);
	\draw[fill=black!10,thick] (-5,-1) arc (90:0:4) to (1,-5) arc (180:90:4) to (5,1) arc (270:180:4) to (-1,5) arc (0:-90:4);
	\node[marked] (1) at (-2.17,2.17) {};
	\node[marked] (2) at (2.17,2.17) {};
	\node[marked] (3) at (2.17,-2.17) {};
	\node[marked] (4) at (-2.17,-2.17) {};
	\draw[thick] (4) to [relative,out=-30,in=150] node[left] {$\curve$} (2);
	\draw[thick] (1) to [relative,out=-30,in=150]  (2);
	\draw[thick] (2) to [relative,out=-30,in=150]  (3);
	\draw[thick] (3) to [relative,out=-30,in=150]  (4);
	\draw[thick] (4) to [relative,out=-30,in=150]  (1);
	\end{scope}
	\draw[dashed] (0,0) circle (4);
\end{scope}
\node (a) at (1in,0) {$\rightarrow$};
\begin{scope}[xshift=2in,scale=.4]
	\begin{scope}
	\clip (0,0) circle (4);
	\draw[fill=black!10,thick] (-5,-1) arc (90:0:4) to (1,-5) arc (180:90:4) to (5,1) arc (270:180:4) to (-1,5) arc (0:-90:4);
	\node[marked] (1) at (-2.17,2.17) {};
	\node[marked] (2) at (2.17,2.17) {};
	\node[marked] (3) at (2.17,-2.17) {};
	\node[marked] (4) at (-2.17,-2.17) {};
	\draw[thick] (1) to [relative,out=-30,in=150] node[above] {$\curve'$} (3);
	\draw[thick] (1) to [relative,out=-30,in=150]  (2);
	\draw[thick] (2) to [relative,out=-30,in=150]  (3);
	\draw[thick] (3) to [relative,out=-30,in=150]  (4);
	\draw[thick] (4) to [relative,out=-30,in=150]  (1);
	\end{scope}
	\draw[dashed] (0,0) circle (4);
\end{scope}
\end{tikzpicture}
\caption{Flipping an arc}
\label{fig: flip}
\end{figure}

%\begin{rem}
%If marked points on the interior of $\S$ were allowed, then it is possible to have a triangulation with an edge which is two sides of the same triangle.  Such an edge has no flip.  In \cite{GSV05}, this was identified as an obstacle for the cluster algebra theory (Section 7) and resolved with the introduction of 'tagged arcs' in \cite{FST08}.  
%\end{rem}

\begin{prop} \label{prop: triang}
Let $\S$ be a triangulable surface with marked points $\mathcal{M}$.
\begin{enumerate}
\item Triangulations of $\S$ always exist.\footnote{That is, the definitions `triangulable' and `triangulation' behave as expected.}
\item Any simple multicurve of distinct arcs is contained in some triangulation.
\item An arc is in every triangulation if and only if it is a boundary arc.
\item Every triangulation has  $|\Delta| = 6g+3h+2|\M|-6$ arcs, where $g$ is the genus and $h$ is the number of boundary components of $\S$.
%\item Every compatible collection of $N_\Sigma-1$ arcs is contained in either one triangulation (if the missing arc is a boundary arc) or two triangulations.  Two triangulations with $N-1$ arcs in common are said to be \emph{related by a flip}.
\item Every pair of triangulations are related by a sequence of flips.
\end{enumerate}
\end{prop}
\begin{proof}Our triangulations differ from those in \cite{FST08}, in that they forbid boundary arcs.  However, their results can still be applied with appropriate modification.

$(1)$. \cite[Lemma 2.13]{FST08}.

$(2)$. This follows from the given definition of `triangulation'.

$(3)$. A boundary arc is in every triangulation because it has no crossings with any other arcs, and so it can always be added to a simple multicurve without breaking simplicity.  For any non-boundary arc $\curve$, find a triangulation containing $\curve$ and flip $\curve$, to get a new triangulation which does not contain $\curve$.

$(4)$. By \cite[Proposition 2.10]{FST08}, there are $|\Delta|-|\M|$ non-boundary arcs in every triangulation.  Since there are always $|\M|$ boundary arcs, the claim follows.

$(5)$. \cite[Proposition 3.8]{FST08}.
\end{proof}
It will frequently be useful to {index} the arcs in a triangulation with numbers $1,2,...,|\Delta|$; this will often be done without comment.  Then we can write
\[ \Delta=\{ \curve_1,\curve_2,...,\curve_{|\Delta|}\}\]
Let $\Z^\Delta$ denote the rank ${|\Delta|}$ lattice generated by the elements of $\Delta$.  For an indexed triangulation, $\Z^\Delta\simeq \Z^{|\Delta|}$, and we identify elements $\alpha$ of $\Z^\Delta$ with ${|\Delta|}$-tuples of integers $(\alpha_1,\alpha_2,...,\alpha_{|\Delta|})$.

\subsection{The orientation matrix and the signed adjacency matrix}\label{section: skewmatrix}

An \textbf{end} of an arc $\curve$ will be a strand of $\curve$ in a small neighborhood of an endpoint.  For an arc $\curve$, let $\partial_1(\curve)$ and $\partial_2(\curve)$ denote the two ends of $\curve$ (for an arbitrary numbering).

For two simple curves $\curve,\curve[y]$ with $\curve\cup\curve[y]$ simple, define
\[ \Lambda_{\curve,\curve[y]} = \sum_{i,j\in \{1,2\}} \left\{
\begin{array}{c|c}
0 & \text{if $\partial_i(\curve)$ and $\partial_j(\curve[y])$ have different endpoints} \\
1 & \text{if $\partial_i(\curve)$ is clockwise to $\partial_j(\curve[y])$} \\
-1 & \text{if $\partial_i(\curve[y])$ is clockwise to $\partial_j(\curve)$} \\
\end{array}\right\}\]
This measures the power of $q$ which relates the superposition $[\curve][\curve[y]]=[\curve\cdot\curve[y]]$ to the (simultaneous) simple multicurve $[\curve\cup\curve[y]]$.
\begin{prop}
Let $\curve$ and $\curve[y]$ be simple curves with $\multi=\curve\cup\curve[y]$ a simple multicurve.
\[ [\curve][\curve[y]]=q^{\frac{1}{2}\Lambda_{\curve,\curve[y]}}[\multi] = q^{\Lambda_{\curve,\curve[y]}}[\curve[y]][\curve]\]
\end{prop}
\begin{proof}
This is a restatement of the boundary skein relation (Figure \ref{fig: skein}).
\end{proof}

Given an indexed triangulation $\Delta=\{\curve_1,\curve_2,...,\curve_{|\Delta|}\}$, define a skew-symmetric ${|\Delta|}\times {|\Delta|}$-matrix $\Lambda^\Delta$, called the \textbf{orientation matrix of $\Delta$}, by
\[ \Lambda^\Delta_{ij} := \Lambda_{\curve_i,\curve_j}\]
Finally, extend $\Lambda^\Delta$ to a skew-symmetric bilinear form $\Lambda^\Delta:\Z^\Delta\times\Z^\Delta\rightarrow \Z$ by
\[ \Lambda^\Delta(\alpha,\beta):=\langle \alpha, \Lambda^\Delta\beta\rangle = \sum_{1\leq i,j\leq {|\Delta|}}\Lambda^\Delta_{ij}\alpha_i\beta_j\]

Later on, we will also need a related matrix which measures when two ends are immediately clockwise in a triangulation.  For two simple curves $\curve,\curve[y]$ in a indexed triangulation $\Delta$, define
\[ \Q^\Delta_{\curve,\curve[y]} = \sum_{i,j\in \{1,2\}} \left\{
\begin{array}{c|c}
0 & \text{if $\partial_i(\curve)$ and $\partial_j(\curve[y])$ have different endpoints} \\
-1 & \text{if $\partial_i(\curve)$ is immediately clockwise to $\partial_j(\curve[y])$ in $\Delta$} \\
1 & \text{if $\partial_i(\curve[y])$ is immediately clockwise to $\partial_j(\curve)$ in $\Delta$} \\
\end{array}\right\}\]
Note the sign-reversal.
Define a skew-symmetric ${|\Delta|}\times {|\Delta|}$ matrix $\Q^\Delta$, called the \textbf{skew-adjacency matrix of $\Delta$}, by
\[ \Q^\Delta_{ij} := \Q^\Delta_{\curve_i,\curve_j}\]
Finally, extend $\Q^\Delta$ to a skew-symmetric bilinear form $\Q^\Delta:\Z^\Delta\times\Z^\Delta\rightarrow \Z$ by
\[ \Q^\Delta(\alpha,\beta):=\alpha^\dag \Q^\Delta\beta = \sum_{1\leq i,j\leq {|\Delta|}}\Q^\Delta_{ij}\alpha_i\beta_j\]

\subsection{Monomials in $\Delta$}

Fix a triangulation $\Delta$ of $\S$.  For $\alpha\in \N^\Delta$, let $\Delta^\alpha$ denote a simple multicurve which has $\alpha_i$-many curves homotopic to $\curve_i$, for each $i$, and no other components.  The corresponding class $[\Delta^\alpha]\in \Sk_q(\S)$ does not depend on the choice of such a multicurve.  Such an element is called a \textbf{monomial} in the triangulation $\Delta$.

Multiplication of monomials can be computed using the following proposition.
\begin{prop}\label{prop: monomialmulti}
\[ [\Delta^\alpha] = q^{-\frac{1}{2}\sum_{i<j}\Lambda^\Delta_{ij}\alpha_i\alpha_j}[\curve_1]^{\alpha_1}[\curve_2]^{\alpha_2}...[\curve_{|\Delta|}]^{\alpha_{|\Delta|}}\]
\[ [\Delta^\alpha][\Delta^\beta] = q^{\frac{1}{2}\Lambda^\Delta(\alpha,\beta)} [\Delta^{\alpha+\beta}] = q^{\Lambda^\Delta(\alpha,\beta)} [\Delta^\beta][\Delta^\alpha]\]
\end{prop}
\begin{proof}
The superposition product $[\curve_1]^{\alpha_1}[\curve_2]^{\alpha_2}...[\curve_{|\Delta|}]^{\alpha_{|\Delta|}}$ corresponds to a link $\link$ which has the same underlying multicurve as $\Delta^\alpha$; however, the ordering on $\link$ is via superposition, and the ordering on $\Delta^\alpha$ is simultaneous.  By repeatedly applying the boundary skein relation (Figure \ref{fig: skein}), one obtains the first identity.

The second identity follows from the first identity, or by direct application of the boundary skein relation.
\end{proof}

Monomials can be characterized as follows.  For any element $y\in Sk_q(\S)$, define the element $\mu_{\Delta}(y)\in \N^\Delta$ by
\[ \mu_{\Delta}(y) := (\mu([\curve_1],y),\mu([\curve_2],y),...,\mu([\curve_{|\Delta|}],y))\]
Lemma \ref{lemma: crossings} implies that
\[ \mu([\Delta^\alpha],x) = \alpha\cdot \mu_\Delta(x)\]
where the dot product uses the standard basis in $\N^\Delta$.

\begin{prop}\label{prop: monomial}
%For a simple multicurve $\multi$,
For $\multi$ a simple multicurve, $[\multi]$ is a monomial in $\Delta$ if and only if $\mu_\Delta([\multi])= 0$.
\end{prop}
\begin{proof} If $[\multi]=[\Delta^\alpha]$, then
\[ \mu([\curve_i],[\Delta^\alpha]) = \sum_{1\leq j\leq {|\Delta|}} \alpha_j\mu([\curve_i],[\curve_j])=0\]
and so $\mu_\Delta([\Delta^\alpha])=0$.

Now, assume $\mu_\Delta([\multi])=0$.  Then there is a homotopic simple multicurve $\multi'$ which does not cross any $\curve_i\in \Delta$.  Then each component of $\multi'$ is a simple curve which does not cross any $\curve_i\in \Delta$.  Because $\Delta$ is maximal, each component of $\multi'$ is homotopic to some arc in $\Delta$.  Then every component of $\multi'$ is homotopic to an arc in $\Delta$, and so $[\multi']=[\Delta^\alpha]$ for some $\alpha$.
\end{proof}

A \emph{polynomial}\footnote{Some call these `skew-polynomials', to emphasize their monomials only quasi-commute.} in $\Delta$ is a $\Zq$-linear combination of monomials, and the set of polynomials in $\Delta$ is a $\Zq$-subalgebra of $\Sk_q(\S)$ by the proposition.  By the proposition, $x\in \Sk_q(\S)$ is a polynomial in $\Delta$ if and only if $\mu_\Delta(x)=0$.
\begin{rem}
A triangulation $\Delta$ gives $\Sk_q(\S)$ an $\N^N$-filtration, where $\N^N$ has the partial order $\alpha\leq\beta$ if $\alpha_i\leq\beta_i$ for all $i$.  The filtration is
\[\mathcal{F}_{\Delta,\alpha}(\Sk_q(\S)):=\{x\in \Sk_q(\S)|\mu_{\Delta}(x)\leq\alpha\}\]  Then the subalgebra of polynomials in $\Delta$ is $\mathcal{F}_{\Delta,\mathbf{0}}(\Sk_q(\S))$.
\end{rem}

\begin{rem}
If $\Delta$ is a maximal multicurve (possibly with loops), the results of this section remain true (where $\Lambda^\Delta_{ij}: = 0$ if either $\curve_i$ or $\curve_j$ is a loop).
\end{rem}

\subsection{Laurent expressions}

In Section \ref{section: reducing}, it was shown that multiplying $y$ by a sufficiently high power of an arc $[\curve]$ had zero crossing number with $[\curve]$.  This can be directly generalized to triangulations.
%\begin{lemma}
%If $\mu([\Delta],y)=0$, then $y$ is a polynomial in $\Delta$. That is, there are $\lambda_\alpha\in \Zq$ such that
%\[ y = \sum_{\alpha}\lambda_\alpha [\Delta^\alpha]\]
%\end{lemma}
\begin{lemma}
For all $y\in Sk_q(\S)$, $\mu_\Delta([\Delta^{\mu_\Delta(y)}]y)=0$.
\end{lemma}
\begin{proof}
By Proposition \ref{prop: monomialmulti}, there is some $n\in \Z$ such that
%\[ [\Delta^{\mu_\Delta(y)}] = q^{\frac{1}{2}\lambda} [\curve_1]^{\mu(\curve_1,y)}[\curve_2]^{\mu(\curve_2,y)}...[\curve_N]^{\mu(\curve_N,y)}\]
%Then for any $1\leq i\leq N$, there is some $\lambda'\in \Z$ and some $\alpha\in \N^\Delta$ such that
\[ [\Delta^{\mu_\Delta(y)}] = q^{\frac{1}{2}n} [\Delta^\alpha][\curve_i]^{\mu(\curve_i,y)}\]
By Lemma \ref{lemma: crossings},
\[ \mu([\curve_i],[\Delta^{\mu_\Delta(y)}]y) \leq \mu([\curve_i],[\Delta^\alpha])+\mu([\curve_i],[\curve_i]^{\mu(\curve_i,y)}y)\]
The first term on the right is zero by Proposition \ref{prop: monomial}, and the second is zero by Corollary \ref{coro: reducing}.  Therefore, $ \mu([\curve_i],[\Delta^{\mu_\Delta(y)}]y)=0$ and so every term in $\mu_\Delta([\Delta^{\mu_\Delta(y)}]y)$ is zero.
\end{proof}

%In particular, $[\Delta^{\mu_\Delta(y)}]y$ is a polynomial in $\Delta$.
\begin{coro}\label{coro: Laurent}
For all $y\in \Sk_q(\S)$, $[\Delta^{\mu_{\Delta}(y)}]y$ is a polynomial in $\Delta$.% That is, there are $\lambda_\alpha\in \Zq$ such that
%\[ [\Delta^{\mu_{\Delta}(y)}]y = \sum_{\alpha}\lambda_\alpha [\Delta^\alpha]\]
\end{coro}
\begin{rem}\label{rem: Laurent}
If $[\Delta^{\mu_\Delta(y)}]$ had a left inverse in $\Sk_q(\S)$, then we could write
\[ y= \sum_{\alpha}\lambda_\alpha [\Delta^{\mu_\Delta(y)}]^{-1}[\Delta^\alpha]\]
This can be regarded as a (skew) Laurent polynomial in $\Delta$; this will be made precise by introducing quantum tori.  Such an inverse does not exist in $\Sk_q(\S)$, but it will exist in an appropriate localization.
\end{rem}

%\section{Quantum tori}

\subsection{Quantum tori}

Let $\Lambda$ be a skew-symmetric $N\times N$ matrix with integral coefficients.  Define the \textbf{(based) quantum torus} $\T_\Lambda$ of $\Lambda$ to be the associative $\Zq$-algebra such that...
\begin{itemize}
\item As a $\Zq$-module, $\T_\Lambda$ has a free $\Zq$-basis denoted $M^\alpha$ as $\alpha$ runs over $\Z^N$.
\item The product of these basis elements is given by
\[M^\alpha\cdot M^\beta= q^{\frac{1}{2}\Lambda(\alpha,\beta)}M^{\alpha+\beta},\] and general products are determined by $\Zq$-bilinearity.
\end{itemize}
These are `based' quantum tori because the lattice $\Z^N$ comes with an explicit basis, denoted $\{e_1,e_2,...,e_N\}$. There are then distinguished elements of the form $M^{e_i}$, which generate $\T_\Lambda$ together with $M^{-e_i}$.
The basis $\{e_1,e_2,...,e_N\}$ of $\Z^N$ gives elements $\{M^{e_1},M^{e_2},...,M^{e_N}\}$ and $\{M^{-e_1},M^{-e_2},...,M^{-e_N}\}$ which generate the algebra $\T_\Lambda$.

\begin{rem}
The ring $\T_\Lambda$ is also called a ring of `skew-Laurent polynomials'.  The name `quantum torus' is motivated as follows.  The ring $\mathbb{C}\otimes_{\mathbb{Z}}(\T_\Lambda/\langle \rq-1\rangle)$ is a ring of complex Laurent polynomials in $N$ variables (independent of $\Lambda$), which is the ring of regular functions on the variety $(\mathbb{C}^*)^N$, called the `$N$-dimensional algebraic torus'.  In this way, $\mathbb{C}\otimes_\Z\T_\Lambda$ defines a quantization of the algebraic torus with parameter $\rq$.
\end{rem}
%
%A subset $S$ in a ring $R$ satisfies the \textbf{left Ore condition} if
%\[ \forall s\in S, r\in R,\exists s'\in S, r'\in R \text{ s.t. } s'r=r's\]
%The \emph{right Ore condition} is defined symmetrically.  If a multiplicative set $S$ of regular elements satisfies the Ore conditions, then there exists an \textbf{Ore localization} $R[S^{-1}]$, unique up to canonical isomorphism, together with an localization map
%\[ \ell:R\hookrightarrow R[S^{-1}]\]
%such that any other map $R\rightarrow R'$ which sends $S$ to units in $R'$ factors through $\ell$.  If $S$ consists of non-zero divisors, then $\ell$ is also an injection.
%
%An \emph{Ore domain} is a ring $R$ such that $S=R-\{0\}$ satisfies the Ore conditions and consists of non-zero divisors.  These are the rings which can be embedded in a skew-field.
\begin{prop}\cite{GW89}
The quantum torus $\T_\Lambda$ is a Noetherian Ore domain.
\end{prop}
\noindent As a consequence, $\T_\Lambda$ embeds into its skew-field of fractions $\mathcal{F}$.

When $\Lambda=\Lambda^\Delta$, the orientation matrix of a triangulation, we write $\T_\Delta$ for $\T_{(\Lambda^\Delta)}$.

\subsection{Embeddings into quantum tori}

We now have all the tools needed to show that a skein algebra embeds into a quantum torus for each triangulation.
%\begin{lemma}[\ref{?}]
%Let $S$ be a regular multiplicative set in $R$ which satisfies the Ore conditions.  There exists an \textbf{Ore localization} $R[S^{-1}]$, unique up to canonical isomorphism, together with an in9ective localization map
%\[ \ell:R\hookrightarrow R[S^{-1}]\]
%such that any other map $R\rightarrow R'$ which sends $S$ to units in $R'$ factors through $\ell$.
%\end{lemma}

\begin{lemma}
The set of monomials in $\Delta$ generate an Ore set.
\end{lemma}
\begin{proof}
Let $x\in \Sk_q(\S)$ and $[\Delta^{\beta}]$ be a monomial in $\Delta$.  By Corollary \ref{coro: Laurent},
\[ [\Delta^{\mu_\Delta(x)}]x= \sum_{\alpha\in \N^N}\lambda_\alpha[\Delta^\alpha]\]
for finitely many non-zero $\lambda_\alpha\in \Zq$, and so
\begin{eqnarray*}
[\Delta^{\beta+\mu_\Delta(x)}]x &=& q^{-\frac{1}{2}\Lambda^\Delta(\beta,\mu_\Delta(x))}[\Delta^\beta]\sum_{\alpha\in \N^N}\lambda_\alpha[\Delta^\alpha] \\
&=& \left(q^{-\frac{1}{2}\Lambda^\Delta(\beta,\mu_\Delta(x))}\sum_{\alpha\in \N^N}q^{\Lambda^{\Delta}(\beta,\alpha)}\lambda_\alpha[\Delta^\alpha]\right)[\Delta^\beta]
\end{eqnarray*}
Then the set of monomials in $\Delta$ satisfies the left Ore condition.  Since the bar-involution sends monomials to themselves, they automatically satisfy the right Ore condition as well.
\end{proof}

Let $\Sk_q(\S)[\Delta^{-1}]$ be the localization at the monomials in $\Delta$.\footnote{This notation is non-abusive, because $\Sk_q(\S)$ and $[\Delta]^{-1}$ generate $\Sk_q(\S)[\Delta^{-1}]$.}

%\begin{rem}
%The set of monomials in $\Delta$ is not a multiplicative set.  However, the product of two monomials is a monomial times a unit, so the set of monomials in $\Delta$ still determines a complete set of denominators in $\Sk_q(\S)[\Delta^{-1}]$.
%\end{rem}

For any $\alpha\in \Z^\Delta$, define the \emph{Laurent monomial} $[\Delta^\alpha]\in \Sk_q(\S)[\Delta^{-1}]$ by the rule
\[ [\Delta^{\beta'-\beta}]:=q^{\frac{1}{2}\Lambda^\Delta(\beta,\beta')}[\Delta^\beta]^{-1}[\Delta^{\beta'}]\]
One may check that this is independent of the representation $\alpha=\beta'-\beta$, and the multiplication rules of Proposition \ref{prop: monomialmulti} hold for general $\alpha,\beta\in \Z^N$.

\begin{thm}\label{thm: Laurent}
For each triangulation $\Delta$ of $\S$, there is an injective Ore localization
\[ \Sk_q(\S)\hookrightarrow \Sk_q(\S)[\Delta^{-1}]\simeq\T_\Delta\]
which sends $[\Delta^\alpha]$ to $M^\alpha$.
%
%The map $M^\alpha\mapsto[\Delta^\alpha]$ extends to an algebra isomorphism
%\[ \T_\Delta\stackrel{\sim}{\longrightarrow} \Sk_q(\S)[\Delta^{-1}]\]
%The monomials $[\Delta^\alpha]$ as $\alpha$ runs over $\Z^N$ are a $\Zq$-basis of $\Sk_q(\S)[\Delta^{-1}]$.
\end{thm}
\begin{proof}
The injectivity of the Ore localization $\Sk_q(\S)\rightarrow \Sk_q(\S)[\Delta^{-1}]$ follows because the Ore set consists of non-zero-divisors.

Let $f:\T_\Delta\rightarrow \Sk_q(\S)[\Delta^{-1}]$ be the $\Zq$-linear map defined by $f(M^\alpha)=[\Delta^\alpha]$.  This is an algebra homomorphism by Proposition \ref{prop: monomialmulti}.

Let $[\Delta^\alpha]^{-1}x$ be an arbitrary element in $\Sk_q(\S)[\Delta^{-1}]$, with $x\in \Sk_q(\S)$ and $\alpha\in \N^N$.
%We identify $\Sk_q(\S)$ with its image in $\Sk_q(\S)[\Delta^{-1}]$.  Let $x\in \Sk_q(\S)\subset \Sk_q(\S)[\Delta^{-1}]$.
By Corollary \ref{coro: Laurent}, $y=[\Delta^{\mu_\Delta(x)}]x$ is a polynomial in $\Delta$, so there is some $Y\in \T_\Delta$ with $f(Y)=y$.  Then
\begin{eqnarray*}
f(q^{\Lambda(\alpha,\mu_\Delta(x))}M^{-(\alpha+\mu_\Delta(x))}Y)
&=& q^{\Lambda(\alpha,\mu_\Delta(x))}[\Delta^{-(\alpha+\mu_\Delta(x))}]y\\ &=& [\Delta^\alpha]^{-1}[\Delta^{\mu_\Delta(x)}]^{-1}y=[\Delta^\alpha]^{-1}x
\end{eqnarray*}
Therefore, $f$ is surjective.

Let $\gamma=\sum_\alpha\lambda_\alpha M^\alpha$ be an element in the kernel of $f$.  Let $\beta\in \N^N$ such that $\alpha+\beta\in \N^N$ for all $\alpha$ with $\lambda_\alpha\neq0$.
\[ 0=[\Delta^\beta]f\left( \sum_\alpha\lambda_\alpha M^\alpha \right)
 = \sum_\alpha\lambda_\alpha [\Delta^\beta][\Delta^\alpha]
 = \sum_\alpha\lambda_\alpha q^{\lambda^\Delta(\beta,\alpha)/2} [\Delta^{\alpha+\beta}]\]
Since $\alpha+\beta$ is in $\N^N$, the elements $[\Delta^{\alpha+\beta}]$ are simple multicurves.  By Lemma \ref{lemma: basis}, these are independent over $\Zq$, and so $\lambda_\alpha=0$ for all $\alpha$.  Then the kernel of $f$ is $0$, so $f$ is an isomorphism.
\end{proof}
%All the algebraic properties of $\T_\Delta$ are then also true of $\Sk_q(\S)[\Delta^{-1}]$.
\begin{coro}
The Laurent monomials in $\Delta$ are a $\Zq$-basis of $\Sk_q(\S)[\Delta^{-1}]$.
\end{coro}
\begin{proof}
This is true for $\T_\Delta$ by construction.
\end{proof}
\begin{coro}\label{coro: domain}
For any $\S$, $\Sk_q(\S)$ and $\Sk_q^o(\S)$ are Ore domains.
\end{coro}
\begin{proof}
If $\partial\S=\emptyset$, then this is \cite[Theorem 4.7]{PS00}.  For any $\S$ with $\partial\S\neq \emptyset$, it is possible to add marked points to $\S$ to get a marked surface $\S'$ with a triangulation $\Delta$.  By  Theorem \ref{thm: Laurent},
\[\Sk_q(\S)\hookrightarrow \Sk_q(\S')\hookrightarrow \Sk_q(\S')[\Delta^{-1}]\simeq \T_\Delta\]
Then $\Sk_q(\S)$ includes into an Ore domain, so it is an Ore domain.  Since $\Sk_q^o(\S)$ is an injective Ore localization of an Ore domain, it is also an Ore domain.
\end{proof}

\section{Quantum cluster algebras of marked surfaces}\label{section: QCA}

We now turn to cluster algebras of marked surfaces.  %These are algebras constructed from the combinatorial data of triangulations of $\S$.
Cluster algebras are defined in terms of a set of `seeds'; combinatorial objects with the property that the full set of seeds can be recovered from any individual seed by `mutation'.  In the case of triangulable marked surfaces, seeds will correspond to triangulations and mutation will correspond to flipping an arc inside a triangulation.

There are many variations on cluster algebras. We highlight one distinction.
\begin{itemize}
	\item \emph{Commutative cluster algebras} $\A$ are (as you would expect) commutative algebras, defined as subalgebras of $\mathbb{Q}(x_1,x_2,...,x_n)$ generated by a set of elements produced by an iterative mutation rule.
	\item \emph{Quantum cluster algebras} $\A_q$ are $\Zq$-subalgebras of a skew-field $\mathcal{F}$ generated by a set of elements produced by an iterative mutation rule.
\end{itemize}
A quantum cluster algebra $\A_q$ always becomes a commutative cluster algebra $\A_1$ under the specialization $\rq\rightarrow 1$. However, not every commutative cluster algebra can arise this way (see Remark 7.15 for a relevant example), and multiple quantum cluster algebras can have the same commutative specialization (see Section 12.1 for a relevant example).

We focus on the quantum case, and so `cluster algebra' will refer to a quantum cluster algebra.  Commutative cluster algebras will always be labeled as such.

%We will work on the lvelve \emph{quantum cluster algebras}, where a $q$-parameter and a matrix $\Lambda$ are used to measure

%We now introduce quantum cluster algebra of marked surfaces.   Since a full definition would require considerable explanation, it will be more convenient to define it as a subalgebra of the skein algebra, and verify it satisfies the same properties as the quantum cluster algebra.

%\subsection{Based quantum tori}
%
%A \textbf{based quantum torus} is a quantum torus $\T_\Lambda$, together with a choice of basis $\{e_1,e_2,...,e_N\}$ for the defining lattice $\Z^N$ of the quantum torus.  The skew-symmetric matrix $\Lambda$

\subsection{Quantum cluster algebras}\label{section: cluster1}

In \cite{FZ02}, commutative cluster algebras were introduced to axiomatize structures occurring in the study of canonical bases, and it was rapidly discovered that these algebras occur in many areas of math.  In \cite{GSV03}, the authors introduced the idea of a `compatible' Poisson structure on a commutative cluster algebra; and in \cite{BZ05}, these Poisson structures were `quantized' by quantum cluster algebras.

%A quantum cluster algebras will be defined by a \emph{quantum seed}; a collection of data

A \textbf{quantum seed} (of skew-symmetric type\footnote{This is to distinguish from more general `skew-symmetrizable' quantum seeds.}) in a skew-field $\mathcal{F}$ is a triple $(\B,\Lambda,M)$, where...
\begin{itemize}
\item The \emph{exchange matrix} $\B$ is an $N\times \ex$ integer matrix (for a subset $\ex\subseteq \{1,...,N\}$), such that $\pi\B$ is skew-symmetric, where $\pi$ is the $\ex\times N$ matrix which projects $\Z^N$ onto $\Z^\ex$.
\item The \emph{compatibility matrix} $\Lambda$ is an $N\times N$ skew-symmetric, integer matrix, such that $\Lambda\B=D\iota$, where $\iota$ is the $N\times \ex$ matrix which includes $\Z^\ex$ into $\Z^\N$, and $D$ is an $N\times N$ diagonal matrix with entries $D_{ii}>0$.  The identity $\Lambda\B=D\iota$ is called the \emph{compatibility condition}.%$\Lambda$ which determines a skew-symmetric form on $\Z^N$.
\item $M:\Z^N\rightarrow \mathcal{F}-\{0\}$ is a function such that
\[ M(\alpha)M(\beta) = q^{\frac{1}{2}\Lambda(\alpha,\beta)}M(\alpha+\beta)\]
We require that the $\Zq$-span of $M(\Z^N)\subset\mathcal{F}$ is a based quantum torus of $\Lambda$ whose skew-field of fractions is $\mathcal{F}$.
\end{itemize}
Note that $\Lambda$ can be recovered from $M$ by the quasi-commutation relations. %We freely allow reordering the set $\{1,...,N\}$; that is, we consider two quantum seeds $(\B,\Lambda,M)$ and $(\B',\Lambda',M')$ the `same' if there is a bijection $\{1,...,N\}\rightarrow \{1,...,N\}$ which sends $\ex$ to $\ex'$ and conjugates $\B$ to $\B'$, $\Lambda$ to $\Lambda'$ and $M$ to $M'$.

\begin{rem}
The notation for a quantum seed here differs from \cite{BZ05}, who would write $(M,\B)$ where we write $(\B,\Lambda,M)$.
\end{rem}
%The skew-symmetric submatrix $\pi\B$ of $\B$ is one of the most important parts of a seed; many properties will follow just from this part.  We will say two seeds $(\B,\Lambda,M)$ and $(\B',\Lambda',M')$ have the same \textbf{exchange type} if there is a bijection $\ex\simeq \ex'$ which takes the integer matrix $\pi\B$ to $\pi\B'$.

The following proposition is useful to know.
\begin{prop}\label{prop: fullrank}\cite[Proposition 3.3]{BZ05},\cite{GSV03}
For a quantum seed $(\B,\Lambda,M)$, the matrix $\B$ has rank $|\ex|$ (the largest possible).
\end{prop}

A quantum seed $(\B',\Lambda',M')$ is the \textbf{mutation at $i\in \ex$} of a quantum seed $(\B,\Lambda,M)$, both in $\mathcal{F}$, if
\begin{itemize}
\item the \emph{exchange relation} holds:
\[ \B'_{jk} = \left\{\begin{array}{cc}
-\B_{jk} & \text{if $i=j$ or $i=k$} \\
\B_{jk}+\frac{1}{2}(|\B_{ji}|\B_{ik}+\B_{ji}|\B_{ik}|) & \text{otherwise} \\
\end{array}\right.\]
\item for $\alpha\in \Z^N$ such that $\alpha_i=0$, $M(\alpha)=M'(\alpha)$, and
\item the \emph{quantum cluster relation} holds:
\[M'(e_i) = M\left(-e_i+\sum_{\B_{ji}>0}\B_{ji}e_j\right)+M\left(-e_i-\sum_{\B_{ji}<0}\B_{ji}e_j\right)\]
%\[M(e_i)M'(e_i) = q^{-\frac{D_{ii}}{2}}M\left(\sum_{\B_{ij}>0}\B_{ij}e_j\right)+q^{\frac{D_{ii}}{2}}M\left(-\sum_{\B_{ij}<0}\B_{ij}e_j\right)\]
\end{itemize}
For a given quantum seed $(\B,\Lambda,M)$ and $i$, there always exists a unique mutation at $i$ (see \cite[Section 4.4]{BZ05}).  Mutating twice in a row at the same index returns to the original quantum seed, and if $\B_{ij}=0$, then mutating at $i$ and at $j$ commutes.  Two quantum seeds $(\B,\Lambda,M)$ and $(\B',\Lambda',M')$ in $\mathcal{F}$ are \textbf{mutation equivalent} if they can be related by an arbitrary sequence of mutations and reordering indices.

\begin{defn}
The \textbf{quantum cluster algebra} $\A_q(\B,\Lambda,M)$ of a quantum seed $(\B,\Lambda,M)$ is the $\Zq$-subalgebra of $\mathcal{F}$ generated by all elements of the form $M'(\alpha)$, with $(\B',\Lambda',M')$ mutation equivalent to $(\B,\Lambda,M)$, $\alpha_i\in \N$ for $i\in \ex$ and $\alpha_i\in \Z$ for $i\in N-\ex$.
\end{defn}

When the quantum seed is clear, the cluster algebra will be denoted $\A_q$.
An element of the form $M'(e_i)\in \mathcal{F}$ is called a \textbf{cluster variable} in $\A_q(\B,\Lambda,M)$. If $i\in\ex$, then $M(e_i)$ is called a \textbf{mutable variable}; otherwise, it is a \textbf{frozen variable}.  Then $\A_q(\B,\Lambda,M)$ is the subalgebra of $\mathcal{F}$ generated by the cluster variables, together with the inverses of the frozen variables.
\begin{prop}\label{prop: frozenOre}
Any element of $\A_q$ may be written as $a^{-1}b$, where $a$ is a product of frozen variables and $b$ is a polynomial in cluster variables.
\end{prop}
\begin{proof}
Since they are never mutated, frozen variables are represented in every quantum seed of $\A_q$.  Then frozen variables and their inverses quasi-commute with every cluster variable, so they may be collected on the left of any expression in $\A_q$.
\end{proof}
Any quantum cluster algebra determines a quantum upper cluster algebra.
\begin{defn}
The \textbf{quantum upper cluster algebra} $\U_q(\B,\Lambda,M)$ is defined as the intersection of the based quantum tori defined by $M'$, for each quantum seed $(B',\Lambda',M')$ equivalent to $(B,\Lambda,M)$.
\[ \U_q(\B,\Lambda,M) = \bigcap_{(\B',\Lambda',M')\sim (\B,\Lambda,M)} \Zq\cdot M'(\Z^N)\]
\end{defn}
%There is another closely related algebra, the \textbf{quantum upper cluster algebra} $\U_q(\B,\Lambda,M)$.  It is defined as the intersection of the based quantum tori defined by $M'$, for each quantum seed $(B',\Lambda',M')$ equivalent to $(B,\Lambda,M)$.
%\[ \U_q(\B,\Lambda,M) = \bigcap_{(\B',\Lambda',M')\sim (\B,\Lambda,M)} \Zq\cdot M'(\Z^N)\]
\begin{rem}
By \cite[Theorem 5.1]{BZ05}, it suffices to only intersect the $|\ex|+1$ quantum tori corresponding to $(\B,\Lambda,M)$ and its one-step mutations.
\end{rem}
A main result in the theory of cluster algebras is the \emph{Laurent phenomenon}.
\begin{thm}\cite[Corollary 5.2]{BZ05}
$\A_q(\B,\Lambda,M)\subseteq \U_q(\B,\Lambda,M)$.
\end{thm}
While this inclusion is not always equality, there are many important examples where it is.  Determining when $\A_q=\U_q$ is an active area of research in both the quantum and commutative settings.  Techniques for attacking this problem will be developed in Section \ref{section: A=U}.

%The theory of cluster algebras is replete with terminology, but we mention a few relevant terms.  A \textbf{cluster variable} in $\A_q(\B,\Lambda,M)$ is an element of the form $M'(e_i)$, and a \textbf{cluster} is a set of the form $\{M'(e_i)\}$ as $i$ runs from $1$ to $N$.

%While the quantum cluster algebra $\A_q(\B,\Lambda,M)$ depends on the full quantum seed, many properties will only depend on the skew-symmetric submatrix $\pi\B$.  We will say two quantum cluster algebras $\A_q$ and $\A_q'$ have the same \textbf{exchange type} if they have quantum seeds $(\B,\Lambda,M)$ and $(\B',\Lambda',M')$ such that $\pi\B=\pi\B'$.
%\begin{rem}\label{rem: quiver}
%The matrix $\pi\B$ can be encoded in a \emph{quiver} $Q(\pi\B)$, with vertex set $\ex$ and $\B_{ij}$-many arrows from $j$ to $i$ (where negative arrows are from $i$ to $j$).  Many results on a cluster algebra are deduced from graph-theoretic properties of $Q(\pi\B)$.
%\end{rem}

%\begin{rem}
Quantum cluster algebras are quantizations of \emph{commutative cluster algebras}, as defined in \cite{FZ02}.
%The algebras called \emph{cluster algebras} and \emph{upper cluster algebras} are different but related.
These are commutative algebras defined only by an exchange matrix $\B$.  %In this article, we will exclusively call them `commutative cluster algebras' to avoid confusion;  `cluster algebras' will refer exclusively to quantum cluster algebras.

Commutative cluster algebras may be recovered from their quantizations by specializing $\rq$ to $1$; that is, quotienting out by the ideal generated by $\rq-1\in \Zq$.
\[\A_1(\B):=\A_q(\B,\Lambda,M)/\langle \rq-1\rangle\]
\[ \U_1(\B):=\U_q(\B,\Lambda,M)/\langle \rq-1\rangle\]
%Not every commutative cluster algebra comes from a quantum cluster algebras.
%\end{rem}

\subsection{Quantum cluster algebras of marked surfaces}\label{section: cluster2}

In \cite{GSV05}, the authors observe that a triangulable marked surface $\S$ determines a commutative cluster algebra.%, and further fundamental work appeared in \cite{FST08}.
 We now extend their construction to a quantum cluster algebra.

Let $\S$ be a marked surface, and let $\mathcal{F}$ be the skew-field of fractions of the skein algebra $\Sk_q(\S)$. For any triangulation $\Delta$, construct a quantum seed in $\mathcal{F}$ as follows.
\begin{itemize}
\item $\ex\subset\{1,2,...,N\}\simeq \Delta$ is the subset of non-boundary arcs in $\Delta$.\footnote{Recall that a \emph{boundary arc} is an arc homotopic to an arc contained in the boundary $\partial \S$.}
\item $\B^\Delta = \Q^\Delta\circ\iota$, where $\iota:\Z^\ex\rightarrow \Z^N$ is the natural inclusion.
\item $\Lambda^\Delta$ is the orientation matrix of $\Delta$.
\item $M^\Delta:\Z^N\rightarrow\mathcal{F}$ is given by $M_\Delta(\alpha)=[\Delta^\alpha]$.
\end{itemize}
\begin{prop}\label{prop: quantumseed}
The triple $(\B^\Delta,\Lambda^\Delta,M^\Delta)$ is a quantum seed.
\end{prop}
\begin{proof}
%The matrix $\Q^\Delta$ is skew-symmetric, and so $\B^\Delta:=\pi\Q^\Delta\iota$ is skew-symmetric.

The only non-trivial fact to prove is that $\Lambda^\Delta\B^\Delta=4\iota$ (the \emph{compatibility condition}).  Let $\curve_j\in \Delta$ be a non-boundary arc.  For $\curve_i\in \Delta$, consider the matrix entry
\begin{equation*} (\Lambda^\Delta\Q^\Delta)_{ij}=\sum_{1\leq k\leq N} \Lambda^\Delta_{ik}\Q^\Delta_{kj}
%= \left\{ \begin{array}{cc}
%4 & \text{if } i=j \\
%0 & \text{otherwise} \\
%\end{array}\right.
\end{equation*}
The curve $\curve_j$ is an edge in two distinct triangles in $\S-\Delta$.  Let $\curve_{k_1},\curve_{k_2},\curve_{k_3},\curve_{k_4}$ be the other arcs around these triangles, ordered as in Figure \ref{fig: adjacent}.  Note that these arcs need not be distinct nor have distinct endpoints, despite the figure.
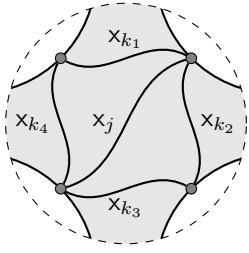
\begin{figure}[h]
\begin{tikzpicture}[scale=.4]
	\begin{scope}
	\clip (0,0) circle (4);
	\draw[fill=black!10,thick] (-5,-1) arc (90:0:4) to (1,-5) arc (180:90:4) to (5,1) arc (270:180:4) to (-1,5) arc (0:-90:4);
	\node[marked] (1) at (-2.17,2.17) {};
	\node[marked] (2) at (2.17,2.17) {};
	\node[marked] (3) at (2.17,-2.17) {};
	\node[marked] (4) at (-2.17,-2.17) {};
	\draw[thick] (4) to [relative,out=-30,in=150] node[left] {$\curve_j$} (2);
	\draw[thick] (1) to [relative,out=-30,in=150] node[above] {$\curve_{k_1}$} (2);
	\draw[thick] (2) to [relative,out=-30,in=150] node[right] {$\curve_{k_2}$} (3);
	\draw[thick] (3) to [relative,out=-30,in=150] node[below] {$\curve_{k_3}$} (4);
	\draw[thick] (4) to [relative,out=-30,in=150] node[left] {$\curve_{k_4}$} (1);
	\end{scope}
	\draw[dashed] (0,0) circle (4);
\end{tikzpicture}
\caption{The adjacent arcs}
\label{fig: adjacent}
\end{figure}

From the definition of $\Q^\Delta$, $\Q^\Delta_{kj}=(-1)^{\ell}$ if $k=k_\ell$, and $0$ otherwise.  Then
\[ (\Lambda^\Delta\Q^\Delta)_{ij}=\sum_{1\leq \ell\leq 4} (-1)^\ell\Lambda^\Delta_{ik_\ell}\]
The arcs $k_\ell$ need not be distinct for the above sum to remain valid.

We consider $\curve_i$ in three cases.
\begin{itemize}
\item Case 1: $i\not\in\{j,k_1,k_2,k_3,k_4\}$.  At each end of $\curve_i$, either there are no ends of the arcs $\curve_{k_\ell}$, or there are two of the form $\curve_{k_\ell}$ and $\curve_{k_{\ell+1}}$ for some $\ell$.  In the latter case, both $\curve_{k_\ell}$ and $\curve_{k_{\ell+1}}$ are either clockwise or counter-clockwise to $\curve_i$, and so $\Lambda^\Delta_{ik_{\ell}}=\Lambda^\Delta_{ik_{\ell+1}}$.  Therefore, $(\Lambda^\Delta\Q^\Delta)_{ij}=0$.
\item Case 2: $i=k_\ell$ for some $\ell$.  Then $\Lambda^\Delta_{ik_{\ell+1}}=-\Lambda^\Delta_{ik_{\ell-1}}$ and all others are zero, so $(\Lambda^\Delta\Q^\Delta)_{ij}=0$.
\item Case 3: $i=j$.  In this case, $\Lambda^\Delta_{ik_{\ell}}=\Q^\Delta_{ik_{\ell}}=(-1)^\ell$, and so $(\Lambda^\Delta\Q^\Delta)_{ij}=4$.
\end{itemize}
By definition, $e_j$ is in the image of $\iota$ if and only if $\curve_j$ is a non-boundary arc, so $\Lambda^\Delta\B^\Delta=\Lambda^\Delta\Q^\Delta\iota = 4\iota$.
%
%\[\Q^\Delta_{kj} = \left\{ \begin{array}{cc}\
%
%Consider $\curve_i,\curve_j\in \Delta$, with $\curve_j$ non-boundary.
%For $\curve_i\in \Delta$, we claim that
%\begin{equation} (\Lambda^\Delta\Q^\Delta)_{ji}=\sum_{1\leq k\leq N} \Lambda^\Delta_{jk}\Q^\Delta_{ki}
%= \left\{ \begin{array}{cc}
%4 & \text{if } i=j \\
%0 & \text{otherwise} \\
%\end{array}\right.
%\end{equation}
%\begin{itemize}
%\item If
%\end{itemize}
%
%Since $\curve_j$ is not a boundary, it is an edge in two distinct triangles in $\S-\Delta$.
%
%We claim this sum is $0$, unless $i=j$, which case it is $4$, which we show in cases.
%
%Case 1. $\curve_i$
\end{proof}

From the definitions, $M^\Delta(e_i)=[\Delta^{e_i}]=[\curve[x]_i]\in \Sk_q(\S)$.

\begin{thm} For any triangulation $\Delta$,
and any flip $\Delta'$ of $\Delta$ at a non-boundary arc $\curve_j$, $(\B^{\Delta'},\Lambda^{\Delta'},M^{\Delta'})$ is the mutation of $(\B^\Delta,\Lambda^\Delta,M^\Delta)$ at $j$.
\end{thm}
\begin{proof}
The exchange relation is unchanged from the commutative version of this theorem, which can be found in \cite[Proposition 4.8]{FST08}.  It is also clear that $M^{\Delta}(\alpha)=M^{\Delta'}(\alpha)$ if $\alpha_j=0$.  The remaining work is the quantum cluster relation.

Let $\curve[x]_j'$ be the flip of $\curve[x]_j$ in $\Delta$, so that $\Delta' = (\Delta -\curve_j)\cup \curve_j'$, and let $\curve_{k_1},\curve_{k_2},\curve_{k_3},\curve_{k_4}$ be as in Figure \ref{fig: adjacent}.  Because the endpoints of $\curve_j$ and $\curve_j'$ need not be distinct, the superposition $\curve_j\cdot\curve_j'$ may not have the simultaneous ordering on all of its ends.  Let $\link$ be the link which is identical to $\curve_j\cdot \curve_j'$ except with the simultaneous ordering on the ends.  There is then some $\lambda\in \Z$  such that
\[[\curve_j][\curve_j']=q^{\frac{1}{2}\lambda}[\link]\]
If the endpoints of $\curve_j$ and $\curve_j'$ are all distinct, this correction is unneeded and $\lambda=0$.

The link $\link$ has a single transverse crossing; by the Kauffman skein relation,
\begin{center}
\begin{tikzpicture}[scale=.5]
	\begin{scope}[xshift=-.1in,scale=.3]
		\begin{scope}
		\clip (0,0) circle (4);
		\draw[fill=black!10,thick] (-5,-1) arc (90:0:4) to (1,-5) arc (180:90:4) to (5,1) arc (270:180:4) to (-1,5) arc (0:-90:4);
		\node[marked] (1) at (-2.17,2.17) {};
		\node[marked] (2) at (2.17,2.17) {};
		\node[marked] (3) at (2.17,-2.17) {};
		\node[marked] (4) at (-2.17,-2.17) {};
		\draw[thick] (1) to [relative,out=-30,in=150](3);
		\draw[thick,line width=1.5mm,draw=black!10] (4) to [relative,out=-30,in=150](2);
		\draw[thick] (4) to [relative,out=-30,in=150] (2);
%		\draw[thick] (1) to [relative,out=-30,in=150] node[near start, above] {$\curve_j$} (3);
%		\draw[thick,line width=1.5mm,draw=black!10] (4) to [relative,out=-30,in=150](2);
%		\draw[thick] (4) to [relative,out=-30,in=150] node[near start, left] {$\curve_j'$} (2);
	%	\draw[thick] (1) to [relative,out=-30,in=150] node[above] {$\curve_{k_4}$} (2);
	%	\draw[thick] (2) to [relative,out=-30,in=150] node[right] {$\curve_{k_3}$} (3);
	%	\draw[thick] (3) to [relative,out=-30,in=150] node[below] {$\curve_{k_2}$} (4);
	%	\draw[thick] (4) to [relative,out=-30,in=150] node[left] {$\curve_{k_1}$} (1);
		\end{scope}
		\draw[dashed] (0,0) circle (4);
	\end{scope}
	\node (=) at (.7in,0) {$=q$};
	\begin{scope}[xshift=1.5in,scale=.3]
		\begin{scope}
		\clip (0,0) circle (4);
		\draw[fill=black!10,thick] (-5,-1) arc (90:0:4) to (1,-5) arc (180:90:4) to (5,1) arc (270:180:4) to (-1,5) arc (0:-90:4);
		\node[marked] (1) at (-2.17,2.17) {};
		\node[marked] (2) at (2.17,2.17) {};
		\node[marked] (3) at (2.17,-2.17) {};
		\node[marked] (4) at (-2.17,-2.17) {};
%		\draw[draw =black!50,fill=black!50] (0,0) circle (1);
		\draw[thick] (4) to [relative,out=-30,in=210] (-.5,-.866) to (-.866,.5) to [relative,out=30,in=150] (1);
		\draw[thick] (2) to [relative,out=-30,in=210] (.5,.866) to (.866,-.5) to [relative,out=30,in=150] (3);
		\end{scope}
		\draw[dashed] (0,0) circle (4);
	\end{scope}
	\node (=) at (2.4in,0.05in) {$+q^{-1}$};
	\begin{scope}[xshift=3.3in,scale=.3,rotate=90]
		\begin{scope}
		\clip (0,0) circle (4);
		\draw[fill=black!10,thick] (-5,-1) arc (90:0:4) to (1,-5) arc (180:90:4) to (5,1) arc (270:180:4) to (-1,5) arc (0:-90:4);
		\node[marked] (1) at (-2.17,2.17) {};
		\node[marked] (2) at (2.17,2.17) {};
		\node[marked] (3) at (2.17,-2.17) {};
		\node[marked] (4) at (-2.17,-2.17) {};
%		\draw[draw =black!50,fill=black!50] (0,0) circle (1);
		\draw[thick] (4) to [relative,out=-30,in=210] (-.5,-.866) to (-.866,.5) to [relative,out=30,in=150] (1);
		\draw[thick] (2) to [relative,out=-30,in=210] (.5,.866) to (.866,-.5) to [relative,out=30,in=150] (3);
		\end{scope}
		\draw[dashed] (0,0) circle (4);
	\end{scope}
	\node (=) at (4.1in,0) {$=q$};
	\begin{scope}[xshift=4.9in,scale=.3]
		\begin{scope}
		\clip (0,0) circle (4);
		\draw[fill=black!10,thick] (-5,-1) arc (90:0:4) to (1,-5) arc (180:90:4) to (5,1) arc (270:180:4) to (-1,5) arc (0:-90:4);
		\node[marked] (1) at (-2.17,2.17) {};
		\node[marked] (2) at (2.17,2.17) {};
		\node[marked] (3) at (2.17,-2.17) {};
		\node[marked] (4) at (-2.17,-2.17) {};
		\draw[thick] (2) to [relative,out=-30,in=150] (3);
		\draw[thick] (4) to [relative,out=-30,in=150] (1);
		\end{scope}
		\draw[dashed] (0,0) circle (4);
	\end{scope}
	\node (=) at (5.8in,0.05in) {$+q^{-1}$};
	\begin{scope}[xshift=6.7in,scale=.3]
		\begin{scope}
		\clip (0,0) circle (4);
		\draw[fill=black!10,thick] (-5,-1) arc (90:0:4) to (1,-5) arc (180:90:4) to (5,1) arc (270:180:4) to (-1,5) arc (0:-90:4);
		\node[marked] (1) at (-2.17,2.17) {};
		\node[marked] (2) at (2.17,2.17) {};
		\node[marked] (3) at (2.17,-2.17) {};
		\node[marked] (4) at (-2.17,-2.17) {};
		\draw[thick] (1) to [relative,out=-30,in=150] (2);
		\draw[thick] (3) to [relative,out=-30,in=150] (4);
		\end{scope}
		\draw[dashed] (0,0) circle (4);
	\end{scope}
\end{tikzpicture}
%\caption{The adjacent arcs}
%\label{fig: adjacent}
\end{center}
In the second equality, we have used homotopy to show that the resulting links have components corresponding to $\curve_{k_2},\curve_{k_4}$ and $\curve_{k_1},\curve_{k_3}$, respectively.  Since $\link$ had simultaneous ends and a single transverse crossing, we can be assured that the right-hand side consists of simple multicurves.  Therefore,
\[ q^{-\frac{1}{2}\lambda}[\curve_j][\curve_j']=q[\curve_{k_2}\cup\curve_{k_4}]+q^{-1} [\curve_{k_1}\cup \curve_{k_3}]\]
which we may rewrite as monomials in $\Delta$, and divide by $q^{-\frac{1}{2}\lambda}[\Delta^{e_j}]$.
\begin{equation} \label{eq: precluster}
[(\Delta')^{e_j}] = q^{\frac{1}{2}(\lambda+2)}[\Delta^{-e_j}][\Delta^{e_{k_2}+e_{k_4}}]+q^{\frac{1}{2}(\lambda-2)}[\Delta^{-e_j}][\Delta^{e_{k_1}+e_{k_3}}]
\end{equation}

\begin{lemma}
$\lambda = \Lambda^\Delta_{jk_2}+\Lambda^\Delta_{jk_4}-2=\Lambda^\Delta_{jk_1}+\Lambda^\Delta_{jk_3}+2$.
\end{lemma}
\begin{proof}
Let $w_1,w_2,w_3,w_4$ denote the four corners of the quadralateral cut out by the $\{\curve_{k_\ell}\}$, thought of as wedges in small neighborhoods of the marked points.
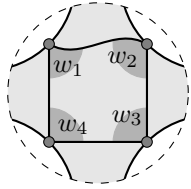
\begin{center}
	\begin{tikzpicture}[scale=.3]
		\begin{scope}
		\clip (0,0) circle (4);
		\draw[fill=black!10,thick] (-5,-1) arc (90:0:4) to (1,-5) arc (180:90:4) to (5,1) arc (270:180:4) to (-1,5) arc (0:-90:4);
		\begin{scope}
			\clip (-2.17,2.17) to [relative,out=-30,in=150]  (2.17,2.17) to (2.17,-2.17) to (-2.17,-2.17) to (-2.17,2.17);
			\draw[fill=black!30,draw=black!30] (-2.17,2.17) circle (1.5);
			\draw[fill=black!30,draw=black!30] (2.17,2.17) circle (1.5);
			\draw[fill=black!30,draw=black!30] (2.17,-2.17) circle (1.5);
			\draw[fill=black!30,draw=black!30] (-2.17,-2.17) circle (1.5);
			\node (w1) at (1.1,1.5) {$w_2$};
			\node (w2) at (1.3,-1.1) {$w_3$};
			\node (w3) at (-1.1,-1.5) {$w_4$};
			\node (w4) at (-1.3,1.1) {$w_1$};
		\end{scope}
		\draw[thick] (-2.17,2.17) to [relative,out=-30,in=150]  (2.17,2.17) to (2.17,-2.17) to (-2.17,-2.17) to (-2.17,2.17);
		\node[marked] (1) at (-2.17,2.17) {};
		\node[marked] (2) at (2.17,2.17) {};
		\node[marked] (3) at (2.17,-2.17) {};
		\node[marked] (4) at (-2.17,-2.17) {};
		\end{scope}
		\draw[dashed] (0,0) circle (4);
\end{tikzpicture}
\end{center}
For $m\neq n\in \{1,2,3,4\}$, define
\[ \Pi_{m,n} = \left\{
\begin{array}{c|c}
0 & \text{if $w_m$ and $w_n$ have disjoint marked points} \\
1 & \text{if $w_m$ is clockwise to $w_n$ at a shared marked point} \\
-1 & \text{if $w_n$ is clockwise to $w_m$ at a shared marked point} \\
\end{array}\right\}\]
Since the interiors of the wedges are disjoint from each other, this is well-defined.  Note that $\Pi_{m,n}=-\Pi_{n,m}$.  From the definitions, %it follows that
\[\lambda= \Pi_{2,1}+\Pi_{2,3}+\Pi_{4,1}+\Pi_{4,3}\]
\[\Lambda^\Delta_{jk_2} = 1+\Pi_{2,3}+\Pi_{4,2}+\Pi_{4,3}\]
\[\Lambda^\Delta_{jk_4} = 1+\Pi_{4,1}+\Pi_{2,4}+\Pi_{2,1}\]
The first equality follows.  The second equality is proved similarly.
\end{proof}
The lemma and Proposition \ref{prop: monomialmulti} imply that
\[q^{\frac{1}{2}(\lambda+2)}[\Delta^{-e_j}][\Delta^{e_{k_2}+e_{k_4}}]
= q^{\frac{1}{2}(\lambda+2)}q^{\frac{1}{2}(-\Lambda^\Delta_{jk_2}-\Lambda^\Delta_{jk_4})}[\Delta^{e_{k_2}+e_{k_4}-e_j}]
= [\Delta^{e_{k_2}+e_{k_4}-e_j}]\]
\[q^{\frac{1}{2}(\lambda-2)}[\Delta^{-e_j}][\Delta^{e_{k_1}+e_{k_3}}]
= q^{\frac{1}{2}(\lambda-2)}q^{\frac{1}{2}(-\Lambda^\Delta_{jk_1}-\Lambda^\Delta_{jk_3})}[\Delta^{e_{k_1}+e_{k_3}-e_j}]
= [\Delta^{e_{k_1}+e_{k_3}-e_j}]\]
Equation \eqref{eq: precluster} then becomes
\[ [(\Delta')^{e_j}] = [\Delta^{e_{k_2}+e_{k_4}-e_j}]+[\Delta^{e_{k_1}+e_{k_3}-e_j}]\]
Switching term on the right, this is the quantum cluster relation, as required.
\[M^{\Delta'}(e_j) = M^\Delta\left(-e_j+\sum_{\B_{kj}>0}\B_{kj}e_k\right)+M^\Delta\left(-e_j-\sum_{\B_{kj}<0}\B_{kj}e_k\right)\]
Then $(\B^{\Delta'},\Lambda^{\Delta'},M^{\Delta'})$ is the mutation of $(\B^\Delta,\Lambda^\Delta,M^\Delta)$ at $j$.
%
%$(2)$. This follows from the first part, since any triangulation $\Delta'$ can be obtained from $\Delta$ by a sequence of flips (Proposition \ref{prop: triang}).
%
%$(3)$. Since $\ex$ is the set of non-boundary arcs, every mutation of a quantum seed will correspond to a flip.
\end{proof}
\begin{coro}\label{coro: independent}
For any two triangulations $\Delta$ and $\Delta'$ of $\S$, the quantum seed $(\B^{\Delta'},\Lambda^{\Delta'},M^{\Delta'})$ is mutation equivalent to $(\B^\Delta,\Lambda^\Delta,M^\Delta)$, and every seed mutation equivalent to $(\B^\Delta,\Lambda^\Delta,M^\Delta)$ is of this form.
\end{coro}
\begin{proof}
Every mutation at $i\in \ex$ corresponds to a flipping a non-boundary arc in a triangulation, so any sequence of mutations corresponds to a sequence of flips.  Since every triangulation $\Delta'$ is related to $\Delta$ by a sequence of flips, every quantum seed coming from a triangulation is mutation equivalent to $(\B^\Delta,\Lambda^\Delta,M^\Delta)$.
\end{proof}

Thus, we can speak unambiguously about `the' quantum cluster algebra $\A_q(\S)$ and quantum upper cluster algebra $\U_q(\S)$ of a triangulable marked surface $\S$.

\begin{defn}\label{defn: CAmarked}
For any triangulation $\Delta$ of $\S$, the subalgebras
\[ \A_q(\S):=\A_q(\B^\Delta,\Lambda^\Delta,M^\Delta)\subset\mathcal{F}\]
\[\U_q(\S):=\U_q(\B^\Delta,\Lambda^\Delta,M^\Delta)\subset\mathcal{F}\]
are the \textbf{quantum cluster algebra of $\S$} and the \textbf{quantum upper cluster algebra of $\S$}, respectively.
\end{defn}

\begin{rem}
These quantum cluster algebras are quantizations of the commutative cluster algebras of marked surfaces defined in \cite{GSV05} and \cite{FST08}, with \emph{boundary coefficients}.  This means the coefficients (in the sense of \cite{FZ02}) are the Laurent ring generated by the set of boundary arcs.
The coefficient-free case may be recovered by quotienting $\A_q(\S)$ or $\U_q(\S)$ by the ideal generated by $\rq-1$ and $\{\curve-1\}$ as $\curve$ runs over the set of boundary arcs.  %However, the coefficient-free commutative cluster algebra of a marked surface often admits no quantization.\footnote{To see this, notice that when $N-|\M|$ is odd, the coefficient-free exchange matrix is an odd-sized, skew-symmetric matrix, and can never be full rank.}
\end{rem}

%\begin{rem}
%In general, there are other quantizations of the commutative cluster algebra of a marked surface.  A different quantization of the commutative cluster algebra of a disc is addressed in Section \ref{section: disc} and \cite{GL11}.  Our justification for considering this quantization is its relation to the skein algebra. 
%
%Another virtue is that this is essentially the only quantization which is compatible with the commutative cluster structure and local in the following sense.  
%\begin{prop}
%Let $\S$ be a triangulable marked surface which is connected and has at least one non-boundary arc. For the exchange matrix $\mathbf{B}^\Delta$ of a triangulation of $\S$, let $\Lambda'$ be a compatibility matrix such that $\Lambda'_{ij}=0$ for any pair of arcs $\curve_i$ and $\curve_j$ with no common endpoints.  Then $\Lambda'$ is a positive integral multiple of the orientation matrix $\Lambda^\Delta$.
%\end{prop}
%\begin{proof}
%To see this, consider any pair of adjacent triangles in the triangulation $\Delta$ (at least one pair exists by assumption).  As in Figure \ref{fig: ???}, let $\curve_j$ be the shared arc, and let $\curve_{k_1},...\curve_{k_4}$ be the other arcs (which may not be distinct).  The compatibility condition at entry $k_m,j$ states that
%\[ (\Lambda'\mathbf{B}^\Delta)_{k_mj}= \sum_{1\leq\ell\leq 4} (-1)^\ell \Lambda'_{k_m k_\ell}=0\]
%By the locality assumption, $\Lambda'_{k_mk_{m+2}}=0$, and so 
%\[ \Lambda'_{k_mk_{m+1}} = -\Lambda'_{k_mk_{m-1}}\]
%\end{proof}

%\end{rem}

\begin{rem}\label{rem: punctures}
We can now justify requiring that the marked points are contained in the boundary.  For a marked surface $\S$ with internal marked points, there is an associated commutative cluster algebra $\A(\S)$ defined in \cite{GSV03} and \cite{FST08} (where the coefficients are the Laurent ring generated by the boundary arcs).  It is possible to use the `tagged arcs' of \cite{FST08} to define a commutative `tagged skein algebra' $\Sk(\S)$ (with $\rq=1$) which has a localization $\Sk^o(\S)$ which is naturally a cluster algebra.\footnote{This is intended for a subsequent publication.}

However, for any triangulation $\Delta$ of $\S$, the corresponding exchange matrix $\B^\Delta$ will never be of full rank.  Therefore, by Proposition \ref{prop: fullrank}, this commutative cluster algebra admits no quantization.  It is possible there is a well-behaved generalization of $\S$ to the case of internal marked points for general $q$, but it cannot correspond to the quantum cluster algebra of $\S$ (with coefficients coming from boundary arcs).
\end{rem}

\subsection{Relation to the skein algebra}

%Since every boundary arc is in every triangulation, a multicurve consisting of boundary arcs is automatically a monomial in every triangulation.

%Define the \textbf{localized skein algebra} $\Sk_q^o(\S)$ to be the Ore localization at the monomials in the boundary arcs.  This algebra contains monomials of the form $[\Delta^\alpha]$, where $\alpha$ is positive on internal arcs and integral on boundary arcs.

The algebras $\A_q(\S)$ and $\U_q(\S)$ were defined as subalgebras of $\mathcal{F}$, the skew-field of fractions of $\Sk_q^o(\S)$, and so the three algebras can be compared as subalgebras.

\begin{thm}\label{thm: main1}
For any triangulable marked surface $\S$,
\[ \A_q(\S)\subseteq \Sk_q^o(\S)\subseteq \U_q(\S)\]
\end{thm}
It will be shown in Theorem \ref{thm: main2} that these inclusions are equalities so long as $\S$ contains at least two marked points.
\begin{proof}[Proof of Theorem \ref{thm: main1}]
Let $\Delta$ be a triangulation of $\S$. %, and let $(\B^\Delta,\Lambda^\Delta,M^\Delta)$ be the corresponding quantum seed.
By definition, $M^\Delta(\alpha)=[\Delta^\alpha]$.  For any $\alpha\in \Z^N$ with $\alpha_i\geq 0$ for $i\in \ex$, write $\alpha=\beta-\beta'$, where $\beta,\beta'\in \N^N$ and $\beta'_i=0$ for $i\in \ex$.  Then
\[ M^\Delta(\alpha)=[\Delta^\alpha]= q^{-\frac{1}{2}\Lambda^\Delta(\beta,\beta')}[\Delta^\beta][\Delta^{\beta'}]^{-1}\]
Since $[\Delta^{\beta'}]$ is a monomial in the boundary arcs, $M^\Delta(\alpha)\in \Sk^o_q(\S)$.  Since this is true for any quantum seed and any $\alpha$ with $\alpha_i\geq0$ for $i\in \ex$, $\A_q(\S)\subseteq \Sk_q^o(\S)$.

The quantum torus $\Zq\cdot M^\Delta(\Z^N)$ is the same as $\T_\Delta$, because they are both the $\Zq$-span of the $[\Delta^\alpha]$ for $\alpha\in \Z^N$.  Then, Theorem \ref{thm: Laurent} implies that $\Sk_q^o(\S)\subseteq \T_\Delta$.  Since this is true for any quantum seed, $\Sk_q^o(\S)\subseteq \U_q(\S)$.
\end{proof}

\begin{rem}\label{rem: main1rem}
Under this inclusion, $\A_q(\S)$ is the $\Zq$-subalgebra of $\Sk^o_q(\S)$ generated by arcs (ie, cluster variables) and inverses to boundary arcs.  Hence, the definition given here for $\A_q(\S)$ (Definition \ref{defn: CAmarked}) agrees with the one give in the introduction.
\end{rem}

\begin{rem}
Let $\A^\natural_q(\S)$ be the $\Zq$-subalgebra of $\mathcal{F}$ generated by the cluster variables, but not the inverses to `frozen' variables.  Then $\A^\natural_q(\S)\subset \Sk_q(\S)$ as the subalgebra generated by the arcs; however, this is only an equality when $\S$ is contractible.  If there is a non-trivial loop $\ell\in \S$, then $[\ell]\in\Sk_q(\S)$ is not a scalar, but does have $E$-degree zero.  This cannot happen in $\A^\natural_q(\S)$, so $\A^\natural_q(\S)\neq \Sk_q(\S)$.
\end{rem}

\subsection{Laurent formulae and denominators}

Given a link $\link$ and a triangulation $\Delta$, the proof of Theorem \ref{thm: Laurent} gives an explicit method to express $[\link]$ as an element of the quantum torus $\mathbb{T}_\Delta$.  Specifically, applying the Kauffman skein relation repeatedly to $[\Delta^{\mu_{\Delta}([\link])}][\link]$ eventually gives a polynomial $\sum_\alpha\lambda_\alpha[\Delta^\alpha]$ in $\Delta$, and so
\[ [\link] = [\Delta^{\mu_{\Delta}([\link])}]^{-1}\sum_\alpha\lambda_{\alpha} [\Delta^{\alpha}]= \sum_\alpha \lambda_\alpha'[\Delta^{\alpha-\mu_{\Delta}([\link])}]\]
Since $\A_q(\S)\subseteq \Sk_q^o(\S)$, this can also be applied to any cluster variable in $\A_q(\S)$.  A cluster variable will correspond to a simple arc $\curve$, and so
\[[\curve] = [\Delta^{\mu_{\Delta}([\curve])}]^{-1}\sum_\alpha\lambda_{\alpha} [\Delta^{\alpha}]= \sum_\alpha \lambda_\alpha'[\Delta^{\alpha-\mu_{\Delta}([\curve])}]\]
This gives an effective method for expressing a cluster variable as a skew-Laurent polynomial in the cluster variables of any other cluster.  This approach is already well-known for commutative cluster algebras.  For discs, explicit formulas appear in the work of Schiffler \cite{Sch08}, and are more explicitly related to the skein relations in \cite[Section 2.1.5]{GSV10}.  For general marked surfaces, a related method for producing Laurent expansions in terms of \emph{T-paths} and \emph{snake graphs} has been developed in \cite{Sch10}, \cite{ST09}, and \cite{MS10}.

One consequence of this formula is a denominator $[\Delta^{\mu_{\Delta}([\curve])}]$ for the skew-Laurent expression.  This is the smallest possible denominator, as this proposition shows.
\begin{prop}\cite[Theorem 8.6]{FST08}\footnote{Their result is for commutative cluster algebras, but it implies the quantum result.}
If $\curve$ is a simple arc in $\S$, $\Delta$ is a triangulation of $\S$ and $[\Delta^\alpha][\curve]\in \Sk_q(\S)$ is a polynomial in $\Delta$, then \[\alpha - \mu_\Delta([\curve])\in \N^\Delta\]
\end{prop}

\subsection{Gradings on $\A_q(\S)$} %(not finished)

In \cite[Section 2.2]{GSV03}, the authors define a grading\footnote{In truth, \cite{GSV03} define a torus action on $\A$, but semi-simple torus actions by $\T$ are equivalent to gradings by the character lattice of $\T$.} on any cluster algebra, which is the `largest possible' compatible grading.  For $\A_q(\S)$, this is shown to coincide with the endpoint $\E$-grading defined in Section \ref{section: grading}.

For any abelian group $L$, an $L$-grading on a cluster algebra $\A_q$ is \emph{compatible} if each cluster variable is homogeneous.
%A compatible $L$-grading on $\A_q$ extends to a unique $L$-grading on $\U_q$ and each quantum torus $\T_{\Lambda}$.
Given a compatible $L$-grading on $\A_q$, a morphism $f:L\rightarrow L'$ induces a compatible $L'$-grading on $\A_q$, by $\deg_{L'}(x):= f(\deg_L(x))$.  A compatible grading on $\A_q$ is \emph{universal} if it is the initial object in the category of compatible gradings of $\A_q$ and induction maps between them.

In \cite{GSV03}, the authors define compatible gradings, and characterize the universal compatible grading of any cluster algebra.  %For a quantum seed $(\B,\Lambda,M)$,

\begin{lemma}\cite[Lemma 2.3]{GSV03}\footnote{The result in \cite{GSV03} is stated for torus actions, but is equivalent to the result stated here.}
Let $(\B,\Lambda,M)$ be a quantum seed for $\A_q$.  Then
\[ \deg(M(\alpha)) = \alpha +\B( \mathbb{Z}^{\ex}) \in (\mathbb{Z}^N/\B( \mathbb{Z}^{\ex}))\]
extends to a universal compatible grading of $\A_q$ by $\Z^N/\B(\Z^\ex)$.
\end{lemma}

For the cluster algebras we consider, this coincides with the endpoint $\E$-grading.

\begin{prop}
For any $\Delta$, the map $\delta:\Z^\Delta\rightarrow \E$ which sends an arc in $\Delta$ to its endpoints induces an isomorphism $\Z^\Delta/\B(\Z^\ex)\stackrel{\sim}{\longrightarrow} \E$.
\end{prop}
\begin{proof}
Simple arcs in $\S$ are $\E$-homogeneous elements in $\Sk_q(\S)$, so $\A_q(\S)$ is generated by $\E$-homogeneous elements (cluster variables and inverses to frozen variables); it follows that $\A_q(\S)$ is compatibly $\E$-graded.

Fix a triangulation $\Delta$, and let $\delta:\Z^\Delta\rightarrow \E$ be the map which sends a monomial in $\Delta$ to its endpoint degree.  The map $\delta$ kills the image $\B^\Delta(\Z^\ex)$, and so it descends to a map $\delta':\Z^\Delta/\B^\Delta(\Z^\ex)\rightarrow \E$ which is the map which induces the $\E$-grading from the $\Z^\Delta/\B^\Delta(\Z^\ex)$-grading.

For every pair of marked points in a connected component of $\S$, there is an arc connecting them.  The degrees of these arcs generated $\E$, and so the map $\delta'$ is surjective.

The lattice $\E$ is a full-rank sublattice of $\Z^\M$, so it has rank $|\M|$.  The lattice $\Z^\Delta/\B(\Z^\ex)$ has rank equal to $|\Delta|-rank(\B^\Delta)$.  By Proposition \ref{prop: fullrank}, $rank(\B^\Delta)=|\ex|$, and $|\Delta|-|\ex|$ is the number of boundary arcs $|\M|$.  Then $\delta'$ is a surjective maps between lattices of the same rank, so it is an isomorphism.
\end{proof}

\begin{coro}
The endpoint $\E$-grading on $\Sk_q^o(\S)$ restricts to a universal compatible grading on $\A_q(\S)$.
\end{coro}

%\subsection{Multiple components}
%
%When $\S$ has multiple components, the cluster algebra $\A_q(\S)$ is the $\Zq$-tensor product of the cluster algebra of each component.
%
%\begin{prop}\label{prop: components}
%If $\S = \coprod_i\S_i$ is a finite disjoint union, then
%\[ \A_q(\S) = \bigotimes_{\Zq} \A_q(\S_i)\;\;\;\text{ and }\;\;\; \U_q(\S) = \bigotimes_{\Zq} \U_q(\S_i)\]
%\end{prop}
%\begin{proof}
%
%\end{proof}
%
%\begin{rem}
%The same result is also true for $\Sk_q(\S)$, but we shall not need it.
%\end{rem}

%\begin{coro}
%The $\E$-grading is the universal compatible grading; that is, for any compatible $L$-grading on $\A_q$, there is a map of abelian groups $\E\rightarrow L$ which induces the $L$-grading from the $\E$-grading.
%\end{coro}

%Because arcs are homogeneous, the endpoint $\E$-grading on $\Sk_q^o(\S)$ restricts to an $\E$-grading on $\A_q(\S)$.  The $\E$-grading also extends to an $\E$-grading on the quantum torus $\T_\Delta$ corresponding to any triangulation, and
%
%extends to   The
%
%The arcs are homogeneous for the $E$-grading on $\Sk^o_q(\S)$, and so $\A_q(\S)$ is an $E$-graded subalgebra of $\Sk^o_q(\S)$.
%
%For an abelian group $L$, an $L$-grading on a quantum cluster algebra is called \textbf{compatible} if every cluster variable is homogeneous. The $E$-grading on $\A_q(\S)$ is compatible, and it is the universal compatible grading in the following sense.
%\begin{prop}
%Every compatible $L$-grading on $\A_q(\S)$ is induced from the $E$-grading by a homomorphism $E\rightarrow L$.
%\end{prop}

\section{A general technique for $\A_q=\U_q$}\label{section: A=U}

In this section, we develop a technique for simultaneously proving $\A_q=\U_q$ for classes of cluster algebras.  Many of the ideas here are quantum analogs of commutative ideas which appeared in \cite{MulLA}.

\subsection{Exchange types}

%The cluster algebra techniques of this section will only depend on the skew-symmetric submatrix $\pi\B$ of any quantum seed, and even then only up to mutation and conjugation by a permutation matrix.  We now formalize this equivalence relation.

An $n\times n$ integral skew-symmetric matrix $\mathsf{A}$ may be \textbf{mutated} at an index $i\in \{1,...,n\}$ using the exchange relation as in Section \ref{section: QCA}.  By construction, this notion of mutation is compatible with mutation of quantum seeds under the map which sends any quantum seed $(\B,\Lambda,M)$ to the matrix $\pi\B$.  \footnote{Recall that $\pi$ is the the $\ex\times N$-matrix which projects $\mathbb{Z}^N$ onto $\mathbb{Z}^\ex$, so that $\pi\B$ is the 'principal part' of $\B$ (see Section 7.1).}

An \textbf{exchange type} $\mathcal{T}$ is an equivalence class of skew-symmetric matrices, under the relation generated by mutation and conjugation by a permutation matrix.  Given a quantum seed $(\B,\Lambda,M)$, the exchange type of $\pi\B$ consists of matrices of the form $\pi\B'$ for quantum seeds $(\B',\Lambda',M')$ mutation equivalent to $(\B,\Lambda,M)$. %\footnote{Recall that reordering the indices of a quantum seed was explicitly allowed.}
We say the \emph{exchange type} of a quantum seed $(\B,\Lambda,M)$ is the exchange type of $\pi\B$, and the \emph{exchange type} of a cluster algebra $\A_q$ is the exchange type of any of its quantum seeds.

The results which follow depend only on the exchange type of a cluster algebra.%, which makes them easy to apply broadly.

%Given two quantum seeds $(\B,\Lambda,M)$ and $(\B',\Lambda',M')$, an \textbf{equivalence of exchange types} is a bijection $f:\ex\rightarrow \ex'$ such that, $\forall i,j\in \ex$, $\B_{ij}=\B'_{f(i)f(j)}$.  Alternatively, it is an $\ex\times \ex'$ permutation matrix $\sigma$ such that $\sigma (\pi\B)\sigma^{-1}=\pi\B'$.
%
%An \textbf{equivalence of exchange types} between two cluster algebras is an equivalence of exchange types between any two of their quantum seeds.  The exchange type of a cluster algebra is determined by a skew-symmetric matrix $\pi\B$ up to mutation and permutation-conjugation.

\begin{rem}\label{rem: quiver}
An  $n\times n$ integral skew-symmetric matrix $\mathsf{A}$ can be encoded in a \emph{quiver} $Q(\mathsf{A})$, with vertex set $\{1,...,n\}$ and $\mathsf{A}_{ij}$-many arrows from $j$ to $i$ (where negative arrows are from $i$ to $j$).  Mutation can be encoded as an operation on a quiver \cite[Section 2]{Kel11}, and exchange types correspond to mutation-equivalence classes of quivers.
%
%An equivalence of exchange types between $(\B,\Lambda,M)$ and $(\B',\Lambda',M')$ is then a quiver isomorphism between $Q(\pi\B)$ an $Q(\pi\B')$.
\end{rem}

%As it happens, the approach of the previous section can still work in the quantum case, once the necessary results have been extended to the quantum case.

\subsection{Isolated cluster algebras}

%The mutation of an $n\times n$ zero matrix is still the zero matrix, so there is an exchange type consisting of the $n\times n$ zero matrix.
A quantum seed or cluster algebra is called \textbf{isolated} if its exchange type is the zero matrix.
Concretely, a cluster algebra is isolated if every quantum seed $(\B,\Lambda,M)$ has $\pi\B=0$.  %Isolated cluster algebras have a very simple structure and many res

\begin{prop}\label{prop: isolated}
If $\A_q$ is isolated, $\A_q=\U_q$.
\end{prop}
\begin{rem}
In \cite[Theorem 7.5]{BZ05}, the authors show that $\A_q=\U_q$ whenever $\A_q$ has an \emph{acyclic} exchange type, which immediately implies this proposition.  We include a proof anyway, because a by-product of the techniques we develop will be a new proof of Berenstein and Zelevinsky's theorem (Proposition \ref{prop: acyclic}).
\end{rem}
\begin{proof}[Proof of Proposition \ref{prop: isolated}]

Let $(\B,\Lambda,M)$ be a seed for $\A_q$, with corresponding quantum torus $\T_\Lambda$.  Let $R$ denote the subring of $\A_q$ generated by the frozen variables and their inverses. The ring $R$ is naturally a quantum subtorus of $\T_\Lambda$, and so $\T_\Lambda$ is a free left $R$-module with basis $\{M(\alpha)\}$ as $\alpha$ runs over $\Z^\ex$.

Since $\B_{ij}=0$ for all $i,j\in \ex$, all mutations commute with each other.  Mutating once at each $i\in \ex$ in any order gives the quantum seed $((-1)^{|\ex|}\B,\Lambda',M')$, with\footnote{Here, and throughout, $q^\bullet$ denotes a half-power of $q$ not worth keeping careful track of.}
\[ P_i:= M'(e_i)M(e_i) = q^{\bullet}M\left(\sum_{\B_{ji}>0}\B_{ji}e_j\right)+q^{\bullet}M\left(-\sum_{\B_{ji}<0}\B_{ji}e_j\right)\]
Since the expression on the right contains no indices in $\ex$, $P_i\in R$.

Choose an element $x\in \U_q\subseteq \T_\Lambda$, write (for $\lambda_\alpha\in R$)
\begin{eqnarray*}
 x  &=& \sum_\alpha\lambda_\alpha M(\alpha)= \sum_{\alpha\in \Z^\ex}q^\bullet\lambda_\alpha M(e_1)^{\alpha_1}M(e_2)^{\alpha_2}...M(e_{|\ex|})^{\alpha_{|\ex|}}
\end{eqnarray*}
Let us rewrite $x$ by replacing some of these cluster variables with their mutation.  Choose any $I\subset \ex$.
\begin{eqnarray*}
 x &=&\sum_{\alpha\in \Z^\ex}q^\bullet\lambda_\alpha \left(\prod_{i\in I}M(e_i)^{\alpha_i}\right)\left(\prod_{i\not\in I} M(e_i)^{\alpha_i}\right)\\
 &=& \sum_{\alpha\in \Z^\ex}q^\bullet\lambda_\alpha \left(\prod_{i\in I}(M'(e_i)^{-1}P_i)^{\alpha_i}\right)\left(\prod_{i\not\in I} M(e_i)^{\alpha_i}\right)\\
&=&\sum_{\alpha\in \Z^\ex}q^\bullet\lambda_\alpha \left(\prod_{i\in I}P_i^{\alpha_i}\right)\left(\prod_{i\in I}M'(e_i)^{-\alpha_i}\right)\left(\prod_{i\not\in I} M(e_i)^{\alpha_i}\right)
\end{eqnarray*}

Let $\T_I$ be the quantum torus corresponding to the seed which is the mutation of $(\B,\Lambda,M)$ at the set $I$ in any order.  The cluster variables in this seed are $\{M'(e_i)\}_{i\in I}\cup \{M(e_i)\}_{i\not\in I}$.  Since the exponents $\alpha_i$ may be negative, it is not immediate that the coefficient $\lambda_\alpha\prod_{i\in I}P_i^{\alpha_i}$ is an element of $R$.  However, because $x\in \U_q$, it is also in $\T_I$, and so this coefficient is in $R$.

Let $I_\alpha\subseteq \ex$ be the set on which $\alpha$ is negative.  Then
\begin{eqnarray*}
 x &=& \sum_{\alpha\in \Z^\ex}q^\bullet\left(\lambda_\alpha \prod_{i\in I_\alpha}P_i^{\alpha_i}\right)\left(\prod_{i\in I_\alpha}M'(e_i)^{-\alpha_i}\right)\left(\prod_{i\not\in I_\alpha} M(e_i)^{\alpha_i}\right)
\end{eqnarray*}
This expression is in $\A_q$, so $\A_q=\U_q$.
\end{proof}
\begin{rem}
This proof is essentially the same as that of \cite[Lemma 4.1]{BFZ05}.
\end{rem}

%A quantum seed $(\B,\Lambda,M)$ is \textbf{isolated} if $\pi\B=0$, and a cluster algebra is isolated if any quantum seed is isolated (equivalently, every quantum seed is isolated).

\subsection{Freezing and cluster localization}

Let $\A_q$ be a quantum cluster algebra, with skew-field of fractions $\mathcal{F}$.  Fix a quantum seed $(\B,\Lambda,M)$ of $\A_q$, and choose a set $s\subset\ex$ of exchangeable indices.  If we let $\ex^{(s)}=\ex-s$ and $\B^{(s)}$ be the restriction of $\B$ to $\ex^{(s)}$, then $(\B^{(s)},\Lambda,M)$ defines a new quantum seed, called the \textbf{freezing} of $(\B,\Lambda,M)$ at $s$.  Let $\A_q^{(s)}$ and $\U_q^{(s)}$ be the corresponding cluster algebras of this new seed.  By construction, these new algebras are subalgebras of $\mathcal{F}$.

On the level of the principal part $\pi\mathbb{B}$, freezing a set $s$ is the square submatrix on the indices $\ex-s$.  As a slight abuse of notation, for any skew-symmetric matrix $\mathsf{A}$, we write $\mathsf{A}^{(s)}$ for the square submatrix of $\mathsf{A}$ after deleting the columns and rows in $s$ (whether or not we think of $\mathsf{A}$ as the skew-symmetric part of an exchange matrix).

%We extend this notation to any skew-symmetric matrix $\mathsf{A}$.  If $\mathsf{A}$ is $n\times n$ and $s\subset \{1,...,n\}$, then $\mathsf{A}^{(s)}$ will denote the square submatrix on the indices $\{1,...,n\}-s$.  In this notation, $(\pi\B)^{(s)}=\pi(\B^{(s)})$.

Denote by
$ S:= \{M(e_i)\,|\, i\in s\}$
the initial cluster variables in $\A_q$.  Let $\A_q[S^{-1}]$ (resp. $\U_q[S^{-1}]$) denote the subalgebra of $\mathcal{F}$ generated by $\A_q$ and $S^{-1}$ (resp. $\U_q$ and $S^{-1}$).  %These localizations are not \emph{a priori} Ore localizations.%\footnote{That is, it is not clear that denominators may be collected on the left or the right.}

%The set of monomials in $S$ is an Ore set in $\A_q$, and the corresponding localizations will be denoted by $\A_q[S^{-1}]$ and $\U_q[S^{-1}]$.

These four algebras can be compared by the following proposition.\footnote{This is the quantum analog of \cite[Proposition 3.1]{MulLA}.}
\begin{prop}\label{prop: inclusions}
There are inclusions in $\mathcal{F}$
\[ \A_q^{(s)}\subseteq \A_q[S^{-1}]\subseteq \U_q[S^{-1}] \subseteq \U_q^{(s)}\]
\end{prop}
\begin{proof}
The cluster variables of $\A_q^{(s)}$ are a subset of the cluster variables of $\A_q$.  The only new generators are the inverses of the newly-frozen variables, but those are in the localization by construction.  This gives the first inclusion.
Similarly, $\A_q^{(s)}$ has fewer clusters than $\A_q$, so the intersection defining $\U_q^{(s)}$ has strictly fewer terms than $\U_q$; so ${\U_q}\subseteq \U_q^{(s)}$.  Since $\U_q^{(s)}$ also contains the inverses of $S$, this gives the last inclusion.
The middle inclusion follows from the inclusion $\A_q\subseteq \U_q$.
\end{proof}
If $\A_q^{(s)}=\A_q[S^{-1}]$, then $\A_q^{(s)}$ is a localization of $\A_q$ which is naturally a cluster algebra; in this case, we call $\A_q^{(s)}$ a \textbf{cluster localization} of $\A_q$.  Determining which freezings give cluster localizations seems to be an interesting problem.

One nice aspect of cluster localizations is that they are Ore localizations.
\begin{prop}
If $\A_q^{(s)}=\A_q[S^{-1}]$ is a cluster localization, then it is an Ore localization of $\A_q$ at the multiplicative set generated by $S$.
\end{prop}
\begin{proof}
%In a cluster algebra $\A_q'$, the frozen variables quasi-commute with every cluster variable.  It follows that any element in $\A_q'$ can be written as $x^{-1}y$, where $x$ is a monomial in the frozen variables.
%
Any $x\in \A_q$ is in $\A_q^{(s)}$, and so by Proposition \ref{prop: frozenOre}, $x=a^{-1}b$ for $a$ a product of frozen variables of $\A_q^{(s)}$ and $b$ a polynomial in the cluster variables of $\A_q^{(s)}$.  The cluster variables of $\A_q^{(s)}$ are a subset of the cluster variables of $\A_q$, so $b\in \A_q$.

Frozen variables in $\A_q'$ are either frozen in $\A_q$ or in $S$, so we can write $a=q^{\lambda}cd$, where $c$ is a product of frozen variables in $\A_q$ and $d$ is a product of elements in $S$.  Then $x=d^{-1}(q^{-\lambda}c^{-1}b)$, where $d$ is a product of elements in $S$, and $q^{-\lambda}c^{-1}b\in \A_q$.  Then $\A_q[S^{-1}]$ is a left Ore localization.  Since the elements of $S$ are fixed by the bar involution, it is also a right Ore localization.
%%
%%Any $x\in \A_q\subset \A_q'$ can then be writt
%%
%%By cluster localization, $\A_q\subset\A_q'$.  We can write any element of $\A_q$ as
%\[ x=M(\beta)^{-1}\sum_i \lambda_i M^{i,1}(\alpha_{i,1})M^{i,2}(\alpha_{i,2})...M^{i,n}(\alpha_{i,n})\]
%where $\lambda_i\in \Z_q$, the $M^{i,j}$ are from seeds equivalent to $(\B,\Lambda,M)$, $\alpha_{i,j}\in \N^N$, and$\beta$ is supported in $N-\ex'=N-\ex-s$.
%Then
%\begin{eqnarray*}
%M(e_i)x &=& M(e_i)M(\beta)^{-1}\sum f_i M_{i,1}(\alpha_{i,1})M_{i,2}(\alpha_{i,2})...M_{i,n}(\alpha_{i,n})\\
%&=& M(\beta)^{-1}\sum q^\bullet f_i M_{i,1}(\alpha_{i,1})M_{i,2}(\alpha_{i,2})...M_{i,n}(\alpha_{i,n})M(e_i)
%%&=&q^{\Lambda(\beta,e_i)}M(\beta)^{-1}\sum_iq^{\Lambda_{i,1}(\alpha_{i,1},\beta)} M_{i,1}(\alpha_{i,1})q^{\Lambda_{i,2}(\alpha_{i,2},\beta)}M_{i,2}(\alpha_{i,2})...q^{\Lambda_{i,n}(\alpha_{i,n},\beta)}M_{i,n}(\alpha_{i,n})M(e_i)\\
%\end{eqnarray*}
%Therefore, for all $x\in\A_q$, there is a $y\in\A_q$ such that $M(e_i)x=yM(e_i)$.
\end{proof}

\begin{rem}
If $s=\ex$, then $\A_q^{(s)}=\A_q[S^{-1}]$ is the quantum torus $\T_\Lambda\subset \mathcal{F}$ corresponding to the quantum seed $(\B,\Lambda,M)$.  In this way, cluster localizations generalize these embeddings.
\end{rem}

%
%Freezing is compatible with exchange types.
%\begin{prop}
%
%\end{prop}
%
%; that is, the exchange type of a freezing only depends on the exchange type of the original cluster algebra.  Specifically, $\pi\B'$ corresponds to deleting the frozen rows and columns in $\pi\B$.
\begin{rem}
In terms of the quiver $Q(\pi\B)$, freezing deletes the vertices in $s$.
\end{rem}

\subsection{Relatively prime elements}

%How does one find collections of cluster localizations whose intersection is the original cluster algebra?
We give a technique for producing localizations of a cluster algebra whose collective intersection is the original cluster algebra.  This algorithm will only depend on the skew-symmetric submatrix $\pi\B$.%, which means that it can be broadly applied to many cluster algebras simultaneously.

%Finding well-behaved cluster localizations of a cluster algebra is only part of the problem.  The other half of the problem is finding a set of well-behaved cluster localizations whose intersection is the original cluster algebra.  This will allow results for the cluster localizations to be lifted to the original cluster algebra.

Given an $n\times n$ skew-symmetric matrix $\mathsf{A}$, $i\in \{1,...,n\}$ is a \textbf{sink} if $\mathsf{A}_{ji}\geq0$ for all $j$.  Similarly, a \textbf{source} is an index $i\in \{1,...,n\}$ such that $\mathsf{A}_{ji}\leq0$ for all $j\in \ex$.  %This only depends on the skew-symmetric submatrix $\pi\B$ of the quantum seed.

\begin{rem}
%The skew-symmetric matrix $\iota\B$ can be encoded in a \emph{quiver}, with vertex set $\ex$ and $\B_{ij}$-many arrows from $j$ to $i$ (where negative arrows are from $i$ to $j$).
In terms of the quiver $Q(\mathsf{A})$, a source is a vertex without outgoing arrows, and a sink is a vertex without incoming arrows.
\end{rem}

Sources and sinks are a source of pairs of cluster variables which generate $\A_q$.

\begin{lemma}\label{lemma: prime}
Let $(\B,\Lambda,M)$ be a quantum seed, with $i,j\in \ex$ such that $\B_{ij}\neq0$ and $i$ is a sink or a source in $\pi\B$.  Then the cluster variables $M(e_i)$ and $M(e_j)$ generate the trivial left ideal in $\A_q(\B,\Lambda,M)$.
\end{lemma}
\begin{proof}
Assume $i$ is a sink (the other case is similar); this implies $\B_{ji}>0$.
If $(\B',\Lambda',M')$ is the mutation of the original seed at $i$, then%\footnote{Here, $q^\bullet$ denotes some half-power of $q$ not worth keeping careful track of.}
\[ M'(e_{i})M(e_{i}) = q^\bullet M\left(\sum_{\B_{ki}>0}\B_{ki}e_k\right)+q^\bullet M\left(-\sum_{\B_{ki}<0}\B_{ki}e_k\right)\]
Since we may always factor a monomial $M(\alpha)$ into $q^\bullet M(\alpha-\beta)M(\beta)$ for any $\alpha$ and $\beta$, we move the first term on the right hand side to the left, and pull out an $M(e_j)$.
\[ q^\bullet M'(e_{i})M(e_{i})-q^\bullet M\left(-e_{j}+\sum_{\B_{ki}>0}\B_{ki}e_k\right)M(e_{j}) =  M\left(-\sum_{\B_{ki}>0}\B_{ki}e_k\right)\]
Since $\B_{ji}>0$, the left-hand side is in any left $\A_q$-ideal containing $M(e_i)$ and $M(e_j)$.  Since $i$ is a sink, $\B_{ki}<0$ implies that that $k\not\in \ex$, and so the right hand side is a monomial in non-exchangeable indices, which are invertible by construction.  It follows that the left $\A_q$-ideal generated by $M(e_i)$ and $M(e_j)$ is trivial.
%
%Let $J$ be a left ideal of $\A_q$ containing $M(e_{i})$ and $M(e_{j})$.
%The left hand side is in the ideal $J$.  If $j\in \ex$ and $\B_{i_2j}>0$, then $j\in I$, and so the right hand side is a monomial whose exchangeable indices are in $I$.
\end{proof}

\begin{rem}\label{rem: coveringpair}
This lemma is weaker than its commutative analog, \cite[Lemma 5.3]{MulLA}, which applies to any `covering pair', which generalizes the condition on $i$ and $j$.  There is no obvious quantum analog of the argument for this more general condition.
%
%which said that cluster variables corresponding to any covering pair \cite[Definition 5.2]{MulLA} are relatively prime.  The indices $i,j$ in the lemma are a covering pair, but not ever covering pair is of this form.
\end{rem}

\begin{lemma}\label{lemma: intersection}
If $M(e_i)$ and $M(e_j)$ generate $\A_q$ as a left ideal, then
\[\A_q[M(e_{i})^{-1}]\cap \A_q[M(e_{j})^{-1}]=\A_q\]
\end{lemma}
\begin{proof}
For any $x\in \A_q[M(e_{i})^{-1}]\cap \A_q[M(e_{j})^{-1}]$, let $n_x\in \N$ be the smallest positive integer such that, $\forall a,b\in N$ such that $a+b\geq n_x$, $M(ae_i+be_j)x\in \A_q$.  Such an $n_x$ exists; to see this, write $x=M(ce_i)^{-1}y=M(de_j)^{-1}z$ for $y,z\in \A_q$ and note that $n_x\leq c+d$.  Clearly, $n_x=0$ if and only if $x\in\A_q$.

For contradiction, assume there exists $x\not\in \A_q$ with $n_x$ minimal among elements of $(\A_q[M(e_{i})^{-1}]\cap \A_q[M(e_{j})^{-1}])-\A_q$.  For any $a,b\in \N$ with $a+b\geq n_x-1$,
\[ M(ae_i+be_i)[M(e_i)x]= M((a+1)e_i+be_j)x\Rightarrow M(ae_i+be_i)[M(e_i)x]\in \A_q\]
This implies that $n_{M(e_i)x}\leq n_x-1$.  Since $n_x$ was minimal, $M(e_i)x\in \A_q$.  By a symmetric computation, $M(e_j)x\in \A_q$.

Define the \emph{left denominator ideal} $I$ of $x$ by
\[ I :=\{y\in \A_q\,|\, yx\in \A_q\}\]
This is a left $\A_q$-ideal.   As has been observed, $M(e_i)$ and $M(e_j)$ are in $I$.  By Lemma \ref{lemma: prime}, $I=\A_q$. In particular, $1\in I$ and $1\cdot x\in \A_q$.  This contradicts $x\not\in\A_q$.
\end{proof}

\subsection{A lemma for proving $\A_q=\U_q$}

These techniques can be combined to give the following criterion for showing large classes of cluster algebras have $\A_q=\U_q$.

\begin{lemma}\label{lemma: A=Uclass}
Let $\mathcal{P}$ be a set of exchange types.  Assume that, for every non-isolated exchange type $\mathcal{T}\in\mathcal{P}$, there is a skew-symmetric matrix $\mathsf{A}\in \mathcal{T}$, and indices $i,j$ such that...
\begin{enumerate}
\item $\mathsf{A}_{ij}\neq0$ and $i$ is either a source or a sink in $\mathsf{A}$, and
\item the exchange types of the freezings $\mathsf{A}^{(i)}$ and $\mathsf{A}^{(j)}$ are both in $\mathcal{P}$.
\end{enumerate}
%there is a quantum seed $(\B,\Lambda,M)$ of $\A_q$ and indices $i,j\in\ex$ such that $\B_{ij}\neq0$ and $i$ is either a source or a sink, and the freezings $\A_q^{(i)}$ and $\A_q^{(j)}$ at $i$ and $j$ are both in $\mathcal{P}$.
Then $\A_q=\U_q$ for all $\A_q$ with exchange type in $\mathcal{P}$.
%\begin{enumerate}
%\item $\A_q^{(i)}$ and $\A_q^{(j)}$ are cluster localizations of $\A_q$,
%\item $\A_q^{(i)}\cap \A_q^{(j)}=\A_q$, and
%\end{enumerate}
\end{lemma}
\begin{proof}
Assume $\mathcal{P}$ non-empty; the alternative case is immediate.
%The \emph{rank} of a cluster algebra is $|\ex|\in \N$.  If $\mathcal{P}$ contains a cluster algebra $\A_q$ of rank $i>0$, then it also contains a cluster algebra of rank $i-1$.  Therefore, $\mathcal{P}$ contains a rank $0$ cluster algebra.

We proceed by induction on the size of $\mathcal{T}$; this is the size of any matrix in $\mathcal{T}$.  Let $\mathcal{T}\in \mathcal{P}$ have minimal size.  If it is not isolated, then there is some $\mathsf{A}\in \mathcal{T}$ with a freezing $\mathsf{A}^{(i)}$ with exchange type in $\mathcal{P}$.  Since the size of $\mathsf{A}^{(i)}$ is less than the size of $\mathsf{A}$, this contracts minimality; so $\mathcal{T}$ is isolated.  Then $\A_q=\U_q$ for any $\A_q$ of type $\mathcal{T}$ by Proposition \ref{prop: isolated}.

Assume that $\A_q=\U_q$ for every $\A_q$ of type $\mathcal{T}\in \mathcal{P}$ with size $<n$.  Let $\mathcal{T}\in\mathcal{P}$ be an exchange type of size $n$, and let $\A_q$ be a cluster algebra of type $\mathcal{T}$.  If $\A_q$ is isolated, then $\A_q=\U_q$.

Else, let $\mathsf{A}\in \mathcal{T}$ be the matrix and $i,j$ be the indices guarenteed by the hypothesis.  Since $\A_q$ has type $\mathcal{T}$, there is a quantum seed $(\B,\Lambda,M)$ of $\A_q$ such that $\pi\B=\mathsf{A}$, and we identify $i,j$ with indices in $\ex$.  Then the freezings $\A_q^{(i)}$ and $\A_q^{(j)}$ are of type $\mathsf{A}^{(i)}$ and $\mathsf{A}^{(j)}$ respectively.  These exchange types are in $\mathcal{P}$ and so by the inductive hypothesis, $\A_q^{(i)}=\U_q^{(i)}$ and $\A_q^{(j)}=U_q^{(j)}$.  Then the inclusions in Proposition \ref{prop: inclusions} are equalities; in particular,
\[ \A_q^{(i)} = \A_q[M(e_i)^{-1}],\text{ and } \A_q^{(j)} = \A_q[M(e_j)^{-1}]\]
By Lemma \ref{lemma: prime}, $M(e_i)$ and $M(e_j)$ generate $\A_q$ as a left ideal, so by Lemma \ref{lemma: intersection},
\[ \U_q\subseteq\U_q^{(i)}\cap\U_q^{(j)}=\A_q^{(i)}\cap \A_q^{(j)}=\A_q[M(e_i)^{-1}]\cap\A_q[M(e_j)^{-1}]=\A_q\]
But $\A_q\subseteq \U_q$, so $\A_q=\U_q$.  By induction, this is true for all $\mathcal{T}\in \mathcal{P}$.
%
%We proceed by induction on the \emph{rank} of $\A_q$; this is the number $|\ex|$.  First, let $\A_q\in \mathcal{P}$ have minimal rank, among cluster algebras in $\mathcal{P}$.  If $\A_q$ was not acyclic, then there would a freezing $\A_q^{(i)}\in \mathcal{P}$, which has strictly lower rank.  This is impossible, so $\A_q$ is acyclic.  By Theorem \ref{thm: BFZ}, $\A_q=\U_q$.
%
%Assume that, for every cluster algebra $\A_q\in \mathcal{P}$ of rank $<n$, $\A_q=\U_q$.  Let $\A_q\in \mathcal{P}$ be a cluster algebra of rank $n$.  If $\A_q$ is acyclic, then $\A_q=\U_q$.
%
%If $\A_q$ is not acyclic, then let $i,j\in \ex$ be the two indices in a quantum seed guarenteed by the hypothesis.  The freezings $\A_q^{(i)}$ and $\A_q^{(j)}$ are both in $\mathcal{P}$, and of strictly lower rank, so $\A_q^{(i)}=\U_q^{(i)}$ and $\A_q^{(j)}=\U_q^{(j)}$.  By Proposition \ref{prop: inclusions},
%\[ \A_q^{(i)} = \A_q[M(e_i)^{-1}],\text{ and } \A_q^{(j)} = \A_q[M(e_j)^{-1}]\]
%By Lemma \ref{lemma: prime}, $M(e_i)$ and $M(e_j)$ generate $\A_q$ as a left ideal, so by Lemma \ref{lemma: intersection},
%\[ \U_q\subseteq\U_q^{(i)}\cap\U_q^{(j)}=\A_q^{(i)}\cap \A_q^{(j)}=\A_q[M(e_i)^{-1}]\cap\A_q[M(e_j)^{-1}]=\A_q\]
%By the Laurent phenomenon, $\A_q\subseteq \U_q$, so $\A_q=\U_q$.
\end{proof}

\begin{rem}\label{rem: A=Ucomm}
The above lemma is a weaker version of the Banff algorithm which appeared in \cite[Section 5]{MulLA}, reformulated as an criterion rather than an algorithm.  Specifically, if the condition $(2)$ in the lemma was replaced by the weaker condition `$(i,j)$ is a covering pair %\footnote{For the definition of a covering pair, see \cite[Definition 5.2]{MulLA}.}
 in $\mathsf{A}$', then a set $\mathcal{P}$ satisfies the hypothesis of the new version of the lemma if and only if the Banff algorithm produces an acyclic cover for every commutative cluster algebra $\A_1$ with exchange type in $\mathcal{P}$.  As a consequence, if $\mathcal{P}$ satisfies the lemma as it is stated above, every commutative cluster algebra $\A_1$ with exchange type in $\mathcal{P}$ is locally acyclic (see Section \ref{section: LA}).
%
%, if $\A_q=\U_q$ by the above lemma, then any commutative cluster algebra $\A'$ with the same exchange type as $\A_q$ is locally acyclic and has $\A'=\U'$ (\cite[Theorems 5.5 and 4.1]{MulLA}).
\end{rem}

\begin{rem}
The union $\overline{\mathcal{P}}$ of all sets $\mathcal{P}$ which satisfy the lemma also satisfies the lemma, so $\overline{\mathcal{P}}$ is the unique maximal set of exchange types satisfying the lemma.  Are there any cluster algebras $\A_q$ with $\A_q=\U_q$ and exchange type not in $\overline{\mathcal{P}}$?  %By the previous remark, $\overline{\mathcal{P}}$ is contained in the set of exchange types for which the Banff algorithm produces an acyclic cover.%This maximal set is mysterious, but we can ask several questions.
%\begin{itemize}
%\item Is there an $\A_q$ with $\A_q=\U_q$ but with exchange type not in $\overline{\mathcal{P}}$?
%\item Is $\overline{P}$ the set of exchange types for which the Banff algorithm produces an acyclic cover?
%\item Is $\overline{P}$ the set of exchange types of locally acyclic commutative cluster algebras?
%\end{itemize}
\end{rem}

\subsection{Digression: acyclic cluster algebras}

A $n\times n$ skew-symmetric matrix $\mathsf{A}$ is called \textbf{acyclic} if there is no sequence of indices $i_1,i_2,...,i_n=i_1$ such that $\mathsf{A}_{i_{j+1}i_{j}}>0$.  An exchange type is acyclic if any matrix in it is.

\begin{rem}
The matrix $\mathsf{A}$ is acyclic iff $Q(\mathsf{A})$ has no directed cycles.
\end{rem}

Cluster algebras of acyclic type are an important class of examples, for which many general results are known.  A byproduct of Lemma \ref{lemma: A=Uclass} is a new proof that $\A_q=\U_q$ for acyclic cluster algebras, which first appeared in \cite[Theorem 7.5]{BZ05}.
\begin{prop}\label{prop: acyclic}
If $\A_q$ has acyclic exchange type, then $\A_q=\U_q$.
\end{prop}
\begin{proof}
If $\mathsf{A}$ is acyclic and not zero, then there is some $i$ which is a sink, and $j$ with $\mathsf{A}_{ji}<0$.  This can be shown by starting at a non-isolated vertex in $Q(\mathsf{A})$ and moving along arrows; eventually a dead-end is reached because the index set is finite and cycles are forbidden.  The freezings $\mathsf{A}^{(i)}$ and $\mathsf{A}^{(j)}$ are also acyclic.  Therefore, the class of acyclic exchange types satisfies the hypothesis of Lemma \ref{lemma: A=Uclass}.
\end{proof}

\begin{rem}
This is of limited usefulness for cluster algebras of marked surfaces, because $\A_q(\S)$ has acyclic exchange type only for certain simple surfaces (see \cite[Remark 10.11]{FST08}).
\end{rem}

\section{$\A_q(\S)=\U_q(\S)$ for (most) marked surfaces.}

The techniques of the previous section can now be applied to the class of triangulable marked surfaces with at least two marked points on each connected component.

\subsection{Marked surfaces with isolated cluster algebras}

The first step is to characterize which cluster algebras of marked surfaces have isolated exchange type.

\begin{prop}\label{prop: topiso}
If $\S$ is a union of topological discs, each with $3$ or $4$ marked points, then $\A_q(\S)$ has isolated exchange type.
\end{prop}
\begin{proof}
Let $\Delta$ be a triangulation of $\S$.  The only non-boundary curves in $\Delta$ will be diagonals across connected components with 4 marked points.  Since any two such curves $\curve[x],\curve[y]$ are in different components, $\Q^{\Delta}_{\curve[x],\curve[y]}=0$, and so $\pi\B^{\Delta}=0$.
\end{proof}
\begin{rem}
These are the only triangulable marked surfaces whose cluster algebras have isolated exchange type.\footnote{We emphasize that this is not true in the larger generality of marked surfaces with non-boundary marked points (not considered in this paper, but cover in \cite{FST08}.}%, but we do not need this fact.
\end{rem}

\subsection{Cutting a marked surface} Freezing a quantum seed $(\B^\Delta,\Lambda^\Delta,M^\Delta)$ can be interpreted as the topological action of `cutting', at least on the level of the skew-symmetric matrix $\pi\B^{\Delta}$.

Let $\curve$ be a simple non-boundary arc\footnote{Cutting at any simple curve may be defined, but the resulting marked surface will only be triangulable for simple non-boundary arcs, so we do not consider the more general case.} in $\S$.  The \textbf{cutting} $\chi_{\curve}(\S)$ of $\S$ along $\alpha$ is the marked surface obtained by cutting $\S$ along $\curve$, compactifying $\S$ by adding boundary along the two sides of $\curve$, and adding marked points where the endpoints of $\curve$ were.  There is a natural map
\[ \chi_{\curve}(\S)\rightarrow \S\]
which is a bijection away from $\curve\subset \S$, a 2-to-1 map over the interior of $\curve$, and such that the preimage of marked points are all marked.  The two types of cut are pictured in Figure \ref{fig: cuts}.
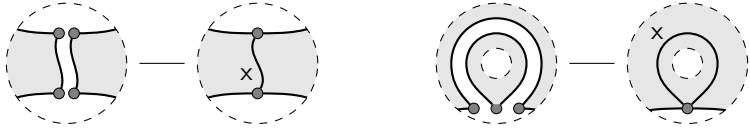
\begin{figure}[h!]
\centering
\begin{tikzpicture}[inner sep=0.5mm,auto]
\begin{scope}[xshift=0in]
\begin{scope}[xshift=-.5in,scale=.2]
	\begin{scope}
	\clip (0,0) circle (4);
	\draw[fill=black!10,thick] (-5,-3) to [in=180,out=30] (-.5,-2) to [in=240,out=60] (-.5,2) to [out=180,in=-30] (-5,3);
	\draw[fill=black!10,thick] (5,3) to [in=0,out=210] (.5,2) to [in=60,out=240] (.5,-2) to [out=0,in=150] (5,-3);
	\node (1) at (-.5,-2) [marked] {};
	\node (2) at (-.5,2) [marked] {};
	\node (3) at (.5,-2) [marked] {};
	\node (4) at (.5,2) [marked] {};
	\end{scope}
	\draw[dashed] (0,0) circle (4);
\end{scope}
	\draw (-.3,0) to [-angle 90] (.3,0);
\begin{scope}[xshift=.5in,scale=.2]
	\begin{scope}
	\clip (0,0) circle (4);
	\draw[fill=black!10,thick] (-5,-3) to [in=180,out=30] (0,-2) to [in=150,out=0] (5,-3) to [line to] (5,3) to [in=0,out=210] (0,2) to [in=-30,out=180] (-5,3);
	\node (1) at (0,-2) [marked] {};
	\node (2) at (0,2) [marked] {};
	\draw[thick] (1) to [in=240,out=60] node {$\curve$} (2);
	\end{scope}
	\draw[dashed] (0,0) circle (4);
\end{scope}
\end{scope}
\begin{scope}[xshift=2.25in]
\begin{scope}[xshift=-.5in,scale=.2]
	\begin{scope}
	\clip (0,0) circle (4);
	\draw[fill=black!10,thick] (-5,-4) to [in=180,out=30] (-1.5,-3) to [in=270,out=135] (-3,0) to [in=180,out=90] (0,3) to [in=90,out=0] (3,0) to [in=45,out=270] (1.5,-3) to [in=150,out=0] (5,-4) to [line to] (5,5) to (0,5) to (-5,5);
	\node (1) at (-1.5,-3) [marked] {};
	\node (2) at (0,-3) [marked] {};
	\node (3) at (1.5,-3) [marked] {};
	\draw[fill=black!10,thick] (2) to [out=45,in=270] (2,0) to [out=90,in=0] (0,2) to [out=180,in=90] (-2,0) to [out=270,in=135] (2);
	\draw[fill=black!0,dashed] (0,0) circle (1);
	\end{scope}
	\draw[dashed] (0,0) circle (4);
\end{scope}
	\draw (-.3,0) to [-angle 90] (.3,0);
\begin{scope}[xshift=.5in,scale=.2]
	\begin{scope}
	\clip (0,0) circle (4);
	\draw[fill=black!10,thick] (-5,-4) to [in=180,out=30] (0,-3) to [in=150,out=0] (5,-4) to [line to] (5,5) to (0,5) to (-5,5);
	\node (1) at (0,-3) [marked] {};
	\draw[thick] (1) to [out=45,in=270] (2,0) to [out=90,in=0] (0,2) to [out=180,in=90] (-2,0) to [out=270,in=135] (1);
	\draw[fill=black!0,dashed] (0,0) circle (1);
	\node (x) at (-2,2) {$\curve$};
	\end{scope}
	\draw[dashed] (0,0) circle (4);
\end{scope}
\end{scope}
\end{tikzpicture}
\caption{Types of cuts}
\label{fig: cuts}
\end{figure}

The map $\chi_{\curve}(\S)\rightarrow \S$ takes a triangulation of $\chi_{\curve}(\S)$ to a triangulation of $\S$ which contains $\curve$.  This induces a bijection between triangulations of $\chi_{\curve}(\S)$ and triangulations of $\S$ which contain $\curve$.

\begin{prop}\label{prop: topfreeze}
Let $\curve$ be a simple non-boundary arc in $\S$.  Let $\Delta$ be a triangulation of $\S$ containing $\curve$, and $\Delta'$ be the corresponding triangulation of $\chi_{\curve}(\S)$.  Then the skew-symmetric matrix $\pi\B^{\Delta'}$ is the submatrix $(\pi\B^{\Delta})^{(\curve)}$ of $\pi\B^{\Delta}$ where the row and column corresponding to $\curve$ has been removed.
%
%Then the freezing of $(\B^\Delta,\Lambda^\Delta,M^\Delta)$ at $\curve$ has the same exchange type as $(\B^{\Delta'},\Lambda^{\Delta'},M^{\Delta'})$.
\end{prop}
\begin{proof}
Let $\curve[y],\curve[z]\in \Delta'$.  Then
\[(\pi\B^{\Delta'})_{\curve[y],\curve[z]}=\Q^{\Delta'}_{\curve[y],\curve[z]} = \Q^{\Delta}_{\curve[y],\curve[z]} = (\pi\B^\Delta)_{\curve[y],\curve[z]}\]
The set $\ex'\subset\Delta'$ of non-boundary arcs is $\ex-\{\curve\}$, so $\pi\B^{\Delta'}$ is the restriction of $\pi\B^{\Delta}$ away from $\curve$.
\end{proof}

\begin{coro}
Let $\curve$ be a simple non-boundary arc in $\S$, and let $\Delta$ be a triangulation of $\S$ containing $\curve$.  Then $\A_q(\chi_{\curve}(\S))$ has the same exchange type as $\A_q(\S)^{(\curve)}$, the freezing of $\curve$ in the quantum seed corresponding to $\Delta$.
\end{coro}

\begin{rem}
It is not true that $\A_q(\chi_{\curve}(\S))=\A_q(\S)^{(\curve)}$.  The induced triangulation $\Delta'$ of $\chi_{\curve}(\S)$ has one more element than $\Delta$, and so the cluster algebras in question do not have isomorphic skew-fields of fractions.
\end{rem}

\subsection{Finding relatively prime elements}

We now topologically characterize pairs of cluster variables (ie, simple arcs) which satisfy Lemma \ref{lemma: prime}.

\begin{lemma}\label{lemma: topprime}
Let $\curve[x], \curve[y],\curve[z]$ be non-crossing, simple arcs in $(\S,\M)$ as in Figure \ref{fig: typeA}, with the endpoints of $\curve[y]$ distinct and $\curve[x],\curve[y]$ non-boundary.\footnote{Other pairs of marked points may coincide, and $\curve[x]$ and $\curve[z]$ may coincide.}  Then, for any triangulation $\{\curve[x],\curve[y],\curve[z]\}\subset\Delta$ of $\S$, $\curve[y]$ is a sink of the matrix $(\pi\B^{\Delta})$ with $\B^\Delta_{\curve[y]\curve[x]}>0$.
\end{lemma}
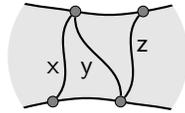
\begin{figure}[h!]
\begin{tikzpicture}[inner sep=0.5mm,auto,scale=.75]
	\useasboundingbox (-2.5,-.8) rectangle (2.5,1);
\begin{scope}[scale=.4]
\begin{scope}
	\begin{scope}
	\clip (0,0) circle (4);
	\draw[fill=black!10,thick] (-5,-3) to [in=190,out=30] (-2,-2) to [in=170,out=10] (1,-2) to [in=150,out=-10] (5,-3) to [line to] (5,3) to [in=10,out=210] (2,2) to [in=-10,out=190] (-1,2) to [in=-30,out=170] (-5,3);
	\node (1) at (1,-2) [marked] {};
	\node (2) at (-1,2) [marked] {};
	\node (3) at (-2,-2) [marked] {};
	\node (4) at (2,2) [marked] {};
	\draw[thick] (1) to [in=150,out=-30,relative] node {$\curve[y]$} (2);
	\draw[thick] (2) to [in=150,out=-30,relative] node[swap] {$\curve[x]$} (3);
	\draw[thick] (1) to [in=150,out=-30,relative] node[swap] {$\curve[z]$} (4);
	\end{scope}
%	\draw[dashed] (0,0) circle (4);
\end{scope}
\end{scope}
\end{tikzpicture}
\caption{A configuration of arcs.}\label{fig: typeA}
\end{figure}

\begin{proof}
%The only arcs in $\Delta$ immediately counterclockwise to $\beta$ will be boundary arcs, which correspond to frozen indices.  Then $\Q^\Delta_{

In any triangulation $\Delta$ of $(\S,\M)$ which contains $\curve[x],\curve[y]$ and $\curve[z]$, there can be no non-boundary arcs immediate clockwise or counterclockwise to $\curve[y]$ other than $\curve[x]$ and $\curve[z]$.  Therefore, $\Q^\Delta_{\curve}$  At each end of $\curve[y]$, there will be no arcs in $\Delta$ which are counter-clockwise to $\curve[y]$, and so there are no arrows out of $\curve[y]$ in $\Q_\Delta$.
\end{proof}

\subsection{Proving $\A_q=\U_q$ for most marked surfaces}

We are now in a position to prove that $\A_q(\S)=\U_q(\S)$ for a many marked surfaces.

\begin{thm}\label{thm: markedA=U}
If $\A_q$ is a cluster algebra with the same exchange type as $\A_q(\S)$ for $\S$ a triangulable marked surfaces with at least two marked points in each connected component, then $\A_q=\U_q$.
%
%Let $\mathcal{P}$ be the set of exchange types of cluster algebras $\A_q(\S)$ as $\S$ runs over triangulable marked surfaces with at least two marked points in each components.  If $\A_q$ has exchange type
\end{thm}
\begin{proof}
Let $\mathcal{P}$ be the set of exchange types comings from such marked surfaces.
%We observe that this class of exchange types is closed under all freezings.  Given $\mathsf{A}\in \mathcal{T}\in\mathcal{P}$ and an index $i$, there is a marked surface $\S$ with triangulation $\Delta$ and a simple non-boundary arcs $\curve$ so that $\pi\B^\Delta=\mathsf{A}$ and $\curve$ corresponds to $i$.  Then, by Proposition \ref{prop: topfreeze}, the freezing \[\mathsf{A}^{(i)}=(\pi\B^{\Delta})^{(i)}=(\pi\B^{\Delta'})\]
%where $\Delta'$ is the induced triangulation on the cutting $\chi_{\curve}(\S)$.  Since the cutting $\chi_{\curve}(\S)$ is still a marked surface with at least two marked points in each component, the exchange type of $\mathsf{A}^{(i)}$ is in $\mathcal{P}$.
We show $\mathcal{P}$ satisfies the hypothesis of Lemma \ref{lemma: A=Uclass}.
Let $\mathcal{T}$ be an exchange type in $\mathcal{P}$, and let $\S$ be such that $\A_q(\S)$ has exchange type $\mathcal{T}$.
If every connected component of $\S$ is a disc $3$ or $4$ marked points, then $\mathcal{T}$ is isolated (Proposition \ref{prop: topiso}).

Otherwise, choose a connected component $\S_0$ of $\S$ which is not a disc with $3$ or $4$ marked points.  %By hypothesis, $\S_0$ has at least two marked points; choose a simple non-boundary arc $\curve[y]$ connecting them.
Choose a simple non-bounadry arc $\curve[y]$ with distinct endpoints (by hypothesis, $\S_0$ has at least two marked points).
There exists non-crossing simple arcs $\curve[x]$ and $\curve[z]$ (which may coincide) so that $\curve[x],\curve[y],\curve[z]$ are as in Figure \ref{fig: typeA}.
The curves $\curve[x]$ and $\curve[z]$ cannot both be boundary arcs, since that would force $\S_0$ to be a disc with $4$ marked points.  Assume $\curve$ is a non-boundary arc (the other case is identical).

Choose a triangulation $\Delta$ containing $\curve[x],\curve[y],\curve[z]$.  By Proposition \ref{prop: topfreeze}, the freezing $(\pi\B^{\Delta})^{(\curve)}=(\pi\B^{\Delta'})$
where $\Delta'$ is the induced triangulation on the cutting $\chi_{\curve}(\S)$.  Since the cutting $\chi_{\curve}(\S)$ is still a marked surface with at least two marked points in each connected component, the exchange type of $(\pi\B^{\Delta})^{(\curve)}$ is in $\mathcal{P}$.  By an identical argument, the exchange type of $(\pi\B^{\Delta})^{(\curve[y])}$ is in $\mathcal{P}$.

Then $\pi\B^\Delta\in \mathcal{T}$ is non-isolated, with indices $\curve[x],\curve[y]$ such that...
\begin{enumerate}
\item $(\pi\B^\Delta)_{\curve[x],\curve[y]}>0$ and $\curve[y]$ is a sink of $\pi\B^\Delta$ (by Lemma \ref{lemma: topprime}), and
\item $(\pi\B^\Delta)^{(\curve)}$ and $(\pi\B^\Delta)^{(\curve)}$ have exchange type in $\mathcal{P}$.
\end{enumerate}
Thus, $\mathcal{P}$ satisfies Lemma \ref{lemma: A=Uclass}.
%By Corollary \ref{coro: topiso}, the freezing $\A_q(\pi\B^{\Delta})^{(\curve[x])}$ has the same exchange type as the cutting $\A_q(\chi_{\curve[x]}(\S))$.  The cutting $\chi_{\curve}(\S)$ is still a triangulable marked surface with at least two marked points in each component, and so the exchange type of $\A_q(\pi\B^{\Delta})^{(\curve[x])}$ and $\A_q(\chi_{\curve[x]}(\S))$ is in $\mathcal{P}$.
%
%\begin{enumerate}
%\item $(\pi\B^\Delta)_{\curve[x],\curve[y]}>0$ and $\curve[y]$ is a sink of $\pi\B^\Delta$ (by Lemma \ref{lemma: topprime}), and
%\item
%\end{enumerate}
%
%.  By Lemma \ref{lemma: topprime}, the indices $\curve[x],\curve[y]$
%
%By Corollary \ref{coro: topcut}, the freezings
%
%The cuttings $\chi_{\curve}(\S)$ and $\chi_{\curve[y]}(\S)$ are both triangulable marked surfaces with
\end{proof}

The localized skein algebra $\Sk_q^o(\S)$ is between $\A_q(\S)$ and $\U_q(\S)$, so they coincide.
\begin{thm}\label{thm: main2}
If $\S$ is triangulable and has at least two marked points in each connected component, then
\[\A_q(\S)=\Sk_q^o(\S)=\U_q(\S)\]
\end{thm}
\begin{proof}
This is an immediate consequence of Theorems \ref{thm: main1} and \ref{thm: markedA=U}.
\end{proof}

\begin{rem}
One immediate advantage of this theorem is computational.  Computations in cluster algebras can be quite difficult, for several reasons.  Working with expressions in $\A_q$ in different seeds requires choosing an explicit sequence of mutations relating the seeds, and the complexity grows rapidly with the number of mutations.  The upper cluster algebra $\U_q$ does not come with a generating set, and so working with general elements can be daunting.

The localized skein algebra is much easier to work with.  Elements are expressed in terms of topological objects which fit on a piece of paper.  The skein relations are local, and links may be freely homotoped; both of which keep complexity low.
\end{rem}

%
%\begin{proof}[Proof of Theorem \ref{thm: main2}]
%Let $\mathcal{P}$ be the class of cluster algebras with the same exchange type as $\A_q(\S)$, as $\S$ runs over all triangulable marked surfaces with at least two marked points.  We show that $\mathcal{P}$ satisfies the hypothesis of Lemma \ref{lemma: A=Uclass}.
%
%Let $\S$ be a triangulable marked surface with at least two marked points.  If $\S$ is a disc with $3$ or $4$ marked points, then $\A_q(\S)$ is acyclic
%
%There exists a simple non-boundary arc $\beta$ connecting distinct marked points.\footnote{Recall a disc with two marked points is defined to be `not triangulable'.}  Find the unique $\alpha$ and $\gamma$ which complete Figure \ref{fig: typeA}.  By Lemma \ref{lemma: typeA}, $\B^\Delta_{\alpha\beta}>0$ and $\beta$ is a sink in the quantum seed $(\B^\Delta,\Lambda^\Delta,M^\Delta)$.
%
%
%\end{proof}
%
%\begin{lemma}
%Let $\curve$ be a simple curve in $\S$, and let $\chi_{\curve}(\S)$ be the cut of $\S$ along $\curve$.  Then $\A_q(\chi_{\curve}(\S))$ has the same exchange type as the freezing of $\A_q$ at $\curve$ in any quantum seed containing $\curve$.
%\end{lemma}
%
%\begin{thm}
%Let $\A_q$ be a quantum cluster algebra with the same exchange type as $\A_q(\S)$, for $\S$ a marked surface with at least 2 marked points.  Then $\A_q=\U_q$.
%\end{thm}
%
%\begin{thm}
%Let $\S$ be a triangulable marked surface with at least two marked points in each component.  Then
%\[ \A_q(\S)=\Sk_q^o(\S)=\U_q(\S)\]
%\end{thm}

\section{Loop elements}

\subsection{Loop elements in $\A_q(\S)$}

By definition, the subalgebra $\A_q(\S)\subset\Sk_q^o(\S)$ contains arcs and inverses to boundary arcs.  %Equality of these algebras is then equivalent to the statement that, for each simple loop $\curve{l}\in \S$, the element $[\curve[l]]\in \Sk_q^o(\S)$ can be written in terms of arcs and inverses to boundary arcs.
Therefore, the equality $\A_q(\S)=\Sk_q^o(\S)$ in Theorem \ref{thm: main2} is equivalent to the following proposition.
\begin{prop}\label{prop: loop}
Let $\S$ be a triangulable marked surface with at least two marked points in each connected component.  For each simple loop $\ell\in \S$,
\[ [\ell] = [\link[Y]]^{-1}\sum_i\lambda_i[\curve_{i,1}][\curve_{i,2}]...[\curve_{i,n_i}]\]
where $\link[Y]$ is a link of boundary arcs, each $\curve_{i,j}$ is an arc, and $\lambda_i\in \Zq$.
\end{prop}
\begin{proof}
Cluster variables in $\A_q(\S)$ correspond to arcs, and so products of cluster variables correspond to general links.  Frozen variables correspond to boundary arcs, and so a general element of $\A_q(\S)$ can be written in the above form.  By Theorem \ref{thm: main2}, this is equally true of all elements of $\Sk_q^o(\S)$.
\end{proof}
\noindent 
These expressions are distinct from the skew-Laurent expressions from Corollary \ref{coro: Laurent}; the arcs $\curve_{i,j}$ are allowed to cross each other, but there are no negative powers of non-boundary arcs.

Finding such an expression for a simple loop is typically very different from writing it as a skew-Laurent polynomial of the arcs in a triangulation.  First, while every curve has a unique expression as a skew-Laurent polynomial in the arcs of a triangulation, there will be many ways to write a simple loop as a polynomial in arcs divided by a monomial in boundary arcs.  Second, the author does not know of any direct algorithm to produce any such expression, analogous to Remark \ref{rem: Laurent} or the \emph{band graph} techniques found in \cite{MSW11}.

%\begin{rem}
%Writing down such an expression for a given simple loop seems difficult, in general. % Concrete computations do not suggest any specific form, or even a bound o
%\end{rem}

From a cluster algebraic perspective, these loop elements are compelling.  They are not an ingredient in the cluster structure on $\A_q(\S)$, but they are a useful tool in computations.  For example, a product of two simple arcs can have many crossings, and applying the Kauffman skein relation to each crossing may produce loops.

%This begs a natural question.
%\begin{quest}
%Given a simple loop $\ell$, how does one explicitly produce such an expression for $[\ell]$?
%\end{quest}
%The proof technique we have employed is of no help
%
% given a loop, what is this such an expression in terms of arcs?  The proof technique

%As was observed before, $\A_q(\S)$ is the $\Zq$-subalgebra of $\Sk_q^o(\S)$ generated by arcs and inverses to boundary arcs, essentially by definition.  The equality $\A_q(\S)=\Sk_q^o(\S)$ is equivalent to the fact that, for each simple loop $\curve[l]\in \S$, it is possible to write $[\curve[l]]\in \Sk_q^o(\S)$ as an expression only involving arcs.

\subsection{The $\Zq$-basis of weighted simple multicurves}

The localized skein algebra has a natural $\Z_q$-basis, given by the set $\Multi^o$ of weighted simple multicurves with positive weights on non-boundary curves (Proposition \ref{prop: basiso}).   Theorem \ref{thm: main2} implies this is also a basis for $\A_q(\S)$ and $\U_q(\S)$.
\begin{prop}
Let $\S$ be a triangulable marked surface with at least two marked points in each connected component.  Then $\Multi^o$ maps to a $\Zq$-basis of $\A_q(\S)$ and $\U_q(\S)$ under the map $\multi\rightarrow [\multi]$.
\end{prop}
The problem of finding natural bases for cluster algebras goes back to the origins of their study.  Commutative cluster algebras were discovered in the study of Lusztig's dual canonical basis for $\mathbb{C}[G]$ of a reductive group, as a conjectural method of explicitly producing classical limits of elements of the dual canonical basis \cite[Introduction]{FZ02}.  A good history and bibliography of recent work on bases of cluster algebras can be found in the Introduction to \cite{MSW11}.

%, and the connection between dual canonical bases for $\U_q\mathfrak{g}$ and cluster monomials was proven in \cite{Lam11}.

Some of this basis comes directly from the cluster structure.  If $\multi$ is a weighted simple multicurve without loops, then there is some triangulation $\Delta$ which contains every arc in $\multi$.  Then $\multi$ is a monomial in $\Delta$; in the language of cluster algebras, this is called a \emph{cluster monomial}\footnote{Some references regard the Laurent ring of frozen variables as coefficients, rather than as cluster variables (as we are).  In the former case, a cluster monomial would be a weighted simple multicurve without loops or boundary arcs, but the coefficient ring would be much larger.} in the seed corresponding to $\Delta$.

%\begin{rem}
%It has been shown in \cite{Lam11} that the subalgebra $U^+_v(w)$ of the quantum enveloping algebra of $\mathfrak{sl}_{n+1}$
%\end{rem}

The remaining basis elements contain loop elements.  As has been mentioned, loop elements are difficult to express as explicit elements of $\A_q(\S)$, and so these basis elements of $\A_q(\S)$ do not follow naively from the cluster structure.

\begin{rem}
In the specialization $\rq=1$, this basis automatically goes to a $\Z$-basis of the commutative cluster algebras $\A_1(\S)$ and $\U_1(\S)$.  However, this basis is \emph{not} a `canonically positive' (or `atomic') basis.  That is, an element $x\in \A_1(\S)$ can have a positive Laurent expression for each seed $\Delta$, without being a positive combination of the basis elements $\Multi^o$.

An alternative basis for $\A_1(\S)$ which may be canonically positive has been put forward in \cite{FG06}, \cite{Dup10} and \cite{MSW11}.  This basis is related to the basis of weighted simple multicurves, by replacing simple loops with multiplicity by a single loop with self-crossings.  Proofs of canonically positivity for some $S$ can be found in \cite{SZ04} and \cite{DT11}.
%
%This basis for $\A_1(\S)$ has been studied in \cite{DT11} (for annuli) and \cite{MSW12} (in general).  These sources also find related canonically positive basis by replacing loops with multiplicity $\geq2$ by a certain loop with self-crossings.
\end{rem}

%\subsection{Intersections of quantum tori}

%\subsection{Properties coming from $\Sk_q^o(\S)$}
%
%Loop elements
%
%Basis of simple multicurves
%
%
%
%\subsection{Properties coming from $\A_q(\S)$}
%
%\subsection{Properties coming from $\U_q(\S)$}
%
%%\subsection{Loop elements}
%%
%%\subsection{The
%
%\section{Algebraic consequences}
%
%\begin{itemize}
%\item Working with cluster variables in distant clusters
%\item Loop elements
%\item Intersection of quantum tori
%\item $\Sk_q^o(\S)$ has a natural $\Zq$-basis.  Relate to MSW
%\end{itemize}
%
%\subsection{Loop elements}

\section{The commutative specialization $q^{\frac{1}{2}}=1$}\label{section: LA}

In the specialization $\rq=1$, Theorem \ref{thm: main2} becomes equalities
\[ \A_1(\S)=\Sk_1^o(\S)=\U_1(\S)\]
This endows $\Sk_1^o(\S)$ with the structure of a commutative cluster algebra.

%Each of the three algebras $\A_q(\S)$, $\Sk_q(\S)$ and $\U_q(\S)$ has a commutative specialization at $q^{\frac{1}{2}}=1$, denoted by $\A_1(\S)$, $\Sk_1(\S)$ and $\U_1(\S)$, respectively.\footnote{Explicitly, these are the quotients by the ideals generated by $q^{\frac{1}{2}}-1$.}    The algebras $\A_1(\S)$ and $\U_1(\S)$ are commutative cluster algebras, and were studied in an earlier work \cite{MulLA}.

% may be specialized to $q^{\frac{1}{2}}=1$; that is,
%
%When the variable $q^{\frac{1}{2}}$ is specialized to $1$, the cluster algebras $\A_1(\S)$ and $\U_1(\S)$ are commutative cluster algebras.
%
%The next two sections will focus on the question of when $\A_q(\S)=\U_q(\S)$; by Theorem \ref{thm: main1}, this would imply that they are both equal to $\Sk_q^o(\S)$.  We first focus on the commutative case, when $q$ has been set to $1$.  In this case, all the ingredients for a complete answer already exist.

%This section asks when $\A_q(\S)=\U_q(\S)$; by Theorem \ref{thm: main1}, this would imply that they are both equal to $\Sk_q^o(\S)$.

\subsection{Geometry of commutative cluster algebras}

The equality $\A_1(\S)=\U_1(\S)$ was already shown in a previous work by the author \cite[Theorem 10.6]{MulLA}, using the idea of `\emph{local acyclicity}'.  This is a geometric notion which does not directly generalize to the quantum setting.\footnote{However, the techniques of Section \ref{section: A=U} are based on this geometric approach (Remark \ref{rem: A=Ucomm}).}

Given a cluster algebra $\A_q$, the specialization $\A_1$ is commutative and so it can be studied geometrically, by considering the scheme $Spec(\A_1)$.  If $\A_q^{(s)}$ is a cluster localization of $\A_q$, then $\A_1^{(s)}$ is localization of $\A_1$, and so
\[ Spec\left(\A_1^{(s)}\right)\subseteq Spec(\A_1)\]
is an open subscheme.

A collection $\{\A_1^{(s_i)}\}$ of cluster localizations  of $\A_1$ is a \emph{cover} if the corresponding open subschemes cover $Spec(\A_1)$.  If $\{\A_1^{(s_i)}\}$ is a cover of $\A_1$, then
\[ \A_1 = \bigcap_i \A_1^{(s_i)}\]
though the converse is not true in general.

%Algebraically, collection $\{\A_1^{(s_i)}\}$ of cluster localizations  of $\A_1$ is a \emph{cover} if...
%\begin{enumerate}
%\item $\A_1=\bigcap \A_1^{(s_i)}$, and
%\item for each prime ideal $P\subset \A_1$, there is some $i$ with $\A_1^{(s_i)}P\neq \A_1^{(s_i)}$.
%\end{enumerate}

\subsection{Local acyclicity}

%The corresponding commutative question (of when does $\A_1(\S)=\U_1(\S)$) already has a complete answer, which we review now.

Recall that an exchange type $\mathcal{T}$ is \emph{acyclic} if there is a skew-symmetric matrix $\mathsf{A}\in \mathcal{T}$ with no cycles.\footnote{A cycle is a list of indices $i_1,i_2,...,i_{n-1},i_n=i_1\in \ex$ such that $\B_{i_ji_{j+1}}>0$ for all $j$.}  If $\A$ has acyclic exchange type, then $\A=\U$ (\cite[Corollary 1.19]{BFZ05} or Proposition \ref{prop: acyclic} and Remark \ref{rem: A=Ucomm}).

This can be generalized, by checking acyclicity locally.
\begin{defn}\cite[Definition 3.9]{MulLA}
A commutative cluster algebra $\A$ is \textbf{locally acyclic} if it has a cover $\{\A^{(s_i)}\}$ by acyclic cluster localizations.
\end{defn}

%In \cite{BFZ05}, the authors proved that $\A_1=\U_1$ when $\A_1$ has acyclic exchange type (this is the commutative analog of Proposition \ref{prop: acyclic}).
%
%Recall a quantum seed $(\B,\Lambda,M)$ is \textbf{acyclic} if the skew-symmetric submatrix $\pi\B$ had no cycles.\footnote{A cycle is a list of indices $i_1,i_2,...,i_{n-1},i_n=i_1\in \ex$ such that $\B_{i_ji_{j+1}}>0$ for all $j$.  If $\Q$ is a quiver with skew-adjancency matrix $\iota\B$, then cycles correspond to cyclically oriented paths in $\Q$.}  In \cite{BFZ05}, the authors proved that $\A_1(\B)=\U_1(\B)$ when $\B$ was {acyclic}.
%
%In \cite{MulLA}, the author introduced the idea of \emph{locally acyclic} commutative cluster algebras.
%
%
%Commutative cluster algebras have  certain localizations (\emph{cluster localizations}) which are naturally themselves cluster algebras.  A \textbf{locally acyclic} (commutative) cluster algebra is one which admits a cover (in a geometric sense) by acyclic cluster algebras.

%Local acyclicity is useful here, because it is it broad enough to encompass the majority of triangulable surfaces, but still implies many strong results.

Marked surfaces $\S$ such that $\A_1(\S)$ is locally acyclic have been characterized.
%
%Also obtained is a characterization of when $\A_1(\S)$ is locally acyclic.
\begin{thm}\cite[Theorems 10.6, 10.10]{MulLA}
The cluster algebra $\A_1(\S)$ is locally acyclic if and only if $\S$ has at least two marked points in each connected component of $\S$.
\end{thm}
\begin{rem}
Marked surfaces in \cite{MulLA} are allowed to have interior marked points, so the statements there are more general.
\end{rem}
%\begin{rem}
%The paper \cite{MulLA} is in the larger generality of marked surfaces with interior marked points.
%\end{rem}

\subsection{Consequences}

Local acyclicity has several consequences.
\begin{prop} Let $\A$ be a locally acyclic commutative cluster algebra. Then
\begin{enumerate}
\item \cite[Theorem 4.1]{MulLA} $\A=\U$,
\item \cite[Theorem 4.2]{MulLA} $\A$ is finitely generated, integrally closed and locally a complete intersection, and
\item \cite[Theorem 7.7]{MulLA} $\mathbb{Q}\otimes \A$ is a regular domain.
\end{enumerate}
\end{prop}
These results can then be applied to commutative cluster algebras of marked surfaces, and the $\rq=1$ localized skein algebra.
\begin{coro}Let $\S$ be a triangulable marked surface with at least two marked points in each connected component.
\begin{enumerate}
\item $\A_1(\S) = \Sk_1^o(\S)=\U_1(\S)$,\footnote{This is to say; locally acyclic provides an alternative (though fundamentally the same) proof.}
\item $\Sk_1^o$ is finitely generated, integrally closed, and locally a complete intersection, and
\item $\mathbb{Q}\otimes\Sk_1^o(\S)$ is a regular domain.
\end{enumerate}
\end{coro}
As a consequence of the last fact (see \cite[Corollary 7.9]{MulLA}),
\begin{itemize}
\item $Spec(\mathbb{Q}\otimes\Sk_1^o(\S))$ is a smooth scheme,
\item $Hom(\Sk_1^o(\S),\mathbb{C})$ is a smooth complex manifold, and
\item $Hom(\Sk_1^o(\S),\mathbb{R})$ is a smooth real manifold.
\end{itemize}
Here, both $Hom$s are as rings.

\begin{rem}There is an open inclusion
\[Spec(\Sk_1^o(\S))\subset Spec(\Sk_1(\S))\]  By analogy with the cluster structure on double Bruhat cells in semisimple Lie groups, it seems possible that $Spec(\Sk_1^o(\S))$ is the `big cell' in some natural stratification of $Spec(\Sk_1(\S))$.  Ideally, this is a finite stratification by smooth affine schemes, whose coordinate rings are commutative cluster algebras.
%
%Analogous remarks hold for the other two geometric objects.
%
%In each of these geometric constructions, if we replace $\Sk_1^o(\S)$ with the full skein algebra $\Sk_1(\S)$, we get a geometric object which contains the version for $\Sk_1^o(\S)$ as a subobject.
\end{rem}

\section{Examples and non-examples}

\subsection{Marked discs}\label{section: disc}

Let $\S_n$ be the disc with $n$ marked points on the boundary.  A simple curve in $\S_n$ will always be homotopic to a chord $\curve_{a,b}$ connecting distinct marked points $a,b$, and so $\Sk_q(\S)$ is generated by the $\binom{n}{2}$-elements of the form $[\curve_{a,b}]$ (Corollary \ref{coro: gens}).  The relations are
\[ [\curve_{a,b}][\curve_{b,c}] = q[\curve_{b,c}][\curve_{a,b}],\;\;\; [\curve_{a,b}][\curve_{c,d}] = [\curve_{c,d}][\curve_{a,b}]\]
\[ [\curve_{a,c}][\curve_{b,d}] = q[\curve_{a,b}][\curve_{c,d}] +q^{-1} [\curve_{a,d}][\curve_{b,c}]\]
as $a,b,c,d$ run over distinct marked points in clockwise order around $\partial \S_n$.  The boundary arcs are the elements $[\curve_{a,b}]$ for $a,b$ adjacent on the boundary, and $\Sk_q^o(\S_n)$ is the Ore localization at this set.

The surface is triangulable when $n\geq3$, and so $\A_q(\S_n)=\Sk_q^o(\S_n)=\U_q(\S_n)$ (Theorem \ref{thm: main2}).  The cluster variables coincide with the set of chords $[\curve_{a,b}]$, with clusters corresponding to triangulations.

The commutative cluster algebra $\A_1(\S)$ is a basic example in cluster algebras; thorough investigations can be found in \cite[Section 2.1]{GSV10} and \cite[Section 3]{FZ03b}.  In our language, the main observation is that $\Sk_1(\S_n)$ coincides with the homogeneous coordinate ring $\mathcal{O}[Gr_\mathbb{C}(2,n)]$ of the Grassmannian $Gr_\mathbb{C}(2,n)$.  This isomorphism depends on an identification of the marked points with a basis of $\mathbb{C}^n$; a cluster variable $[\curve_{a,b}]$ then corresponds to the Pl\"ucker coordinate $p_{a,b}$.

In \cite{GL11}, Grabowski and Launois exhibit a quantum cluster algebra structure on the \emph{quantum Grassmannian} $\mathcal{O}_q[Gr(2,n)]$, a specific quantization of $\mathcal{O}[Gr_{\mathbb{C}}(2,n)]$. %(as defined in \cite{???}).
One might hope that the quantum Grassmannian would coincide with $Sk_q(\S_n)$.  However, this is impossible; the quantum Grassmannian depends on an identification of the basis elements with the set $\{1,2,..,n\}$; a cyclic permutation does not induce an automorphism of $\mathcal{O}_q[Gr(2,n)]$ \cite{LL11} (cf. \cite{Yak10}).
The skein algebra $\Sk_q(\S_n)$ has no such dependency.  Inspecting the quantum seeds in \cite[Section 3.1]{GL11} confirms that these are different quantizations of the same commutative cluster algebra.

\subsection{A marked annulus}

Let $\S$ be the annulus with a single marked point on each boundary component.  Let $\curve[a]$ and $\curve[b]$ denote the two boundary arcs, and let $\ell$ denote the unique simple loop (Figure \ref{fig: annulus1}).  The remaining simple curves are arcs connecting the two marked points; they may be parametrized by $\Z$ as follows.  Choose such an arc to be $\curve_0$, and define the rest by the conditions that $\curve_i$ and $\curve_{i+1}$ do not intersect, and both ends of $\curve_{i+1}$ are clockwise to both ends of $\curve_i$.  %Intuitively, $\curve_i$ is the simple arc with counterclockwise winding number $i$ around the

\begin{figure}
\begin{tikzpicture}
\begin{scope}[scale=.5]
	\draw[fill=black!10] (-2,-2) to (-2,2) to (2,2) to (2,-2) to (-2,-2);
	\draw[thick, dashed,draw=black!0](-2,-2) to (-2,2);
	\draw[thick, dashed,draw=black!0](2,-2) to (2,2);
	\draw[thick] (-2,0) to node[above] (l) {$\ell$} (2,0);
	\draw[thick] (-2,-1) to  [out=0,in=135] node[above] (a) {$\mathsf{a}$} (0,-2) to [out=45,in=180] (2,-1);
	\draw[thick] (-2,1) to [out=0,in=225] (0,2) to [out=315,in=180] node[below right] (b) {$\mathsf{b}$} (2,1);
	\node[inner sep=0.5mm,circle,draw,fill=black!50] (1) at (0,-2) {};
	\node[inner sep=0.5mm,circle,draw,fill=black!50] (2) at (0,2) {};
\end{scope}
\begin{scope}[xshift=1.5in,scale=.5]
	\draw[fill=black!10] (-2,-2) to (-2,2) to (2,2) to (2,-2) to (-2,-2);
	\draw[thick, dashed,draw=black!0](-2,-2) to (-2,2);
	\draw[thick, dashed,draw=black!0](2,-2) to (2,2);
	\draw[thick] (0,-2) to node[above right] (x0) {$\mathsf{x}_0$} (0,2);
	\draw[thick] (0,-2) to [out=60,in=210] (2,0);
	\draw[thick] (-2,0) to [out=30,in=240] (0,2);
	\node (x1) at (1.2,-1.2) {$\mathsf{x}_1$};
	\node[inner sep=0.5mm,circle,draw,fill=black!50] (1) at (0,-2) {};
	\node[inner sep=0.5mm,circle,draw,fill=black!50] (2) at (0,2) {};
\end{scope}
\end{tikzpicture}
\caption{Simple curves in $\S$ (The two dashed edges are identified).}
\label{fig: annulus1}
\end{figure}
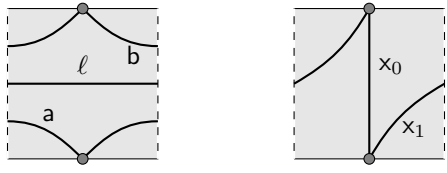

The simple curves $\curve[a],\curve[b],\ell$ and $\{\curve_i\}_{i\in \Z}$ generate $\Sk_q(\S)$ as a $\Zq$-algebra.  The elements $[\curve[a]]$ and $[\curve[b]]$ are central.  Some relations among these generators are
\[ [\ell][\curve_i]  = q[\curve_{i+1}]+q^{-1}[\curve_{i-1}]\]
\[ [\curve_i][\curve_{i+1}] = q^{-1}[\curve_{i+1}][\curve_i]\]
\[ [\curve_i][\curve_{i+2}] = [\curve[a]][\curve[b]] + 	q^{-2}[\curve_{i+1}]^2\]
\[ [\curve_i][\curve_{i+3}] = q[\ell] [\curve[a]][\curve[b]] + q^{-2} [\curve_{i+1}][\curve_{i+2}] \]
%The full set of relations among these generators are rather complicated, because two curves may cross each other many times, and require multiple applications of the Kauffman skein relation.%However, writing down a complete set of relations is daunting, because two curves $\curve_i$ and $\curve_j$ may have very high crossing number, and so require many applications of the Kauffman skein relation.
Since $[\curve_{i+1}] = q[\ell][\curve_{i}]-q^2[\curve_{i-1}]$, the five elements $\curve[a],\curve[b],\ell,\curve_0,\curve_1$ generate $\Sk_q(\S)$.  %The localized skein algebra $\Sk_q(\S)$ is the localization at $\curve[a]$ and $\curve[b]$.

The triangulations of $\S$ are the sets $\{\curve[a],\curve[b],\curve_i,\curve_{i+1}\}$ for some $i$.  Since $\S$ has two marked points, $\A_q(\S)=\Sk_q(\S)=\U_q(\S)$ (Theorem \ref{thm: main2}).

The loop element $[\ell]$ can be written as a skew-Laurent polynomial in any triangulation (Theorem \ref{thm: Laurent}),
\[ [\ell] = ([\curve_i][\curve_{i+1}])^{-1}\left(q[\curve_i]^2 + q^{-1}[\curve[a]][\curve[b]] + q^{-3}[\curve_{i+1}]^2\right)\]
and as a product of cluster variables divided by frozen variables (Proposition \ref{prop: loop}),
\[ [\ell] = ([\curve[a]][\curve[b]])^{-1}(q^{-1}[\curve_i][\curve_{i+3}]-q^{-3}[\curve_{i+1}][\curve_{i+2}])\]

\appendix

\section{Finite generation of $\Sk_q(\S)$.}

It has been shown by Bullock that $\Sk_q(\S)$ is finitely generated when $\S$ is unmarked \cite[Theorem 1]{Bul99}.  The idea of his proof still works in the marked case, with the necessary modifications.

\begin{rem}
What follows is a simplified version of Bullock's proof, since we will not explicitly bound the number of generators.
\end{rem}

%This proof can be straight-forwardly applied to the case of general marked surfaces, provided the necessary adaptations are made.

%Bullock's proof makes use of the handle decomposition in Figure \ref{fig: handle} (we may assume $\partial\S\neq\emptyset$; otherwise the original theorem applies directly).  Observe that the marked points of $\S$ can be assumed to lie on the boundary of the 0-handle.  Since

We assume $\partial\S\neq \emptyset$ (otherwise, Bullock's result applies directly).  The marked surface $\S$ has a handle decomposition (Figure \ref{fig: handle}); observe that every marked point can be placed on the boundary of the 0-handle.  %Since Bullock's proof relies only manipulations in the 1-handles and local manipulations in the 0-handle, most arguments will apply directly to links with endpoints.
\begin{figure}
\begin{tikzpicture}[scale=.3]
    \draw[thick,fill=black!10] (2,4) to (11,4) arc (90:0:1) to (12,-1) arc (360:270:1) to (2,-2);
    \draw[thick,fill=black!10] (-2,-2) to (-11,-2) arc (270:180:1) to (-12,3) arc (180:90:1) to (-2,4);
    \draw[thick,dashed] (-2,4) to (2,4);
    \draw[thick,dashed] (-2,-2) to (2,-2);
    \draw[draw=black!10,fill=black!10] (-5.5,1.5) arc (270:-90:2.5) to (-5.5,2.5) arc (-90:270:1.5);
    \draw[thick] (-8,4) arc (180:0:2.5);
    \draw[thick] (-7,4) arc (180:0:1.5);
    \draw[draw=black!10,fill=black!10] (-7.5,1.5) arc (270:-90:2.5) to (-7.5,2.5) arc (-90:270:1.5);
    \draw[thick] (-10,4) arc (180:0:2.5);
    \draw[thick] (-9,4) arc (180:0:1.5);
    \draw[draw=black!10,fill=black!10] (7.5,1.5) arc (270:-90:2.5) to (7.5,2.5) arc (-90:270:1.5);
    \draw[thick] (5,4) arc (180:0:2.5);
    \draw[thick] (6,4) arc (180:0:1.5);
    \draw[draw=black!10,fill=black!10] (5.5,1.5) arc (270:-90:2.5) to (5.5,2.5) arc (-90:270:1.5);
    \draw[thick] (3,4) arc (180:0:2.5);
    \draw[thick] (4,4) arc (180:0:1.5);
    \draw[draw=black!10,fill=black!10] (-6.5,-4.5) arc (270:-90:2.5) to (-6.5,-3.5) arc (-90:270:1.5);
    \draw[thick] (-9,-2) arc (180:360:2.5);
    \draw[thick] (-8,-2) arc (180:360:1.5);
    \draw[draw=black!10,fill=black!10] (6.5,-4.5) arc (270:-90:2.5) to (6.5,-3.5) arc (-90:270:1.5);
    \draw[thick] (4,-2) arc (180:360:2.5);
    \draw[thick] (5,-2) arc (180:360:1.5);
    \node[marked] (1) at (-12,1) {};
    \node[marked] (2) at (-6.5,-2) {};
    \node[marked] (3) at (6.5,-2) {};
\end{tikzpicture}
\caption{The handle decomposition of $\S$.  There are $g$-many pairs of 1-handles along the top, $h$-many 1-handles along the bottom, and any marked point may denote multiple close marked points (or none).}
\label{fig: handle}
\end{figure}
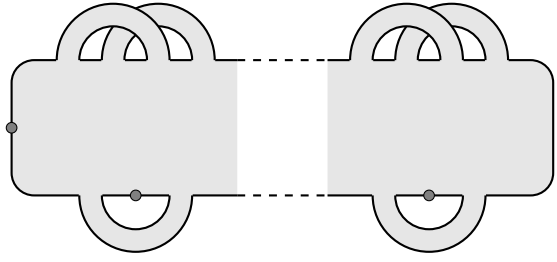

A link is in \emph{standard position} (with respect to the handle decomposition) if its intersection with any 1-handle is a union of strands homotopic to the core, and the number of strands is minimal with respect to homotopy. Every link is homotopic to one in standard position; for the remainder of the section we assume all links are in standard position.

The \emph{complexity} of a link is the total number of strands in the intersection with the 1-handles, minus the number of 1-handles it intersects.  So, a link has complexity zero if its intersection with any 1-handle contains at most one strand.

\begin{prop}
The set of simple curves of complexity zero is finite.
\end{prop}
\begin{proof}
Fix a subset $S$ of the 1-handles.   If $\curve$ is a simple curve of complexity zero which intersects exactly the 1-handles in $S$, then $\curve$ is determined by its intersection with the 0-handle.  The intersection of $\curve$ with the 0-handle is a non-crossing matching between the attaching points of the 1-handles in $S$, and either 2 or 0 marked points.  There are finitely many such non-crossing matchings, and finitely many subsets $S$ of the 1-handles, so the set of zero complexity simple curves is finite.
\end{proof}

\begin{lemma}
The set of simple curves of complexity zero generates $\Sk_q(\S)$.
\end{lemma}
\begin{proof}
We claim every simple curve $\curve$ in $\S$ is in the $\Zq$-subalgebra of $\Sk_q(\S)$ generated by simple curves of complexity zero.  The proof is by induction on complexity $\kappa$.  The case $\kappa=0$ is trivial.

Assume $\kappa\geq1$.  Then there is some 1-handle which $\curve$ intersects in multiple strands.  Choose the two innermost strands, and consider the following picture, where there may be additional components in the 1-handle.
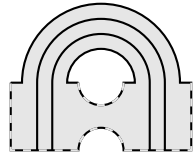
\begin{center}
\begin{tikzpicture}[scale=.3]
	\draw[thick,fill=black!10] (-4,0) to (-3.5,0) arc (180:0:3.5) to (4,0) to (4,-3) to (1,-3) arc (0:180:1) to (-4,-3) to (-4,0);
%	\draw[thick,fill=black!10] (-3,0) to (-2.5,0) arc (180:0:2.5) to (3,0) to (3,-3) to (1,-3) to (1,-2) to (-1,-2) to (-1,-3) to (-3,-3) to (-3,0);
	\draw[thick,fill=black!0] (-1,0) to (-1.5,0) arc (180:0:1.5) to (1,0) arc (360:180:1);
	\draw[thick,dashed,draw=black!10] (4,0) to (4,-3) to (1,-3) arc (0:180:1) to (-4,-3) to (-4,0);
	\draw[thick,dashed,draw=black!10] (1,0) arc (360:180:1);
	\draw[thick] (-2.16,-3) to (-2.16,0) arc (180:0:2.16) to (2.16,-3);
	\draw[thick] (-2.83,-3) to (-2.83,0) arc (180:0:2.83) to (2.83,-3);
\end{tikzpicture}
\end{center}

By repeated application of the Kauffman skein relation,
\begin{center}
\begin{tikzpicture}
\begin{scope}[scale=.2]
	\draw[thick,fill=black!10] (-4,0) to (-3.5,0) arc (180:0:3.5) to (4,0) to (4,-4) to (1,-4) arc (0:180:1) to (-4,-4) to (-4,0);
	\draw[thick,fill=black!0] (-1,0) to (-1.5,0) arc (180:0:1.5) to (1,0) arc (360:180:1);
	\draw[thick,dashed,draw=black!10] (4,0) to (4,-4) to (1,-4) arc (0:180:1) to (-4,-4) to (-4,0);
	\draw[thick,dashed,draw=black!10] (1,0) arc (360:180:1);
	\draw[thick] (-2.16,-4) to (-2.16,0) arc (180:0:2.16) to (2.16,-4);
	\draw[thick] (-2.83,-4) to (-2.83,0) arc (180:0:2.83) to (2.83,-4);
\end{scope}
	\node (=1) at (.45in,0) {$=$};
	\node (q11) at (.65in,.02in) {$-q^2$};
\begin{scope}[xshift =1.15in,scale=.2]
	\draw[thick,fill=black!10] (-4,0) to (-3.5,0) arc (180:0:3.5) to (4,0) to (4,-4) to (1,-4) arc (0:180:1) to (-4,-4) to (-4,0);
	\draw[thick,fill=black!0] (-1,0) to (-1.5,0) arc (180:0:1.5) to (1,0) arc (360:180:1);
	\draw[thick,dashed,draw=black!10] (4,0) to (4,-4) to (1,-4) arc (0:180:1) to (-4,-4) to (-4,0);
	\draw[thick,dashed,draw=black!10] (1,0) arc (360:180:1);
	\draw[thick] (-2.83,-4) to (-2.83,0) arc (180:105:2.83) arc (105:-75:.33) arc (105:180:2.16) to (-2.16,-4);
	\draw[thick] (2.83,-4) to (2.83,0) arc (0:75:2.83) arc (75:255:.33) arc (75:0:2.16) to (2.16,-4);
\end{scope}
	\node (+2) at (1.6in,0) {$+$};
	\node (q12) at (1.71in,0) {$q$};
\begin{scope}[xshift=2.15in,scale=.2]
	\draw[thick,fill=black!10] (-4,0) to (-3.5,0) arc (180:0:3.5) to (4,0) to (4,-4) to (1,-4) arc (0:180:1) to (-4,-4) to (-4,0);
	\draw[thick,fill=black!0] (-1,0) to (-1.5,0) arc (180:0:1.5) to (1,0) arc (360:180:1);
	\draw[thick,dashed,draw=black!10] (4,0) to (4,-4) to (1,-4) arc (0:180:1) to (-4,-4) to (-4,0);
	\draw[thick,dashed,draw=black!10] (1,0) arc (360:180:1);
	\draw[thick] (-2.83,-4) to (-2.83,0) arc (180:105:2.83) to[out=15,in=150] (0,2.5);
	\draw[thick] (2.16,-4) to (2.16,0) arc (0:75:2.16) to[out=165,in=-30] (0,2.5);
	\draw[thick,line width=1.5mm,draw=black!10] (-2.16,-4) to (-2.16,0) arc (180:105:2.16) to[out=15,in=210] (0,2.5);
	\draw[thick,line width=1.5mm,draw=black!10] (2.83,-4) to (2.83,0) arc (0:75:2.83) to[out=165,in=30] (0,2.5);
	\draw[thick] (-2.16,-4) to (-2.16,0) arc (180:105:2.16) to[out=15,in=210] (0,2.5);
	\draw[thick] (2.83,-4) to (2.83,0) arc (0:75:2.83) to[out=165,in=30] (0,2.5);
\end{scope}
\begin{scope}[yshift=-.8in]
		\node (=2) at (.45in,0) {$=$};
		\node (q21) at (.65in,.02in) {$-q^2$};
	\begin{scope}[xshift =1.15in,scale=.2]
		\draw[thick,fill=black!10] (-4,0) to (-3.5,0) arc (180:0:3.5) to (4,0) to (4,-4) to (1,-4) arc (0:180:1) to (-4,-4) to (-4,0);
		\draw[thick,fill=black!0] (-1,0) to (-1.5,0) arc (180:0:1.5) to (1,0) arc (360:180:1);
		\draw[thick,dashed,draw=black!10] (4,0) to (4,-4) to (1,-4) arc (0:180:1) to (-4,-4) to (-4,0);
		\draw[thick,dashed,draw=black!10] (1,0) arc (360:180:1);
		\draw[thick] (-2.83,-4) to (-2.83,0) arc (180:105:2.83) arc (105:-75:.33) arc (105:180:2.16) to (-2.16,-4);
		\draw[thick] (2.83,-4) to (2.83,0) arc (0:75:2.83) arc (75:255:.33) arc (75:0:2.16) to (2.16,-4);
	\end{scope}
		\node (+2) at (1.59in,0) {$+$};
		\node (q22) at (1.72in,.02in) {$q^2$};
	\begin{scope}[xshift=2.15in,scale=.2]
		\draw[thick,fill=black!10] (-4,0) to (-3.5,0) arc (180:0:3.5) to (4,0) to (4,-4) to (1,-4) arc (0:180:1) to (-4,-4) to (-4,0);
		\draw[thick,fill=black!0] (-1,0) to (-1.5,0) arc (180:0:1.5) to (1,0) arc (360:180:1);
		\draw[thick,dashed,draw=black!10] (4,0) to (4,-4) to (1,-4) arc (0:180:1) to (-4,-4) to (-4,0);
		\draw[thick,dashed,draw=black!10] (1,0) arc (360:180:1);
		\draw[thick] (-2.83,-4) to (-2.83,0) arc (180:105:2.83) to[out=15,in=150] (0,2.5);
		\draw[thick] (0,-2) to [out=30,in=180] (1,-1.66) to [line to] (1.16,-1.66) arc (270:360:1) to (2.16,0) arc (0:75:2.16) to[out=165,in=-30] (0,2.5);
		\draw[thick] (-2.16,-4) to (-2.16,-3.33) arc (180:90:1) to (-1,-2.33) to [out=0,in=210] (0,-2);
		\draw[thick,line width=1.5mm,draw=black!10] (0,-2) to [out=30,in=180] (-1,-1.66) to [line to] (-1.16,-1.66) arc (270:180:1) to (-2.16,0) arc (180:105:2.16) to[out=15,in=210] (0,2.5);
		\draw[thick,line width=1.5mm,draw=black!10] (2.83,-4) to (2.83,0) arc (0:75:2.83) to[out=165,in=30] (0,2.5);
		\draw[thick,line width=1.5mm,draw=black!10] (2.16,-4) to (2.16,-3.33) arc (0:90:1) to (1,-2.33) to [out=180,in=-30](0,-2);
		\draw[thick] (0,-2) to [out=30,in=180] (-1,-1.66) to [line to] (-1.16,-1.66) arc (270:180:1) to (-2.16,0) arc (180:105:2.16) to[out=15,in=210] (0,2.5);
		\draw[thick] (2.83,-4) to (2.83,0) arc (0:75:2.83) to[out=165,in=30] (0,2.5);
		\draw[thick] (2.16,-4) to (2.16,-3.33) arc (0:90:1) to (1,-2.33) to [out=180,in=-30](0,-2);
	\end{scope}
		\node (+3) at (2.59in,0) {$+$};
		\node (q23) at (2.72in,.02in) {$q^3$};
	\begin{scope}[xshift=3.15in,scale=.2]
		\draw[thick,fill=black!10] (-4,0) to (-3.5,0) arc (180:0:3.5) to (4,0) to (4,-4) to (1,-4) arc (0:180:1) to (-4,-4) to (-4,0);
		\draw[thick,fill=black!0] (-1,0) to (-1.5,0) arc (180:0:1.5) to (1,0) arc (360:180:1);
		\draw[thick,dashed,draw=black!10] (4,0) to (4,-4) to (1,-4) arc (0:180:1) to (-4,-4) to (-4,0);
		\draw[thick,dashed,draw=black!10] (1,0) arc (360:180:1);
		\draw[thick] (-2.83,-4) to (-2.83,0) arc (180:105:2.83) to[out=15,in=150] (0,2.5);
		\draw[thick] (0,-1.66) to (1.16,-1.66) arc (270:360:1) to (2.16,0) arc (0:75:2.16) to[out=165,in=-30] (0,2.5);
		\draw[thick,line width=1.5mm,draw=black!10] (-2.16,0) arc (180:105:2.16) to[out=15,in=210] (0,2.5);
		\draw[thick,line width=1.5mm,draw=black!10] (2.83,0) arc (0:75:2.83) to[out=165,in=30] (0,2.5);
		\draw[thick] (0,-1.66) to (-1.16,-1.66) arc (270:180:1) to (-2.16,0) arc (180:105:2.16) to[out=15,in=210] (0,2.5);
		\draw[thick] (2.83,-4) to (2.83,0) arc (0:75:2.83) to[out=165,in=30] (0,2.5);
		\draw[thick] (-2.16,-4) to (-2.16,-3.33) arc (180:90:1) to (0,-2.33);
		\draw[thick] (2.16,-4) to (2.16,-3.33) arc (0:90:1) to (0,-2.33);
	\end{scope}
\end{scope}
\begin{scope}[yshift=-1.6in]
		\node (=2) at (.45in,0) {$=$};
		\node (q21) at (.65in,.02in) {$-q^2$};
	\begin{scope}[xshift =1.15in,scale=.2]
		\draw[thick,fill=black!10] (-4,0) to (-3.5,0) arc (180:0:3.5) to (4,0) to (4,-4) to (1,-4) arc (0:180:1) to (-4,-4) to (-4,0);
		\draw[thick,fill=black!0] (-1,0) to (-1.5,0) arc (180:0:1.5) to (1,0) arc (360:180:1);
		\draw[thick,dashed,draw=black!10] (4,0) to (4,-4) to (1,-4) arc (0:180:1) to (-4,-4) to (-4,0);
		\draw[thick,dashed,draw=black!10] (1,0) arc (360:180:1);
		\draw[thick] (-2.83,-4) to (-2.83,0) arc (180:105:2.83) arc (105:-75:.33) arc (105:180:2.16) to (-2.16,-4);
		\draw[thick] (2.83,-4) to (2.83,0) arc (0:75:2.83) arc (75:255:.33) arc (75:0:2.16) to (2.16,-4);
	\end{scope}
		\node (+2) at (1.59in,0) {$+$};
		\node (q22) at (1.72in,.02in) {$q^2$};
	\begin{scope}[xshift=2.15in,scale=.2]
		\draw[thick,fill=black!10] (-4,0) to (-3.5,0) arc (180:0:3.5) to (4,0) to (4,-4) to (1,-4) arc (0:180:1) to (-4,-4) to (-4,0);
		\draw[thick,fill=black!0] (-1,0) to (-1.5,0) arc (180:0:1.5) to (1,0) arc (360:180:1);
		\draw[thick,dashed,draw=black!10] (4,0) to (4,-4) to (1,-4) arc (0:180:1) to (-4,-4) to (-4,0);
		\draw[thick,dashed,draw=black!10] (1,0) arc (360:180:1);
		\draw[thick] (-2.83,-4) to (-2.83,0) arc (180:105:2.83) to[out=15,in=150] (0,2.5);
		\draw[thick] (0,-2) to [out=30,in=180] (1,-1.66) to [line to] (1.16,-1.66) arc (270:360:1) to (2.16,0) arc (0:75:2.16) to[out=165,in=-30] (0,2.5);
		\draw[thick] (-2.16,-4) to (-2.16,-3.33) arc (180:90:1) to (-1,-2.33) to [out=0,in=210] (0,-2);
		\draw[thick,line width=1.5mm,draw=black!10] (0,-2) to [out=30,in=180] (-1,-1.66) to [line to] (-1.16,-1.66) arc (270:180:1) to (-2.16,0) arc (180:105:2.16) to[out=15,in=210] (0,2.5);
		\draw[thick,line width=1.5mm,draw=black!10] (2.83,-4) to (2.83,0) arc (0:75:2.83) to[out=165,in=30] (0,2.5);
		\draw[thick,line width=1.5mm,draw=black!10] (2.16,-4) to (2.16,-3.33) arc (0:90:1) to (1,-2.33) to [out=180,in=-30](0,-2);
		\draw[thick] (0,-2) to [out=30,in=180] (-1,-1.66) to [line to] (-1.16,-1.66) arc (270:180:1) to (-2.16,0) arc (180:105:2.16) to[out=15,in=210] (0,2.5);
		\draw[thick] (2.83,-4) to (2.83,0) arc (0:75:2.83) to[out=165,in=30] (0,2.5);
		\draw[thick] (2.16,-4) to (2.16,-3.33) arc (0:90:1) to (1,-2.33) to [out=180,in=-30](0,-2);
	\end{scope}
		\node (+3) at (2.59in,0) {$+$};
		\node (q23) at (2.72in,.02in) {$q^2$};
	\begin{scope}[xshift=3.15in,scale=.2]
		\draw[thick,fill=black!10] (-4,0) to (-3.5,0) arc (180:0:3.5) to (4,0) to (4,-4) to (1,-4) arc (0:180:1) to (-4,-4) to (-4,0);
		\draw[thick,fill=black!0] (-1,0) to (-1.5,0) arc (180:0:1.5) to (1,0) arc (360:180:1);
		\draw[thick,dashed,draw=black!10] (4,0) to (4,-4) to (1,-4) arc (0:180:1) to (-4,-4) to (-4,0);
		\draw[thick,dashed,draw=black!10] (1,0) arc (360:180:1);
		\draw[thick] (-2.83,-4) to (-2.83,0) arc (180:90:2.83) ;
		\draw[thick] (0,-1.66) to (1.16,-1.66) arc (270:360:1) to (2.16,0) arc (0:90:2.16);
		\draw[thick] (0,-1.66) to (-1.16,-1.66) arc (270:180:1) to (-2.16,0) arc (180:90:2.16);
		\draw[thick] (2.83,-4) to (2.83,0) arc (0:90:2.83);
		\draw[thick] (-2.16,-4) to (-2.16,-3.33) arc (180:90:1) to (0,-2.33);
		\draw[thick] (2.16,-4) to (2.16,-3.33) arc (0:90:1) to (0,-2.33);
	\end{scope}
		\node (+4) at (3.59in,0) {$+$};
		\node (q7) at (3.72in,.02in) {$q^4$};
	\begin{scope}[xshift=4.15in,scale=.2]
		\draw[thick,fill=black!10] (-4,0) to (-3.5,0) arc (180:0:3.5) to (4,0) to (4,-4) to (1,-4) arc (0:180:1) to (-4,-4) to (-4,0);
		\draw[thick,fill=black!0] (-1,0) to (-1.5,0) arc (180:0:1.5) to (1,0) arc (360:180:1);
		\draw[thick,dashed,draw=black!10] (4,0) to (4,-4) to (1,-4) arc (0:180:1) to (-4,-4) to (-4,0);
		\draw[thick,dashed,draw=black!10] (1,0) arc (360:180:1);
		\draw[thick] (-2.83,-4) to (-2.83,0) arc (180:105:2.83) arc (105:-75:.33) arc (105:180:2.16);
		\draw[thick] (2.83,-4) to (2.83,0) arc (0:75:2.83) arc (75:255:.33) arc (75:0:2.16);
		\draw[thick] (0,-1.66) to (1.16,-1.66) arc (270:360:1) to (2.16,0);
		\draw[thick] (0,-1.66) to (-1.16,-1.66) arc (270:180:1) to (-2.16,0);
		\draw[thick] (-2.16,-4) to (-2.16,-3.33) arc (180:90:1) to (0,-2.33);
		\draw[thick] (2.16,-4) to (2.16,-3.33) arc (0:90:1) to (0,-2.33);
	\end{scope}
\end{scope}
\end{tikzpicture}
\end{center}

The four links on the right-hand side are products of simple curves with complexity $<\kappa$.  By induction, $[\curve]$ is in the subalgebra generated by the simple curves of complexity zero, and so every simple curve is.  By Corollary \ref{coro: gens}, this set generates all of $\Sk_q(\S)$.
\end{proof}

Finite generation follows immediately.

\begin{thm}\label{thm: fingen}
$\Sk_q(\S)$ and $\Sk_q^o(\S)$ are finitely generated.
\end{thm}
\begin{proof}
$\Sk_q(\S)$ is generated by the simple curves of complexity zero, which is finite.  The localized skein algebra $\Sk_q^o(\S)$ is generated by the simple curves of complexity zero and the inverses to boundary curves, which is again finite.
\end{proof}

\section{Relation with Teichm\"uller space and quantum Teichm\"uller space}

Here, we briefly describe the relation of the skein algebra of a marked surface to certain geometric and algebraic objects in Teichm\"uller theory.

\subsection{Teichm\"uller spaces and moduli space of local systems}

As before, let $\S$ be a marked surface; that is, a compact, oriented surface with a finite set of marked points $\M$ on the boundary.  Let $\S^o$ be the corresponding \emph{opened surface}, where a small ball around each point of $\M$ has been removed.  Let $\partial\M \subset \S^o$ be the boundary of these removed neighborhoods.  Note that the boundary of $\S^o$ alternates between restrictions of components of $\partial\S$ and components of $\partial \M$.
\begin{itemize}
	\item The \emph{Teichm\"uller space} $\mathcal{T}(\S)$ is the moduli space of hyperbolic metrics (up to isotopy) on $\S-\M$  such that each component of$\partial\S-\M$ is a geodesic, and each point in $\M$ is a `cusp'.
	\item The \emph{decorated Teichm\"uller space} $\widehat{\mathcal{T}}(\S)$ is  the moduli space of hyperbolic metrics (up to isotopy) on $\S^o$  such that each component of the restriction of $\partial\S$ is a geodesic, and each component of $\partial \M$ is a horocycle (that is, an arc of constant curvature).
\end{itemize}
There is a natural projection
\[ \widehat{\mathcal{T}}(\S,\M) \rightarrow \mathcal{T}(\S,\M)\]
which extends the metric to the small balls around points in $\M$.  Both of these are real manifolds with a canonical \emph{Weil-Petersson} Poisson structure \cite{Pen12}.

This projection map was realized in \cite{FG06} as the \emph{positive part}\footnote{Here, `positive part' means the subset on which a certain system of distinguished coordinates has positive real values.} of a projection of complex varieties which parametrize certain local systems.
\[ \begin{tikzpicture}
	\node (DTS) at (0,0) {$\widehat{\mathcal{T}}(\S)$};
	\node (TS) at (3,0) {$\mathcal{T}(\S)$};
	\node (A) at (0,-1.5) {$A(\S)$};
	\node (X) at (3,-1.5) {$X(\S)$};
	\draw[->>] (DTS) to (TS);
	\draw[->>] (A) to (X);
	\draw[right hook->] (DTS) to (A);
	\draw[right hook->] (TS) to (X);
\end{tikzpicture}\]
Here, $A(\S)$ is the moduli space of \emph{decorated $SL_2$-local systems on $(\S,\M)$},\footnote{In fact, these are local systems twisted by a `spin structure'; specifically, they are $SL_2$-local systems on the unit tangent bundle to $\S$ with monodromy $-Id_2$ around any fiber.} and $X(\S)$ is the moduli space of \emph{framed $PGL_2$-local systems on $(\S,\M)$}.

Via Fock and Goncharov's theory of \emph{cluster ensembles}, the moduli space $X(\S)$ has a canonical Poisson structure, which extends the \emph{Weil-Petersson form} on $\mathcal{T}(\S)$.  They also describe a \emph{quantization} of $X(\S)$: each triangulation $\Delta$ of $\S$ defines a quantum torus $QTS_{q,\Delta}(\S)$ inside a common skew-field $QTS_q(\S)$.  This skew-field $QTS_q(\S)$ had been previously introduced by Chekhov and Fock, and called the \emph{quantum Teichm\"uller space} of $\S$ \cite{CF99}. The $q=1$ specialization produces a field $QTS_1(\S)$ which is canonically isomorphic to the field of rational functions on the variety $X(\Sigma)$, and every element of $QTS_1(\S)$ restricts to a well-defined function on $\mathcal{T}(X)$.

%Specializing $q$ to $1$ in this skew-field produces the field of rational functions on the variety $X(\Sigma)$.

\begin{rem}
Unlike the quantum tori defining a quantum cluster algebra, the intersection of the quantum tori $QTS_{q,\Delta}(\S)$ may be too small to generate $QTS_q(\S)$ as a skew-field.
\end{rem}

\subsection{Relation to skein algebra}

Fock and Goncharov also define a (commutative) cluster structure on the variety $A(\S)$.  This cluster structure has a cluster variable for each arc in $\S$, and the clusters correspond to triangulations.\footnote{This cluster algebra associated to $\S$ was independently introduced in \cite{GSV03}, who also highlighted its realization as functions on $\mathcal{T}(\S)$.}  This can be extended to a canonical isomorphism 
\[ Sk_1(\S)\stackrel{\sim}{\longrightarrow} \mathcal{O}(A(\S))\]
Under this isomorphism, elements in $Sk_1(\S)$ restrict to functions on $\widehat{\mathcal{T}}(\S)$ which take positive real values.  Specifically, arcs and loops in $Sk_1(\S)$ restrict to the corresponding \emph{Penner coordinates} on $\widehat{\mathcal{T}}(\S)$.\footnote{An excellent reference for the connection between $Sk_1(\S)$ and Penner coordinates is \cite{FST08}.}

This story may be quantized as follows.  For each triangulation $\Delta$ of $\S$, there is a map
\[ \rho_\Delta: QTS_{q,\Delta}(\S) \rightarrow Sk_{q}(\S)[\Delta^{-1}]\simeq \mathbb{T}_\Delta\]
A non-boundary arc $\mathsf{x}\in \Delta$ defines an element $X_\mathsf{x}\in QTS(\S)_{q,\Delta}$.  Then $\rho_\Delta$ is defined by
\[ \rho_{\Delta}(X_\mathsf{x}) = [\Delta^{\mathsf{Q}^\Delta \mathsf{x}} ] \]
That is, $\rho_{\Delta}(X_{\mathsf{x}})$ is the `cross-ratio' of the four arcs in $\Sk_q(\S)$ immediately adjacent to $\mathsf{x}$ (normalized by a power of $q$ so that $\rho_{\Delta}(X_{\mathsf{x}})$ is invariant under the bar involution).  The map $\rho_\Delta$ extends to an inclusion of fraction fields
\[ \rho:QTS_q(\S)\rightarrow \mathcal{F}(Sk_q(\S))\]
which does not depend on a choice of triangulation $\Delta$.  Hence, Chekhov and Fock's quantum Teichm\"uller space can be realized as a sub-skew-field of the fraction field of $Sk_q(\S)$.\footnote{Note that $\rho$ lands in the sub-skew-field of degree $0$ for the endpoint grading of $Sk_1(\S)$.} Under the $q=1$ specialization, the map 
\[\rho: QTS_1(\S) \rightarrow \mathcal{F}(Sk_1(\S)) \]
is the same as the map induced on fraction fields by the cluster ensemble map
\[ A(\S)\rightarrow X(\S) \]

\begin{rem}
In many ways, the algebra $Sk_q(\S)$ is the `decorated' analog of the quantum Teichm\"uller space $QTS_q(\S)$.  In each case, each triangulation determines a quantum torus inside a fixed skew-field.  The main difference is that the intersection of the quantum tori in $QTS_q(\S)$ is too small, and so one must keep track of the whole skew-field $QTS_q(\S)$ to have a reasonable invariant.  By contrast, the intersection of quantum tori containing $Sk_q(\S)$ is $Sk_q(\S)$, which is large enough for every quantum torus to be recoverable as an Ore localization.
\end{rem}

\section{Proof of Lemma \ref{lemma: ugh}}\label{app: lemma}

\subsection{The initial multicurve has positive smoothings}

%\begin{lemma}
%Let $\curve$ be a simple arc in $\S$.  Then the map
%\[ \multi[Y] \mapsto \init([\curve][\multi[Y]]) \]
%is injection from $\Multi$ to itself.
%\end{lemma}

Let $\curve$ be a simple arc in $\S$. For a given $\multi[Y]$, choose a homotopy representative of $\multi[Y]$ so that $\curve \cdot \multi[Y]$ is transverse with minimal crossings.  

Let $\curve\cap \multi[Y]$ denote the set of crossings (that is, non-boundary intersections) in the superposition $\curve\cdot \multi[Y]$.  For any function $\sigma:\curve\cap \multi[Y]\rightarrow \{-,+\}$, let $\multi[R]_\sigma$ be the multicurve obtained by applying the local relation (called a \emph{positive smoothing})
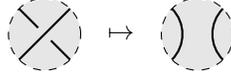
\begin{center}
\begin{tikzpicture}[scale=.8]
%\useasboundingbox (-1in,-.1in) rectangle (1in,.1in);
\begin{scope}[xshift=-.5in,scale=.15]
    \draw[fill=black!10,dashed] (0,0) circle (4);
    \draw[thick] (-2.83,-2.83) to (2.83,2.83);
%    \clip (0,0) circle (1);
%    \draw[thick] (-2.83,2.83) to (2.83,-2.83);
    \draw[thick] (-2.83,2.83) to (-.71,.71);
    \draw[thick] (.71,-.71) to (2.83,-2.83);
\end{scope}
\node (=) at (0,0) {$\mapsto$};
\begin{scope}[xshift=.5in,scale=.15]
    \draw[fill=black!10,dashed] (0,0) circle (4);
    \draw[thick] (-2.83,-2.83) to [out=45,in=-45] (-2.83,2.83);
    \draw[thick] (2.83,-2.83) to [out=135,in=-135] (2.83,2.83);
\end{scope}
\end{tikzpicture}
\end{center}
to each crossing sent to $+$ by $\sigma$, and by applying the local relation (called a \emph{negative smoothing})
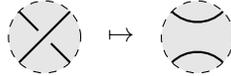
\begin{center}
\begin{tikzpicture}[scale=.8]
%\useasboundingbox (-1in,-.1in) rectangle (1in,.1in);
\begin{scope}[xshift=-.5in,scale=.15]
    \draw[fill=black!10,dashed] (0,0) circle (4);
    \draw[thick] (-2.83,-2.83) to (2.83,2.83);
%    \clip (0,0) circle (1);
%    \draw[thick] (-2.83,2.83) to (2.83,-2.83);
    \draw[thick] (-2.83,2.83) to (-.71,.71);
    \draw[thick] (.71,-.71) to (2.83,-2.83);
\end{scope}
\node (=) at (0,0) {$\mapsto$};
\begin{scope}[xshift=.5in,scale=.15]
    \draw[fill=black!10,dashed] (0,0) circle (4);
    \draw[thick] (-2.83,-2.83) to [out=45,in=135] (2.83,-2.83);
    \draw[thick] (-2.83,2.83) to [out=-45,in=-135] (2.83,2.83);
\end{scope}
\end{tikzpicture}
\end{center}
to each crossing sent to $-$ by $\sigma$.  The purpose if this is that
\[ [\curve][\multi[Y]] = q^a\sum_{\sigma:\curve\cap \multi[Y]\rightarrow \{-,+\}} q^{|\sigma^{-1}(+)|-|\sigma^{-1}(-)|}[\multi[R]_\sigma] \]
%\[ [\curve][\multi[Y]] = q^a\sum_{\sigma:\curve\cap \multi[Y]} q^{|\sigma^{-1}(+)|-|\sigma^{-1}(-)|}(-q^2-q^{-2})^{L_\sigma}\delta_\sigma[\R_\sigma] \]
where $a\in \mathbb{Z}/2$ is the exponent produced by making the endpoints in $\curve\cdot \multi[Y]$ simultaneous.

%The following sublemma is crucial.
%Let $|\sigma^{-1}(-)|$ denote the number of negative smoothings in $\multi[R]_\sigma$.
\begin{lemma}\label{lemma: sublemma}
%In any $\multi[R]_\sigma$, the number of contractible loops is at most the number of negative smoothings.
In any $\multi[R]_\sigma$,
\[ |\{\text{contractible loops}\}| + \frac{1}{2}|\{\text{contractible arcs} \}| \leq |\{\text{negative smoothings} \}|\]
If equality holds, then each strand in each negative smoothing is in a contractible curve.
%The number of contractible loops in $\multi[R]_\sigma$ is at most $|\sigma^{-1}(-)|$.  If there are $|\sigma^{-1}(-)|$-many contractible loops in $\multi[R]_\sigma$, then every strand in a negative smoothing is in a contractible loop.
\end{lemma}
\begin{proof}
Choose a tubular neighborhood $T$ of $\curve$ small enough that $\multi[Y]$ intersects $T$ a minimal number of times, but large enough to contain the chosen neighborhoods of each crossing in $\curve\cdot \multi[Y]$.

Let $\curve[z]$ be a contractable curve in $\multi[R]_\sigma$, and let $D\subset \S$ denote the disc with boundary $\curve[z]$.  Construct a graph $\Gamma$ whose vertices are connected components of $D\smallsetminus \partial T$, and with an edge between two components if they have common boundary in $D\cap \partial T$.  The graph $\Gamma$ is then a retract of the disc $D$, which implies that $\Gamma$ is a tree. Since $\curve[z]$ intersects $T$ but is not contained in $T$, there must be at least one component of $D\cap \partial T$ contained in $T$, and one component of $D\cap \partial T$ disjoint from $T$. It follows that $T$ is a tree with at least two vertices, and so it has at least two vertices of degree $1$.

A degree $1$ vertex of $\Gamma$ corresponds to a component of $D\smallsetminus \partial T$ with a single boundary component in $D\cap \partial T$. Let $D_0$ be such a component of $D\smallsetminus \partial T$; we split into three cases.
\begin{itemize}
	\item $\partial D_0$ contains a marked point.  This implies that $\curve[z]$ is a contractible arc; and so the marked point in $\partial D_0$ is the unique marked point in $\curve[z]$.  It follows that there can be at most one component of $D\smallsetminus \partial T$ of this type.
	\item $D_0\subset D\cap T$ and $\partial D_0$ contains no marked points.  The boundary of $D_0$ must contain a component of $\curve\smallsetminus \multi[Y]$, whose endpoints are crossings in $\curve\cap \multi[Y]$.  In $\multi[R]_\sigma$, one must be a negative smoothing and one must be a positive smoothing in $\multi[R]_\sigma$; otherwise, $D_0$ could cross over $\curve$ and have at least two boundary components in $D\cap \partial T$.    
	\item $D_0\subset D\smallsetminus T$ and $\partial D_0$ contains no marked points.  This implies that the boundary of $D_0$ in $\curve[z]$ can be deformed to the interior $T$, contracting the assumption that $T$ and $\multi[Y]$ intersect a minimum number of times.  There are no components of this type.
\end{itemize}
Hence, if $\curve[z]$ is a contractible arc, then it must pass through a negative smoothing at least once, and if $\curve[z]$ is a contractible loop, then it must pass through a negative smoothing at least twice.  Since each negative smoothing has two strands, this implies the stated lemma.
%The connected components of $D\cap T$ and of $D \smallsetminus T$ are each contractible.  The interiors of these components are the $2$-cells in a cellular decomposition of $D$, whose $1$-cells are strands in $\multi[R]_\sigma$ or segments in $\partial T$, and whose $0$ cells are points in $\multi[R]_\sigma \cap \partial T$ or in $\M$.
%
%Along $\curve[z]=\partial D$, the $0$-cells alternate with $1$-cells that are strands in $\multi[R]_\sigma$.  Hence, there are the same number of $0$-cells as there are $1$-cells coming from strands in $\multi[R]_{\sigma}$, so they cancel in the Euler characteristic.  The Euler characteristic of $D$ is then given by
%\[  |\{\text{components of $D\cap T$}\} |+|\{\text{components of $D\smallsetminus T$}\} |-|\{\text{components of $D\cap \partial T$}\} |\]
\end{proof}

%A contractible loop in $\multi[R]_\sigma$ must pass through negative smoothings at least twice.  Since each negative smoothing only has two strands of $\multi[R]_\sigma$ in its local neighborhood, the number of contractible loops in $\multi[R]_\sigma$ is at most $|\sigma^{-1}(-)|$.  Furthermore, if the number of contractible loops in $\multi[R]_\sigma$ is equal to $|\sigma^{-1}(-1)|$, then every strand in every negative smoothing is in a contractible loop.

Using the skein relations, we may write
\[ [\multi[R]_{\sigma}] = q^{|\sigma^{-1}(+)|-|\sigma^{-1}(-)|}(-q^2-q^{-2})^{L_\sigma}\delta_\sigma[\R_\sigma]\] 
where $\R_\sigma$ be the simple multicurve obtained by deleting contractible curves in $\multi[R]_\sigma$, $L_\sigma$ is the number of contractible loops in $\multi[R]_\sigma$, and $\delta_\sigma$ is $0$ if $\multi[R]_\sigma$ has a contractible arc and $1$ if it does not.  
\begin{align*}
[\curve][\multi[Y]] &= q^a\sum_{\sigma:\curve\cap \multi[Y]} q^{|\sigma^{-1}(+)|-|\sigma^{-1}(-)|}(-q^2-q^{-2})^{L_\sigma}\delta_\sigma[\R_\sigma] \\
&= q^a\sum_{\sigma:\curve\cap \multi[Y]} \sum_{i=0}^{L_\sigma}\binom{L_\sigma}{i}(-1)^iq^{|\sigma^{-1}(+)|-|\sigma^{-1}(-)|+2L_\sigma-4i}\delta_\sigma[\R_\sigma] 
\end{align*}
Using the fact that $|\curve\cap \multi[Y]|=\sigma^{-1}(+)+\sigma^{-1}(-)$, we deduce that
\begin{align*}
[\curve][\multi[Y]] &= q^{a+|\curve\cap\multi[Y]|}\sum_{\sigma:\curve\cap \multi[Y]} \delta_\sigma[\R_\sigma]\sum_{i=0}^{L_\sigma}\binom{L_\sigma}{i}(-1)^iq^{2(L_\sigma-|\sigma^{-1}(-)|-2i)}
\end{align*}
Since $|\sigma^{-1}(-)|$ is the number of negative smoothings, $L_\sigma-|\sigma^{-1}(-)|\leq0$ and equality only holds when $\multi[R]_\sigma$ has an equal number of contractible loops and negative smoothings (which implies that $\delta_\sigma=1$).

\begin{lemma}\label{lemma: sublemma2}
Let $\sigma_+$ denote the function which sends every crossing to $+$.  If $\sigma$ is another function such that $\multi[R]_\sigma$ has an equal number of contractible loops and negative smoothings, then $\R_\sigma\prec \R_{\sigma_+}$.  Consequently,
\[  \init([\curve][\multi[Y]]) = [\R_{\sigma_+}] \]
\end{lemma}
\begin{proof}
By Lemma \ref{lemma: sublemma}, neither $\multi[R]_{\sigma_+}$ nor $\multi[R]_{\sigma}$ have contractible arcs. Consequently, both $[\multi[R]_{\sigma_+}]$ and $[\multi[R]_{\sigma}]$ are non-zero.

Choose any negative smoothing in $\multi[R]_\sigma$, and let $\multi[R]_{\sigma'}$ denote the multicurve in which it has been replaced by a positive smoothing.  This alteration involves at most two curves in $\multi[R]_{\sigma}$ which must become at least one curve in $\multi[R]_{\sigma'}$, and so the total number of curves can decrease by at most $1$.  However, since the number of negative smoothings has decreased by $1$, the number of contractible loops must have decreased by exactly $1$, and so $\multi[R]_{\sigma'}$ has an equal number of contractible loops and negative smoothings.

The two strands in the chosen negative smoothing in $\multi[R]_\sigma$ must both be in contractible loops, by the preceding lemma.  Since the number of contractible loops decreases by $1$, there are two possibilities.
\begin{itemize}
	\item The two strands were in distinct contractible loops in $\multi[R]_\sigma$. They become one contractible loop in $\multi[R]_{\sigma'}$.
	\item The two strands are in the same contractible loop in $\multi[R]_\sigma$. They become two loops in $\multi[R]_{\sigma'}$, which must be non-contractible.
\end{itemize}
In the first case, $\R_\sigma \preceq\R_{\sigma'}$, and in the second, $\R_\sigma \prec \R_{\sigma'}$.  

By switching negative smoothings to positive smoothings one at a time, we may construct a sequence 
\[ \multi[R]_\sigma = \multi[R]_{\sigma_0},\multi[R]_{\sigma_1}, ... , \multi[R]_{\sigma_n}=\multi[R]_{\sigma^+}\]
such that at each step, $\R_{\sigma_i} \preceq \R_{\sigma_{i+1}}$.  Furthermore, since $\multi[R]_{\sigma^+}$ has no contractible loops, the last step in this sequence must be of the second type above; that is, $\R_{\sigma_{n-1}}\prec \R_{\sigma_+}$.  By transitivity, $\R_{\sigma}\prec \R_{\sigma^+}$.
\end{proof}

\subsection{The map is an injection}

Choose a tubular neighborhood $T$ of $\curve$.  For a given $\multi[Y]$, choose a homotopy representative of $\multi[Y]$ so that $\multi[Y]$ intersects both $\curve$ and $T$ a minimal number of times.

Let $\gamma_{\curve}$ denote the map $\multi[Y]\mapsto \init([\curve][\multi[Y]])=[\R_{\sigma^+}]$.  The multicurve $\gamma_{\curve}(\multi[Y])$ has the following concrete construction inside the tubular neighborhood $T$: cut $\multi[Y]$ along each crossing in $\curve\cdot \multi[Y]$ and reconnect the strands by shifting to the right along $\curve$.  Any spare ends on either side are attached to the endpoints of $\curve$.  The two cases (distinct versus identical endpoints of $\curve$) are illustrated in the Figure \ref{fig: gamma}. %(for $\curve$ an arc with distinct endpoints, though the other cases are similar).
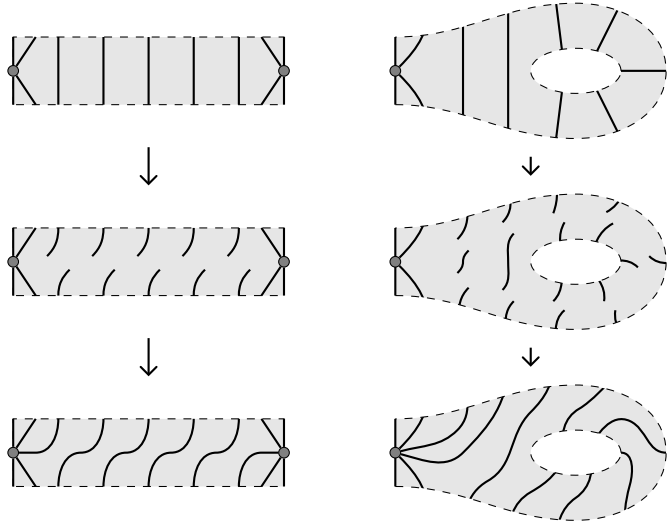
\begin{figure}[h!t]
\begin{tikzpicture}
\begin{scope}[scale=.3]
	\draw[fill=black!10,dashed] (-6,-1.5) rectangle (6,1.5);
	\node[marked] (1) at (-6,0) {};
	\node[marked] (2) at (6,0) {};
	\draw[thick] (-6,-1.5) to (1) to (-6,1.5);
	\draw[thick] (6,-1.5) to (2) to (6,1.5);
	
	\draw[thick] (1) to (-5,1.5);
	\draw[thick] (1) to (-5,-1.5);
	\draw[thick] (2) to (5,1.5);
	\draw[thick] (2) to (5,-1.5);
		
	\draw[thick] (-4,-1.5) to (-4,1.5);
	\draw[thick] (-2,-1.5) to (-2,1.5);
	\draw[thick] (0,-1.5) to (0,1.5);
	\draw[thick] (2,-1.5) to (2,1.5);
	\draw[thick] (4,-1.5) to (4,1.5);
\end{scope}
\draw[thick, -angle 90] (0,-.4in) to (0,-.6in);
\begin{scope}[yshift=-1in,xshift=0in,scale=.3]
	\draw[fill=black!10,dashed] (-6,-1.5) rectangle (6,1.5);
	\node[marked] (1) at (-6,0) {};
	\node[marked] (2) at (6,0) {};
	\draw[thick] (-6,-1.5) to (1) to (-6,1.5);
	\draw[thick] (6,-1.5) to (2) to (6,1.5);
		
	\draw[thick] (1) to (-5,1.5);
	\draw[thick] (1) to (-5,-1.5);
	\draw[thick] (2) to (5,1.5);
	\draw[thick] (2) to (5,-1.5);
	
	\draw[thick] (-4.5,0.25) to [out=45,in=270]  (-4,1.5);
	\draw[thick] (-2.5,0.25) to [out=45,in=270]  (-2,1.5);
	\draw[thick] (-0.5,0.25) to [out=45,in=270]  (0,1.5);
	\draw[thick] (1.5,0.25) to [out=45,in=270]  (2,1.5);
	\draw[thick] (3.5,0.25) to [out=45,in=270]  (4,1.5);
	\draw[thick] (-4,-1.5) to [out=90,in=225] (-3.5,-.35);
	\draw[thick] (-2,-1.5) to [out=90,in=225] (-1.5,-.35);
	\draw[thick] (0,-1.5) to [out=90,in=225] (0.5,-.35);
	\draw[thick] (2,-1.5) to [out=90,in=225] (2.5,-.35);
	\draw[thick] (4,-1.5) to [out=90,in=225] (4.5,-.35);
\end{scope}
\draw[thick, -angle 90] (0,-1.4in) to (0,-1.6in);
\begin{scope}[yshift=-2in,xshift=0in,scale=.3]
	\draw[fill=black!10,dashed] (-6,-1.5) rectangle (6,1.5);
	\node[marked] (1) at (-6,0) {};
	\node[marked] (2) at (6,0) {};
	\draw[thick] (-6,-1.5) to (1) to (-6,1.5);
	\draw[thick] (6,-1.5) to (2) to (6,1.5);
		
	\draw[thick] (1) to (-5,1.5);
	\draw[thick] (1) to (-5,-1.5);
	\draw[thick] (2) to (5,1.5);
	\draw[thick] (2) to (5,-1.5);
	
	\draw[thick] (1) to [out=0,in=180] (-5,0) to [out=0,in=270]  (-4,1.5);
	\draw[thick] (-4,-1.5) to [out=90,in=180] (-3,0) to [out=0,in=270] (-2,1.5);
	\draw[thick] (-2,-1.5) to [out=90,in=180] (-1,0) to [out=0,in=270] (0,1.5);
	\draw[thick] (0,-1.5) to [out=90,in=180] (1,0) to [out=0,in=270] (2,1.5);
	\draw[thick] (2,-1.5) to [out=90,in=180] (3,0) to [out=0,in=270] (4,1.5);
	\draw[thick] (4,-1.5) to [out=90,in=180] (5,0) to [out=0,in=180] (2);	
\end{scope}

\begin{scope}[xshift=2in,scale=.3]
	\draw[fill=black!10,dashed] (-6,-1.5) to [out=0,in=180] (2,-3) to [out=0,in=270] (6,0) to [out=90,in=0] (2,3) to [out=180,in=0] (-6,1.5);
	\draw[thick] (-6,1.5) to [out=270,in=90] (-6,-1.5);
	\node[marked] (1) at (-6,0) {};
	\clip (-6,-1.5) to [out=0,in=180] (2,-3) to [out=0,in=270] (6,0) to [out=90,in=0] (2,3) to [out=180,in=0] (-6,1.5);
	
	\draw[thick] (1) to [out=45,in=270] (-4.5,3);
	\draw[thick] (-3,-4) to (-3,4);
	\draw[thick] (-1,-4) to (-1,4);
	\draw[thick] (1,-4) to (1.5,0) to (1,4);
	\draw[thick] (4.5,-4) to (2.5,0) to (4.5,4);
	\draw[thick] (3,0) to (7,0);
	\draw[thick] (1) to [out=-45,in=-270] (-4.5,-3);
	
	\draw[fill=white,dashed] (2,0) ellipse (2 and 1);
	
%	\draw[red] (1) to [out=-15,in=180] (3,-2) to [out=0,in=270] (5,0) to [out=90,in=0] (3,2) to [in=15, out=180] (1);
\end{scope}
\draw[thick, -angle 90] (2in,-.45in) to (2in,-.55in);
\begin{scope}[yshift=-1in,xshift=2in,scale=.3]
	\draw[fill=black!10,dashed] (-6,-1.5) to [out=0,in=180] (2,-3) to [out=0,in=270] (6,0) to [out=90,in=0] (2,3) to [out=180,in=0] (-6,1.5);
	\draw[thick] (-6,1.5) to [out=270,in=90] (-6,-1.5);
	\node[marked] (1) at (-6,0) {};
	\clip (-6,-1.5) to [out=0,in=180] (2,-3) to [out=0,in=270] (6,0) to [out=90,in=0] (2,3) to [out=180,in=0] (-6,1.5);
	
	\draw[thick] (1) to [out=45,in=270] (-4.5,3);
	\draw[thick] (1) to [out=-45,in=-270] (-4.5,-3);
	\draw[thick] (-3,-4) to [out=90,in=225](-2.75,-1.15);
	\draw[thick] (-3.25,-.5) to [out=45,in=270] (-3,0) to [out=90,in=225] (-2.75,.5);
	\draw[thick] (-3.25,1) to [out=45,in=270] (-3,4);
	\draw[thick] (-1,-4) to [out=90,in=225](-.75,-1.75);
	\draw[thick] (-1.25,-1.25) to [out=45,in=270] (-1,0) to [out=90,in=225] (-.75,1.25);
	\draw[thick] (-1.25,1.75) to [out=45,in=270] (-1,4);
	\draw[thick] (1,-4) to [out=75,in=220] (1.5,-2.15);
	\draw[thick] (1,-1.75) to [out=45,in=255] (1.5,0);
	\draw[thick] (1.5,0) to [out=105,in=240] (1.5,1.75);
	\draw[thick] (1,2.15) to [out=70,in=285] (1,4);
	\draw[thick] (4.5,-4) to [out=117,in=-98] (3.75,-2.15);
	\draw[thick] (3.25,-1.75) to [out=82,in=-63] (2.5,0);
	\draw[thick] (4.5,4) to [out=-117,in=28] (3.35,2.2);
	\draw[thick] (3.65,1.7) to [out=-152,in=63] (2.5,0);
	\draw[thick] (3,0) to [out=0,in=135] (4.75,-0.25);
	\draw[thick] (5.25,0.25) to [out=-45,in=180] (7,0);
	
	\draw[fill=white,dashed] (2,0) ellipse (2 and 1);

%	\draw[red] (1) to [out=-15,in=180] (3,-2) to [out=0,in=270] (5,0) to [out=90,in=0] (3,2) to [in=15, out=180] (1);
\end{scope}
\draw[thick, -angle 90] (2in,-1.45in) to (2in,-1.55in);
\begin{scope}[yshift=-2in,xshift=2in,scale=.3]
	\draw[fill=black!10,dashed] (-6,-1.5) to [out=0,in=180] (2,-3) to [out=0,in=270] (6,0) to [out=90,in=0] (2,3) to [out=180,in=0] (-6,1.5);
	\draw[thick] (-6,1.5) to [out=270,in=90] (-6,-1.5);
	\node[marked] (1) at (-6,0) {};
	\clip (-6,-1.5) to [out=0,in=180] (2,-3) to [out=0,in=270] (6,0) to [out=90,in=0] (2,3) to [out=180,in=0] (-6,1.5);
	
	\draw[thick] (1) to [out=45,in=270] (-4.5,3);
	\draw[thick] (1) to [out=-45,in=-270] (-4.5,-3);
%	\draw[thick] (1) to [out=-15,in=225] (-3.25,-.5) to [out=45,in=270] (-3,0) to [out=90,in=225] (-2.75,.5)  to [out=45,in=225] (-1.25,1.75) to [out=45,in=270] (-1,4);
	\draw[thick] (1) to [out=15,in=225] (-3.25,1) to [out=45,in=270] (-3,4);
	\draw[thick] (1) to [out=-15,in=225] (-3,0) to [out=45,in=270] (-1,4);
	\draw[thick] (-3,-4) to [out=90,in=210] (-1.9,-1.25) to [out=30,in=225] (0,1.8) to [out=45,in=270] (1,4);
%	\draw[thick] (-3,-4) to [out=90,in=225] (-2.75,-1.15) to [line to] (-1.25,-1.25) to [out=45,in=270] (-1,0) to [out=90,in=225] (-.75,1.25);
	\draw[thick] (-1,-4) to [out=90,in=210] (0,-1.75) to [out=30,in=255] (1.5,0);
	\draw[thick] (1,-4) to [out=75,in=215] (2,-2) to [out=35,in=-63] (2.5,0);
	\draw[thick] (4.5,-4) to [out=117,in=-105] (4.4,-1.4) to [out=75,in=0] (4,0);
	\draw[thick] (6,0) to [out=180,in=-25] (4.4,1.4) to [out=155,in=63] (2.5,0);
	\draw[thick] (1.5,0) to [out=105,in=210] (2.15,1.95) to [out=30,in=63] (4.5,4);
	
	\draw[fill=white,dashed] (2,0) ellipse (2 and 1);

%	\draw[red] (1) to [out=-15,in=180] (3,-2) to [out=0,in=270] (5,0) to [out=90,in=0] (3,2) to [in=15, out=180] (1);
\end{scope}
\end{tikzpicture}
\caption{Explicit construction of $\gamma_{\curve}(\multi[Y])$.}
\label{fig: gamma}
\end{figure}

%\begin{center}
%\begin{tikzpicture}
%\begin{scope}[scale=.3]
%	\draw[fill=black!10,dashed] (-6,-1.5) to [out=0,in=180] (2,-3) to [out=0,in=270] (6,0) to [out=90,in=0] (2,3) to [out=180,in=0] (-6,1.5);% to [out=270,in=90] (-6,-1.5);
%	\draw[fill=white,dashed] (2,0) ellipse (2 and 1);
%	\draw[thick] (-6,1.5) to [out=270,in=90] (-6,-1.5);
%	\node[marked] (1) at (-6,0) {};
%	\clip (-6,-1.5) to [out=0,in=180] (2,-3) to [out=0,in=270] (6,0) to [out=90,in=0] (2,3) to [out=180,in=0] (-6,1.5);% to [out=270,in=90] (-6,-1.5);
%	
%	\draw[thick] (1) to [out=45,in=270] (-4.5,3);
%	\draw[thick] (-3,-4) to (-3,4);
%	\draw[thick] (-1.5,-4) to (-1.5,4);
%	\draw[thick] (2,1) to (2,3);
%	\draw[thick] (2,-1) to (2,-3);
%	\draw[thick] (1) to [out=-45,in=-270] (-4.5,-3);
%\end{scope}
%\node (=) at (1in,0) {$\mapsto$};
%\begin{scope}[xshift=2in,scale=.3]
%	\draw[fill=black!10,dashed] (-6,-1.5) to [out=0,in=180] (2,-3) to [out=0,in=270] (6,0) to [out=90,in=0] (2,3) to [out=180,in=0] (-6,1.5);% to [out=270,in=90] (-6,-1.5);
%\end{scope}
%\end{tikzpicture}
%\end{center}

\begin{lemma}\label{lemma: gammainj}
For any simple curve, the map $\gamma_{\curve}:\Multi\rightarrow \Multi$ is injective.
\end{lemma}
\begin{proof}
Let $\multi[Y]$ be a simple multicurve transverse to $\curve$ with minimal crossings.  Define a new multicurve $\nu_{\curve}(\multi[Y])$ as follows.
\begin{itemize}
\item If $\curve$ has one end at a marked point $p$, and there are no strands of $\multi[Y]$ counter-clockwise to $\curve$ at $p$, then $\nu_{\curve}(\multi[Y])$ is the empty multicurve $\emptyset$.
\item If $\curve$ has both ends at a marked point $p$, and there are fewer than two strands of $\multi[Y]$ counter-clockwise to $\curve$, then $\nu_{\curve}(\multi[Y])$ is the empty multicurve $\emptyset$.
%
%If, at any end of $\curve$, there are no strands of $\multi[Y]$ counter-clockwise to $\curve$, then $\nu_{\curve}(\multi[Y])$ is the empty multicurve $\emptyset$.
\item Otherwise, construct $\nu_{\curve}(\multi[Y])$ as follows. Cut $\multi[Y]$ along $\curve$, and at each end of $\curve$, disconnect the first strand of $\multi[Y]$ counter-clockwise to $\curve$.
Reconnect these ends by shifting to the left along $\curve$.
%This is just the inverse construction to the construction of $\gamma_{\curve}$.

\begin{figure}[h!t]
\begin{tikzpicture}
\begin{scope}[yshift=-2in,scale=.3]
	\draw[fill=black!10,dashed] (-6,-1.5) rectangle (6,1.5);
	\node[marked] (1) at (-6,0) {};
	\node[marked] (2) at (6,0) {};
	\draw[thick] (-6,-1.5) to (1) to (-6,1.5);
	\draw[thick] (6,-1.5) to (2) to (6,1.5);
	
	\draw[thick] (1) to (-5,1.5);
	\draw[thick] (1) to (-5,-1.5);
	\draw[thick] (2) to (5,1.5);
	\draw[thick] (2) to (5,-1.5);
		
	\draw[thick] (-4,-1.5) to (-4,1.5);
	\draw[thick] (-2,-1.5) to (-2,1.5);
	\draw[thick] (0,-1.5) to (0,1.5);
	\draw[thick] (2,-1.5) to (2,1.5);
	\draw[thick] (4,-1.5) to (4,1.5);
\end{scope}
\draw[thick, -angle 90] (0,-.4in) to (0,-.6in);
\begin{scope}[yshift=-1in,xshift=0in,scale=.3]
	\draw[fill=black!10,dashed] (-6,-1.5) rectangle (6,1.5);
	\node[marked] (1) at (-6,0) {};
	\node[marked] (2) at (6,0) {};
	\draw[thick] (-6,-1.5) to (1) to (-6,1.5);
	\draw[thick] (6,-1.5) to (2) to (6,1.5);
		
	\draw[thick] (1) to (-5,1.5);
	\draw[thick] (1) to (-5,-1.5);
	\draw[thick] (2) to (5,1.5);
	\draw[thick] (2) to (5,-1.5);
	
	\draw[thick] (-4.5,0.25) to [out=45,in=270]  (-4,1.5);
	\draw[thick] (-2.5,0.25) to [out=45,in=270]  (-2,1.5);
	\draw[thick] (-0.5,0.25) to [out=45,in=270]  (0,1.5);
	\draw[thick] (1.5,0.25) to [out=45,in=270]  (2,1.5);
	\draw[thick] (3.5,0.25) to [out=45,in=270]  (4,1.5);
	\draw[thick] (-4,-1.5) to [out=90,in=225] (-3.5,-.35);
	\draw[thick] (-2,-1.5) to [out=90,in=225] (-1.5,-.35);
	\draw[thick] (0,-1.5) to [out=90,in=225] (0.5,-.35);
	\draw[thick] (2,-1.5) to [out=90,in=225] (2.5,-.35);
	\draw[thick] (4,-1.5) to [out=90,in=225] (4.5,-.35);
\end{scope}
\draw[thick, -angle 90] (0,-1.4in) to (0,-1.6in);
\begin{scope}[yshift=-0in,xshift=0in,scale=.3]
	\draw[fill=black!10,dashed] (-6,-1.5) rectangle (6,1.5);
	\node[marked] (1) at (-6,0) {};
	\node[marked] (2) at (6,0) {};
	\draw[thick] (-6,-1.5) to (1) to (-6,1.5);
	\draw[thick] (6,-1.5) to (2) to (6,1.5);
		
	\draw[thick] (1) to (-5,1.5);
	\draw[thick] (1) to (-5,-1.5);
	\draw[thick] (2) to (5,1.5);
	\draw[thick] (2) to (5,-1.5);
	
	\draw[thick] (1) to [out=0,in=180] (-5,0) to [out=0,in=270]  (-4,1.5);
	\draw[thick] (-4,-1.5) to [out=90,in=180] (-3,0) to [out=0,in=270] (-2,1.5);
	\draw[thick] (-2,-1.5) to [out=90,in=180] (-1,0) to [out=0,in=270] (0,1.5);
	\draw[thick] (0,-1.5) to [out=90,in=180] (1,0) to [out=0,in=270] (2,1.5);
	\draw[thick] (2,-1.5) to [out=90,in=180] (3,0) to [out=0,in=270] (4,1.5);
	\draw[thick] (4,-1.5) to [out=90,in=180] (5,0) to [out=0,in=180] (2);	
\end{scope}

\begin{scope}[yshift=-2in,xshift=2in,scale=.3]
	\draw[fill=black!10,dashed] (-6,-1.5) to [out=0,in=180] (2,-3) to [out=0,in=270] (6,0) to [out=90,in=0] (2,3) to [out=180,in=0] (-6,1.5);
	\draw[thick] (-6,1.5) to [out=270,in=90] (-6,-1.5);
	\node[marked] (1) at (-6,0) {};
	\clip (-6,-1.5) to [out=0,in=180] (2,-3) to [out=0,in=270] (6,0) to [out=90,in=0] (2,3) to [out=180,in=0] (-6,1.5);
	
	\draw[thick] (1) to [out=45,in=270] (-4.5,3);
	\draw[thick] (-3,-4) to (-3,4);
	\draw[thick] (-1,-4) to (-1,4);
	\draw[thick] (1,-4) to (1.5,0) to (1,4);
	\draw[thick] (4.5,-4) to (2.5,0) to (4.5,4);
	\draw[thick] (3,0) to (7,0);
	\draw[thick] (1) to [out=-45,in=-270] (-4.5,-3);
	
	\draw[fill=white,dashed] (2,0) ellipse (2 and 1);
	
%	\draw[red] (1) to [out=-15,in=180] (3,-2) to [out=0,in=270] (5,0) to [out=90,in=0] (3,2) to [in=15, out=180] (1);
\end{scope}
\draw[thick, -angle 90] (2in,-.45in) to (2in,-.55in);
\begin{scope}[yshift=-1in,xshift=2in,scale=.3]
	\draw[fill=black!10,dashed] (-6,-1.5) to [out=0,in=180] (2,-3) to [out=0,in=270] (6,0) to [out=90,in=0] (2,3) to [out=180,in=0] (-6,1.5);
	\draw[thick] (-6,1.5) to [out=270,in=90] (-6,-1.5);
	\node[marked] (1) at (-6,0) {};
	\clip (-6,-1.5) to [out=0,in=180] (2,-3) to [out=0,in=270] (6,0) to [out=90,in=0] (2,3) to [out=180,in=0] (-6,1.5);
	
	\draw[thick] (1) to [out=45,in=270] (-4.5,3);
	\draw[thick] (1) to [out=-45,in=-270] (-4.5,-3);
	\draw[thick] (-3,-4) to [out=90,in=225](-2.75,-1.15);
	\draw[thick] (-3.25,-.5) to [out=45,in=270] (-3,0) to [out=90,in=225] (-2.75,.5);
	\draw[thick] (-3.25,1) to [out=45,in=270] (-3,4);
	\draw[thick] (-1,-4) to [out=90,in=225](-.75,-1.75);
	\draw[thick] (-1.25,-1.25) to [out=45,in=270] (-1,0) to [out=90,in=225] (-.75,1.25);
	\draw[thick] (-1.25,1.75) to [out=45,in=270] (-1,4);
	\draw[thick] (1,-4) to [out=75,in=220] (1.5,-2.15);
	\draw[thick] (1,-1.75) to [out=45,in=255] (1.5,0);
	\draw[thick] (1.5,0) to [out=105,in=240] (1.5,1.75);
	\draw[thick] (1,2.15) to [out=70,in=285] (1,4);
	\draw[thick] (4.5,-4) to [out=117,in=-98] (3.75,-2.15);
	\draw[thick] (3.25,-1.75) to [out=82,in=-63] (2.5,0);
	\draw[thick] (4.5,4) to [out=-117,in=28] (3.35,2.2);
	\draw[thick] (3.65,1.7) to [out=-152,in=63] (2.5,0);
	\draw[thick] (3,0) to [out=0,in=135] (4.75,-0.25);
	\draw[thick] (5.25,0.25) to [out=-45,in=180] (7,0);
	
	\draw[fill=white,dashed] (2,0) ellipse (2 and 1);

%	\draw[red] (1) to [out=-15,in=180] (3,-2) to [out=0,in=270] (5,0) to [out=90,in=0] (3,2) to [in=15, out=180] (1);
\end{scope}
\draw[thick, -angle 90] (2in,-1.45in) to (2in,-1.55in);
\begin{scope}[yshift=0in,xshift=2in,scale=.3]
	\draw[fill=black!10,dashed] (-6,-1.5) to [out=0,in=180] (2,-3) to [out=0,in=270] (6,0) to [out=90,in=0] (2,3) to [out=180,in=0] (-6,1.5);
	\draw[thick] (-6,1.5) to [out=270,in=90] (-6,-1.5);
	\node[marked] (1) at (-6,0) {};
	\clip (-6,-1.5) to [out=0,in=180] (2,-3) to [out=0,in=270] (6,0) to [out=90,in=0] (2,3) to [out=180,in=0] (-6,1.5);
	
	\draw[thick] (1) to [out=45,in=270] (-4.5,3);
	\draw[thick] (1) to [out=-45,in=-270] (-4.5,-3);
%	\draw[thick] (1) to [out=-15,in=225] (-3.25,-.5) to [out=45,in=270] (-3,0) to [out=90,in=225] (-2.75,.5)  to [out=45,in=225] (-1.25,1.75) to [out=45,in=270] (-1,4);
	\draw[thick] (1) to [out=15,in=225] (-3.25,1) to [out=45,in=270] (-3,4);
	\draw[thick] (1) to [out=-15,in=225] (-3,0) to [out=45,in=270] (-1,4);
	\draw[thick] (-3,-4) to [out=90,in=210] (-1.9,-1.25) to [out=30,in=225] (0,1.8) to [out=45,in=270] (1,4);
%	\draw[thick] (-3,-4) to [out=90,in=225] (-2.75,-1.15) to [line to] (-1.25,-1.25) to [out=45,in=270] (-1,0) to [out=90,in=225] (-.75,1.25);
	\draw[thick] (-1,-4) to [out=90,in=210] (0,-1.75) to [out=30,in=255] (1.5,0);
	\draw[thick] (1,-4) to [out=75,in=215] (2,-2) to [out=35,in=-63] (2.5,0);
	\draw[thick] (4.5,-4) to [out=117,in=-105] (4.4,-1.4) to [out=75,in=0] (4,0);
	\draw[thick] (6,0) to [out=180,in=-25] (4.4,1.4) to [out=155,in=63] (2.5,0);
	\draw[thick] (1.5,0) to [out=105,in=210] (2.15,1.95) to [out=30,in=63] (4.5,4);
	
	\draw[fill=white,dashed] (2,0) ellipse (2 and 1);

%	\draw[red] (1) to [out=-15,in=180] (3,-2) to [out=0,in=270] (5,0) to [out=90,in=0] (3,2) to [in=15, out=180] (1);
\end{scope}
\end{tikzpicture}
\caption{Explicit construction of $\nu_{\curve}(\multi[Y])$.}
\end{figure}
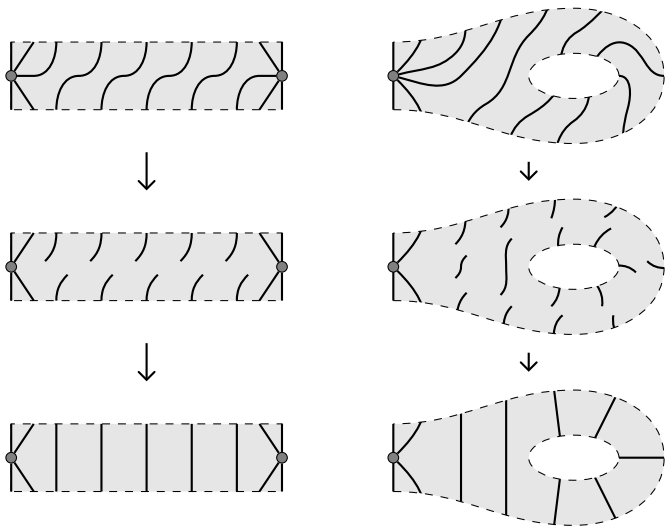

%\begin{center}
%\begin{tikzpicture}
%\begin{scope}[scale=.3]
%	\draw[fill=black!10,dashed] (-6,-1) rectangle (6,1);
%	\node[marked] (1) at (-6,0) {};
%	\node[marked] (2) at (6,0) {};
%	\draw[thick] (-6,-1) to (1) to (-6,1);
%	\draw[thick] (6,-1) to (2) to (6,1);
%	
%	\draw[thick] (1) to (-5,1);
%	\draw[thick] (1) to (-5,-1);
%	\draw[thick] (2) to (5,1);
%	\draw[thick] (2) to (5,-1);
%		
%	\draw[thick] (-4,-1) to (-4,1);
%	\draw[thick] (-2,-1) to (-2,1);
%	\draw[thick] (0,-1) to (0,1);
%	\draw[thick] (2,-1) to (2,1);
%	\draw[thick] (4,-1) to (4,1);
%	
%%	\draw[outline] (1) to (2);
%%	\draw[thick] (1) to (2);
%\end{scope}
%\node (=) at (-1in,0) {$\mapsto$};
%\begin{scope}[xshift=-2in,scale=.3]
%	\draw[fill=black!10,dashed] (-6,-1) rectangle (6,1);
%	\node[marked] (1) at (-6,0) {};
%	\node[marked] (2) at (6,0) {};
%	\draw[thick] (-6,-1) to (1) to (-6,1);
%	\draw[thick] (6,-1) to (2) to (6,1);
%		
%	\draw[thick] (1) to (-5,1);
%	\draw[thick] (1) to (-5,-1);
%	\draw[thick] (2) to (5,1);
%	\draw[thick] (2) to (5,-1);
%	
%	\draw[thick] (1) to [out=0,in=180] (-5,0) to [out=0,in=270]  (-4,1);
%	\draw[thick] (-4,-1) to [out=90,in=180] (-3,0) to [out=0,in=270] (-2,1);
%	\draw[thick] (-2,-1) to [out=90,in=180] (-1,0) to [out=0,in=270] (0,1);
%	\draw[thick] (0,-1) to [out=90,in=180] (1,0) to [out=0,in=270] (2,1);
%	\draw[thick] (2,-1) to [out=90,in=180] (3,0) to [out=0,in=270] (4,1);
%	\draw[thick] (4,-1) to [out=90,in=180] (5,0) to [out=0,in=180] (2);	
%%	\draw[outline] (1) to (2);
%%	\draw[thick] (1) to (2);
%\end{scope}
%\end{tikzpicture}
%\end{center}

\end{itemize}
The composition $\nu_{\curve}(\gamma_{\curve}(\multi[Y]))=\multi[Y]$, therefore $\gamma_{\curve}$ is injective.
\end{proof}

Lemma \ref{lemma: ugh} is an immediate consequence of Lemmas \ref{lemma: sublemma2} and \ref{lemma: gammainj}.

\bibliography{MyNewBib}{}
\bibliographystyle{amsalpha}

%\newpage
%
%\tableofcontents

\end{document}